\documentclass[10pt]{article}

\usepackage{url}
\usepackage{mathtools}
\usepackage{amssymb}
\usepackage{amsthm}
\usepackage{empheq}
\usepackage{latexsym}
\usepackage{enumitem}
\usepackage{eurosym}
\usepackage{dsfont}
\usepackage{appendix}
\usepackage{color} 
\usepackage[unicode]{hyperref}
\usepackage{frcursive}
\usepackage[utf8]{inputenc}
\usepackage[T1]{fontenc}
\usepackage{geometry}
\usepackage{multirow}
\usepackage[colorinlistoftodos]{todonotes}
\usepackage{lmodern}
\usepackage{anyfontsize}
\usepackage{stmaryrd}
\usepackage{natbib}
\usepackage{cleveref}

\usepackage{amsbsy}

\usepackage{comment}

\setcitestyle{numbers,open={[},close={]}}

\definecolor{red}{rgb}{0.7,0.15,0.15}
\definecolor{green}{rgb}{0,0.5,0}
\definecolor{blue}{rgb}{0,0,0.7}
\hypersetup{colorlinks, linkcolor={red},citecolor={green}, urlcolor={blue}}
			
\makeatletter \@addtoreset{equation}{section}

\newtheorem{theorem}{Theorem}[section]
\newtheorem{assumption}[theorem]{Assumption}
\newtheorem{corollary}[theorem]{Corollary}

\newtheorem{proposition}[theorem]{Proposition}

\newtheorem{definition}[theorem]{Definition}
\newtheorem{remark}[theorem]{Remark}

\def\no{\noindent}
\def\beq{\begin{eqnarray}}
\def\eeq{\end{eqnarray}}
\def\be*{\begin{eqnarray*}}
\def\ee*{\end{eqnarray*}}

\def \E{\mathbb{E}}
\def \F{\mathbb{F}}

\def \H{\mathbb{H}}

\def \L{\mathbb{L}}
\def \M{\mathbb{M}}
\def \N{\mathbb{N}}

\def \P{\mathbb{P}}

\def \R{\mathbb{R}}
\def \S{\mathbb{S}}

\def \Pr{\mathrm{P}}

\def\Ac{{\cal A}}

\def\Cc{{\cal C}}

\def\Gc{{\cal G}}
\def\Hc{{\cal H}}

\def\Jc{{\cal J}}

\def\Lc{{\cal L}}
\def\Mc{{\cal M}}
\def\Nc{{\cal N}}

\def\Pc{{\cal P}}

\def\Rc{{\cal R}}

\def\Uc{{\cal U}}
\def\Vc{{\cal V}}
\def\Wc{{\cal W}}

\def\Zc{{\cal Z}}

\def\Jb{{\bar J}}

\def\Mb{{\overline M}}
\def\Nb{{\overline{N}}}

\def\Pb{{\overline \P}}

\def\Ub{{\overline U}}
\def\Vb{{\overline V}}
\def\Wb{{\overline W}}
\def\Xb{{\overline X}}

\def\gammabb{{\overline{\gamma}}}

\def \Acb{ {\overline{\Ac}} }
\def \Ucb{ {\overline{\Uc}} }

\def\x{\times}

\def\Om{\Omega}

\def\om{\omega}

\def\Lambdab{\overline{\Lambda}}

\def\alphabb{\overline{\alpha}}

\def\Vt{\widetilde{V}}
\def\Wt{\widetilde{W}}
\def\Xt{\widetilde{X}}

\def\Mcb{\overline{\Mc}}

\def\Pib{\overline{\Pi}}

\def\Rrb{\overline{\Rr}}

\def\0{\mathbf{0}}

\def \mub{\overline{\mu}}

\def \muh{\widehat{\mu}}

\def\normeL2#1{\left\|{#1}\right\|_{L^2}}

\def\betah{\widehat \beta}
\def\betabb{\overline \beta}

\def\Eh{\widehat E}

\def\Mh{\widehat M}

\def\Nh{\widehat N}

\def\hh{\widehat h}
\def\Hh{\widehat H}
\def\Uh{\widehat U}
\def\uh{\widehat u}
\def\Vh{\widehat V}
\def\vh{\widehat v}
\def\Xh{\widehat X}
\def\xh{\widehat x}

\def\yh{\widehat y}

\def\Zh{\widehat Z}
\def\zh{\widehat z}

\def\Vf{\mathfrak{V}}

\def\Phih{\widehat \Phi}

\def\betat{\widetilde{\beta}}

\def\gammat{\widetilde{\gamma}}

\def\xit{\widetilde{\xi}}

\def\Ut{\widetilde{U}}

\def\Xt{\widetilde{X}}

\def \Lim{\displaystyle\lim}
\def \Liminf{\displaystyle\liminf}
\def \Limsup{\displaystyle\limsup}

\def\esup{{\rm ess \, sup}}

\def \alphab {\boldsymbol{\alpha}}

\def \Xbb{\mathbf{X}}

\def \Zbb{\mathbf{Z}}

\def \1{\mathds{1}}

\def \alphab {\boldsymbol{\alpha}}
\def \betab {\boldsymbol{\beta}}

\def \gammab {\boldsymbol{\gamma}}

\def \xbb {\boldsymbol{x}}

\def \Gammabb {\overline{\Gamma}}

\def \Xbb{\mathbf{X}}

\def \Zbb{\mathbf{Z}}

\def\Er{{\rm E}}
\def\Gr{{\rm G}}
\def\Ir{{\rm I}}

\def\Rr{{\rm R}}

\def \Rr{\mathrm{R}}

\DeclareUnicodeCharacter{014D}{\=o}
 \title{ A non--exchangeable mean field control problem with controlled interactions
 
 }

\author{
 Mao Fabrice Djete\footnote{\'Ecole Polytechnique Paris, Centre de Math\'ematiques Appliqu\'ees, mao-fabrice.djete@polytechnique.edu. This work benefits from the financial support of the Chairs {\it Financial Risk} and {\it Finance and Sustainable Development}} 
    }
             \date{\today}

\begin{document}

\maketitle
 
\begin{abstract}

This paper introduces and analyzes a new class of mean--field control (\textsc{MFC}) problems in which agents interact through a \emph{fixed but controllable} network structure. 
In contrast with the classical \textsc{MFC} framework --- where agents are exchangeable and interact only through symmetric empirical distributions --- we consider systems with heterogeneous and possibly asymmetric interaction patterns encoded by a structural kernel, typically of graphon type. 
A key novelty of our approach is that this interaction structure is no longer static: it becomes a genuine \emph{control variable}. 
The planner therefore optimizes simultaneously two distinct components: a \emph{regular control}, which governs the local dynamics of individual agents, and an \emph{interaction control}, which shapes the way agents connect and influence each other through the fixed structural kernel. 

\medskip
We develop a generalized notion of relaxed (randomized) control adapted to this setting, prove its equivalence with the strong formulation, and establish existence, compactness, and continuity results for the associated value function under minimal regularity assumptions. 
Moreover, we show that the finite $n$--agent control problems with general (possibly asymmetric) interaction matrices converge to the mean--field limit when the corresponding fixed step--kernels converge in cut--norm, with asymptotic consistency of the optimal values and control strategies. 
Our results provide a rigorous framework in which the \emph{interaction structure itself is viewed and optimized as a control object}, thereby extending mean--field control theory to non--exchangeable populations and controlled network interactions.

\end{abstract}


\vspace{3mm}
\no{\bf MSC2010.} 60K35, 60H30, 91A13, 91A23, 91B30.

\section{Introduction}\label{sec:intro}

Classical mean--field models provide powerful frameworks for studying the collective behavior of large populations of interacting agents.  
In the standard setting without common noise, interactions are typically represented through the empirical distribution of the agents' states (and sometimes of their controls).  
Even though the presence of controls may suggest asymmetry, the dependence on the empirical distribution enforces a \emph{symmetric} structure: all agents interact in an identical way and are therefore \emph{exchangeable} in law.  
This symmetry assumption plays a crucial role in ensuring the tractability of the analysis, both from theoretical and numerical perspectives, and naturally leads to elegant limiting descriptions in terms of McKean--Vlasov dynamics and associated mean--field partial differential equations. See \citeauthor*{lasry2006jeux2} \cite{lasry2006jeux2,lasry2007mean}, \citeauthor*{huang2006large} \cite{huang2006large}, \citeauthor*{carmona2018probabilisticI} \cite{carmona2018probabilisticI}

\medskip
However, many systems of practical and theoretical interest violate the exchangeability hypothesis.  
Agents may belong to different classes or communities, interact through structured networks, or influence each other asymmetrically according to social, economic, or spatial heterogeneity.  
Examples include opinion formation in non--homogeneous social graphs, systemic risk models with heterogeneous exposures, energy networks with asymmetric couplings, and multi--population control systems (see some examples in \citeauthor*{MOJackson_social_2008} \cite{MOJackson_social_2008}).  
In such situations, the mean--field approximation based solely on symmetric empirical measures fails to capture the diversity of interaction patterns.

\medskip
To capture such heterogeneous or asymmetric interaction structures, recent research has focused on extending the classical mean--field framework beyond the exchangeable setting. See \citeauthor*{ErhanGraphon2023} \cite{ErhanGraphon2023}, \citeauthor*{jabin2024meanfieldlimitnonexchangeablesystems} \cite{jabin2024meanfieldlimitnonexchangeablesystems}, \citeauthor*{crucianelli2024interactingparticlesystemssparse} \cite{crucianelli2024interactingparticlesystemssparse}, \citeauthor*{Coppini2024NonlinearGM} \cite{Coppini2024NonlinearGM}, $\cdots$ for the study without control. In the case of controlled system, we refer to \citeauthor*{decrescenzo2024meanfieldcontrolnonexchangeable} \cite{decrescenzo2024meanfieldcontrolnonexchangeable}, \citeauthor*{cao2025probabilisticanalysisgraphonmean} \cite{cao2025probabilisticanalysisgraphonmean}, \citeauthor*{kharroubi2025stochasticmaximumprincipleoptimal} \cite{kharroubi2025stochasticmaximumprincipleoptimal}, $\cdots$. Building on the concept of \emph{graphons} developed by \citeauthor*{Lovsz2012LargeNA}~\cite{Lovsz2012LargeNA}, these papers analyze a limit formulation describing the behavior of systems with infinitely many interacting agents/particles. See \citeauthor*{Gao2020LinearQG} \cite{Gao2020LinearQG}, \citeauthor*{Aurell2021StochasticGG} \cite{Aurell2021StochasticGG}, \citeauthor*{10.1007/s00245-023-09996-y} \cite{10.1007/s00245-023-09996-y}, \citeauthor*{repec:spr:finsto:v:28:y:2024:i:2:d:10.1007_s00780-023-00527-9} \cite{repec:spr:finsto:v:28:y:2024:i:2:d:10.1007_s00780-023-00527-9}, \citeauthor*{doi:10.1287/moor.2022.1329} \cite{doi:10.1287/moor.2022.1329}, $\cdots$ for a game--theoretic approach.

\medskip
In most of the existing literature, the interaction structure among agents is assumed to be fixed, typically represented by a given network topology.  
However, in many realistic situations, the pattern of interactions itself can be influenced or designed — for instance, through communication protocols, social influence mechanisms, or network reconfiguration.  
This raises a natural and intriguing question: \emph{can one control the interaction structure of the system?}  
From a mathematical perspective, this corresponds to treating the underlying graph or connectivity kernel as a \emph{control variable} rather than a fixed parameter.  
The objective of this paper is precisely to investigate such a setting.  
We now provide a high--level description of the framework we consider (see \Cref{sec:main} for details).
Let $\Xbb^n=(X^{1,n},\dots,X^{n,n})$ denote the private state processes of $n$ interacting agents, each evolving according to the system of stochastic differential equations
\begin{align*}
    \mathrm{d}X^i_t 
    &= b\!\left(t, X^i_t, M^{1,i,n}_{\gammab^n,t}, M^{2,i,n}_{\gammab^n,t}, \alpha^{i,n}(t,\Xbb^n) \right)\mathrm{d}t 
    + \sigma(t, X^i_t)\,\mathrm{d}W^i_t, \quad t \in [0,T],
    \\
    M^{1,i,n}_{\gammab^n,t} 
    &= \frac{1}{n} \sum_{j=1}^n \delta_{\left( \gamma^n_{ij}(t, \Xbb^n),\, X^j_t,\, \xi^n_{ij} \right)}, 
    \qquad
    M^{2,i,n}_{\gammab^n,t} 
    = \frac{1}{n} \sum_{j=1}^n \delta_{\left( \gamma^n_{ji}(t, \Xbb^n),\, X^j_t,\, \xi^n_{ij} \right)},
\end{align*}
where $T>0$ is a fixed time horizon and $(W^1,\dots,W^n)$ are independent Brownian motions.  
For each agent $i$, the function $\alpha^{i,n}$ is a Borel map representing the agent’s \emph{regular} control, while $\gammab^n = (\gamma^n_{ij})_{1 \le i,j \le n}$ corresponds to the collection of \emph{interaction controls}.

\medskip
A key feature of this formulation is the introduction of the pair $(M^{1,i,n}_{\gammab^n,t}, M^{2,i,n}_{\gammab^n,t})$, which encodes both outgoing and incoming interaction effects.  
The deterministic matrix $(\xi^n_{ij})_{1 \le i,j \le n}$ specifies a fixed underlying structure of interactions—typically representing the strength or type of potential connection between agents.  
In contrast, the terms $(\gamma^n_{ij})_{1 \le i,j \le n}$ represent \emph{control variables} that determine how these potential interactions are activated or modulated over time.  
In particular,
\[
    \text{the variable }\gamma^n_{ij}\text{ describes the  interaction decision of agent }i\text{ towards agent }j.
\]
The simultaneous presence of both $M^{1,i,n}_{\gammab^n,t}$ and $M^{2,i,n}_{\gammab^n,t}$ is \emph{fundamental} to the model.  
Indeed, $M^{1,i,n}_{\gammab^n,t}$ characterizes the \emph{outgoing influence} of agent $i$—that is, how her/his chosen interaction strategy shapes her/his connections with the rest of the population. 
Conversely, $M^{2,i,n}_{\gammab^n,t}$ captures the \emph{incoming influence}, reflecting how agent $i$ is impacted by the interaction decisions of others towards her/him.  
This dual representation allows us to disentangle and analyze both the \emph{active} and \emph{passive} roles each agent plays within the network of interactions.
Incorporating both components is essential for modeling realistic systems in which influence and response are not necessarily symmetric, such as social, economic, or networked engineering applications. With the \emph{regular} controls $\alphab^n:=(\alpha^{i,n})_{1 \le i \le n}$ and \emph{interaction} controls $\gammab^n:=(\gamma_{i,j}^n)_{1 \le i,j \le n}$, the reward of agent $i$ is
\begin{align*}
    \E \left[ \int_0^T L\left(t, X^i_t, M^{1,i,n}_{\gammab^n,t}, M^{2,i,n}_{\gammab^n,t}, \alpha^{i,n}(t, \Xbb) \right)\mathrm{d}t + g \left( X^i_T, \Rr^{i,n}_T \right)  \right]\quad\mbox{with}\quad \Rr^{i,n}_T:=\frac{1}{n} \sum_{j=1}^n \delta_{\left( X^j_T ,\, \xi^n_{ij}  \right)}.
\end{align*}

In this work, we consider a setting in which a central planner seeks to maximize the average performance of a large population of interacting agents.  
When the number of players $n$ becomes very large, it is natural to study the limiting behavior of the corresponding $n$--player control problem, leading to a \emph{mean--field control} formulation that captures the aggregate dynamics of the system.  
In our framework, this limiting problem can be informally described as follows (see \Cref{sec:main} for the precise formulation):  
the planner selects a pair of controls—an \emph{interaction control} $\gamma$ and a \emph{regular control} $\alpha$—in order to maximize
\begin{align*}
    \E \Bigg[ 
        \int_0^T 
            L \!\Big( s, X_s, M^{1,\gamma}_{\mu,s}(X_s,U), M^{2,\gamma}_{\mu,s}(X_s,U), \alpha(s,X_s,U) \Big)\,\mathrm{d}s 
        + g\!\left( X_T, \Rr^{\gamma,\alpha}_T(U) \right) 
    \Bigg],
\end{align*}
where the random measure
\[
    \Rr^{\gamma,\alpha}_t(u) := \Lc\!\left( X_t,\,\Gr(u,U) \right),
    \qquad 
    \mu_t := \Lc\!\left( X_t,\,U \right),
\]
and the state process $(X_t)_{t \in [0,T]}$ satisfies the McKean--Vlasov type dynamics
\begin{align*}
    \mathrm{d}X_t
    &= b\!\left( t, X_t, M^{1,\gamma}_{\mu,t}(X_t,U), M^{2,\gamma}_{\mu,t}(X_t,U), \alpha(t,X_t,U) \right)\mathrm{d}t 
    + \sigma(t,X_t)\,\mathrm{d}W_t,
    \\
    M^{1,\gamma}_{\mu,t}(x,u)
    &:= \Lc\!\left( \gamma(t,x,u,X_t,U),\, X_t,\, \Gr(u,U) \right),
    \qquad
    M^{2,\gamma}_{\mu,t}(x,u)
    := \Lc\!\left( \gamma(t,X_t,U,x,u),\, X_t,\, \Gr(u,U) \right).
\end{align*}

\medskip
In this formulation, the pair $(X_t,U)$ represents respectively the \emph{state} and the \emph{label} of an individual agent, and $\Gr : [0,1]^2 \to \Er$ denotes the limiting \emph{interaction kernel} (or \emph{graphon}) associated with the fixed structure of connections, obtained as the limit of the step--kernels built from the adjacency matrices $(\xi^n_{ij})_{1 \le i,j \le n}$ in the finite--player system.  
This structure enables the limiting dynamics to retain the heterogeneity and asymmetry of the underlying network while allowing for control over how interactions are formed and weighted.

\medskip
To the best of our knowledge, this type of control problem—where the interaction structure itself is a controlled variable—has not been studied before in the literature.  
The most closely related contribution we are aware of is the work of \citeauthor*{10.1214/23-AAP1993} \cite{10.1214/23-AAP1993}, which investigates a similar idea within a \emph{mean--field game} (\textsc{MFG}) framework rather than a \emph{mean--field control} (\textsc{MFC}) one.  
In their setting, the interaction control takes a particularly tractable form, arising in a model of \emph{mutual holdings}, where the optimal interaction decision depends only on the target agent $j$ and is independent of the acting agent $i$.  
This structure greatly facilitates the analysis but restricts the generality of the model.  
By contrast, the framework developed in the present paper operates in a more general \textsc{MFC} setting, allowing for fully heterogeneous and asymmetric interaction decisions. It is worth mentioning that this setting incorporate the framework of graphon mean field control studied in  \cite{decrescenzo2024meanfieldcontrolnonexchangeable}, \cite{cao2025probabilisticanalysisgraphonmean}, \cite{kharroubi2025stochasticmaximumprincipleoptimal}.

\medskip
From a mathematical standpoint, introducing \emph{interaction controls} in a non--exchangeable framework presents significant analytical challenges.  
In particular, we aim to study this problem under minimal regularity assumptions on the coefficients—typically only continuity or Lipschitz continuity—without relying on convexity or differentiability.  
In such low--regularity regimes, a classical and powerful approach is the \emph{relaxed control} (or compactification) method, originally formalized by \citeauthor*{el1987compactification}  \cite{el1987compactification} and later adapted to the \textsc{MFC} setting by \citeauthor*{lacker2017limit} \cite{lacker2017limit} and \citeauthor*{djete2019general} \cite{djete2019general}.

\medskip
However, the traditional relaxed formulation, even in its mean--field control variant, is not directly applicable to our setting because of the specific structure of the \emph{interaction control} $\gamma(t,x,u,x',u')$.  
Indeed, the term $M^{1,\gamma}_{\mu,t}$ treats $(x,u)$ as fixed while integrating over $(x',u')$ to form the interaction law, whereas the term $M^{2,\gamma}_{\mu,t}$ does the opposite.  
The same control function $\gamma$ thus appears in two fundamentally different roles—acting once as an \emph{outgoing} and once as an \emph{incoming} influence—which destroys the symmetry structure usually exploited in relaxed formulations.  
This asymmetry, combined with the dependence on the interaction kernel $\Gr$, renders standard compactness and measurability arguments insufficient.  

\medskip
To overcome these obstacles, inspired by the techniques developed in \citeauthor*{MFD-2020-closed} \cite{MFD-2020-closed}, we develop an \emph{extended notion of admissibility and relaxation}, specifically designed to handle non--symmetric couplings and controls acting directly on the interaction structure.  
This generalized framework provides the analytical foundation required to establish compactness, stability, and existence results for non--exchangeable mean--field control problems with controlled interactions.
The main contributions of this work can be summarized as follows:
\begin{enumerate}
    \item We introduce a general notion of \emph{relaxed} (or \emph{randomized}) controls specifically adapted to the structure of controlled interactions.  
    This formulation guarantees compactness and stability of the induced laws of controlled trajectories, even in the absence of exchangeability or convexity.

    \item We establish the \emph{equivalence} between the relaxed formulation and the original (strong) formulation of the mean--field control problem.  
    As a consequence, we prove the \emph{existence} of optimal relaxed controls, and under suitable convexity assumptions, the existence of optimal (non--relaxed) controls.  
    Furthermore, we show that the value function $\nu \mapsto V_{\mathrm{MFC}}(\nu)$ is \emph{continuous} with respect to the initial distribution in the Wasserstein topology.

    \item We prove that the \emph{closed--loop} formulation considered here—where controls depend on the state process—is equivalent to the \emph{open--loop} formulation, in which controls depend only on the exogenous sources of randomness.  
    This equivalence extends classical results in standard mean--field control theory to the present non--exchangeable and interaction--controlled setting.

    \item We demonstrate the convergence of finite $n$--agent control problems, with general (possibly non--symmetric) interaction matrices $(\xi^n_{ij})_{1 \le i,j \le n}$, toward the mean--field limit.  
    In particular, we show that convergence holds whenever the \emph{step--kernels} associated with $(\xi^n_{ij})_{1 \le i,j \le n}$ converge in \emph{cut--norm} to a limiting kernel $\Gr$, thereby connecting discrete interaction structures with their continuum limit.
\end{enumerate}

\medskip
In doing so, we extend the classical mean--field control framework far beyond the symmetric (exchangeable) paradigm.  
Our approach provides a unified and flexible formulation capable of modeling large populations of agents interacting through general, possibly heterogeneous or asymmetric, network structures—represented in the continuum by measurable graphon--type kernels.  
Crucially, this framework not only accounts for the influence of such interaction structures but also allows them to be \emph{controlled} or \emph{optimized} as part of the decision process (see also the examples and discussions in \Cref{sec:examples}).  

\medskip
From a methodological perspective, the resulting theory bridges several analytical domains: it combines the probabilistic treatment of McKean--Vlasov dynamics with tools from graph limit theory, stochastic control, and measure--valued processes.  
This synthesis opens new avenues for studying high--dimensional and networked control systems in a mathematically rigorous yet highly adaptable way.


\medskip
\noindent\textbf{Outline of the paper.}
\Cref{sec:main} introduces the general setup and states the main results, including the relaxed formulation, the equivalence between relaxed and strong controls, and the existence and continuity properties of the value function. 
The end of \Cref{sec:main} establishes the connection between the finite $n$--player optimization problems and their mean--field counterparts, proving the convergence of value functions and optimal controls as the number of agents grows. 
\Cref{sec:examples} illustrates the framework through two representative examples highlighting, respectively, the structure of optimal interaction controls and an application to social network dynamics. 
Finally, \Cref{sec:proof} is devoted to the technical proofs and auxiliary results underlying the main theorems.

\medskip
{\bf \large Notations.}
\noindent $(i)$
Let $(E,\Delta)$ be a Polish space and $p \ge 1$.  
We denote by $\Pc(E)$ the set of all Borel probability measures on $E$, and by $\Pc_p(E)$ the subset of measures $\mu \in \Pc(E)$ satisfying 
$\int_E \Delta(e,e_0)^p\,\mu(\mathrm{d}e) < \infty$ for some (hence any) $e_0 \in E$.  
For $p=0$, we simply set $\Pc_0(E):=\Pc(E)$.  
The space $\Pc(E)$ is endowed with the weak topology, while $\Pc_p(E)$ is equipped with the $p$--Wasserstein distance
\[
	\Wc_p(\mu,\mu')
	~:=~
	\Bigg(
		\inf_{\lambda \in \Lambda(\mu,\mu')} 
		\int_{E\times E} \Delta(e,e')^p\,\lambda(\mathrm{d}e,\mathrm{d}e')
	\Bigg)^{1/p},
\]
where $\Lambda(\mu,\mu')$ is the collection of all couplings of $(\mu,\mu')$.  
It is well known that $(\Pc_p(E),\Wc_p)$ is a Polish space (\cite[Theorem~6.18]{villani2008optimal}).  
For any $\mu \in \Pc(E)$ and $\mu$--integrable $\varphi:E\to\R$, we write $\langle \varphi,\mu \rangle := \int_E \varphi(e)\,\mu(\mathrm{d}e).$
For two metric spaces $(E,\Delta)$ and $(E',\Delta')$ and $(\mu,\mu') \in \Pc(E)\times\Pc(E')$, we denote by $\mu\otimes\mu'$ their product measure on $E\times E'$.

\medskip
\noindent $(ii)$
Let $(\Omega,\Hc,\P)$ be a probability space, and $\Gc\subset\Hc$ a sub--$\sigma$--algebra.  
For a Polish space $E$ and a random variable $\xi:\Omega\to E$, we denote by
\[
	\Lc^{\P}(\xi \mid \Gc)(\omega)
	\quad\text{or equivalently}\quad
	\P^{\Gc}_{\omega}\circ\xi^{-1}
\]
the conditional distribution of $\xi$ given $\Gc$ under $\P$.

\medskip
\noindent $(iii)$
Let $\N^*$ denote the set of positive integers, and $T>0$.  
For a Polish space $(\Sigma,\rho)$, we denote by $C([0,T];\Sigma)$ the space of continuous paths on $[0,T]$ taking values in $\Sigma$.  
When $\Sigma=\R^k$ for some $k\in\N^*$, we write simply $\Cc^k:=C([0,T];\R^k)$.  
Finally, for a measurable space $E$, we denote by $\M(E)$ the set of Borel measures $q(\mathrm{d}t,\mathrm{d}e)$ on $[0,T]\times E$ whose time--marginal is the Lebesgue measure, i.e.
\[
	q(\mathrm{d}e,\mathrm{d}t)=q(t,\mathrm{d}e)\,\mathrm{d}t,
	\quad\text{for some measurable family } (q(t,\cdot))_{t\in[0,T]} \subset \Pc(E).
\]

\section{Setup and main results} \label{sec:main}

\medskip
Let $T>0$ be a maturity and $d \in \N^\star$ be a dimension. Unless specified otherwise, all random elements are defined on a fixed filtered probability space 
$(\Omega, \H, (\Hc_t)_{t \in [0,T]}, \P)$ satisfying the usual conditions.  
We denote by $W$ an $\R^d$--valued $\H$--Brownian motion,  
by $U$ an $[0,1]$--valued random variable uniformly distributed on $[0,1]$ independent of $W$,  
and by $\xi$ an $\Hc_0$--measurable random variable with initial distribution $\nu \in \Pc(\R^d)$. Let $\Er$ be a compact metric space that represents the set of possible values of the \emph{fixed interaction structure}.  
The dynamics and reward of the system are governed by the following bounded measurable maps:
\[
(b,L): [0,T] \times \R^d \times \Pc\!\bigl(A_{\rm int} \times \R^d \times [0,1] \times \Er\bigr)^2 \times A_{\rm reg}
    \;\longrightarrow\; \R^d \times \R,
\]
which determine respectively the drift and the running cost, together with
\[
\sigma: [0,T] \times \R^d \longrightarrow \S^{d },
\qquad
g: \R^d \times \Pc(\R^d \times \Er) \longrightarrow \R,
\]
where $\S^{d}$ denotes the space of $d \times d$ matrices.  
We assume that the diffusion coefficient $\sigma$ is Lipschitz continuous in the state variable $x$, uniformly in time $t$.  
Additional regularity and structural conditions on $(b,L,g)$ will be introduced later, when required for the analysis.

\medskip
\subsection{The strong (\emph{closed--loop}) formulation}
Let $A_{\rm reg}$ and $A_{\rm int}$ be two compact sets that represent the set of values of the \emph{regular} control and \emph{interaction} control respectively. We write $\Cc_{\mathrm{pr}}^d:=C\left([0,T];\Pc(\R^d \x [0,1]) \right)$.
We denote by \( \Ac_{\rm reg} \) the set of \emph{regular controls}, consisting of progressively Borel measurable maps 
\[
    \alpha : [0,T] \times \R^d \times [0,1] \to A_{\mathrm{reg}} .
\] 
The set of \emph{interaction controls} is denoted by \( \Ac_{\mathrm{int}} \), and consists of Borel measurable maps 
\[
    \gamma : [0,T] \times (\R^d \times [0,1]) \times (\R^d \times [0,1]) \to A_{\mathrm{int}} .
\]
Let \( \Gr : [0,1] \times [0,1] \to \Er \) denote a kernel representing the step kernel of a graph. We say that $\eta \in \Pc_\nu$ if $\eta \in \Pc(\R^d \x [0,1])$ with $\eta(\mathrm{d}x,[0,1])=\nu(\mathrm{d}x)$ and $\eta(\R^d, \mathrm{d}u)=\mathrm{d}u$. The set $\Pc_\nu$ will play the role of the initial measures. Let $\eta \in \Pc_\nu$. For any pair of controls \( \alpha \in \Ac_{\rm reg} \), \( \gamma \in \Ac_{\mathrm{int}} \), we define the controlled process 
\[
    (X^{\gamma,\alpha}_s)_{s \in [t,T]} := (X_s)_{s \in [0,T]},\quad (\mu^{\gamma,\alpha}_s)_{s \in [t,T]} := (\mu_s)_{s \in [0,T]}
\] 
with initial law \( \Lc(X^{\gamma,\alpha}_0,U) = \eta \), satisfying for \( s \in [0,T] \),
\begin{align} \label{eq:general_mck_ci}
    \mathrm{d}X_s
    = b\!\left( s, X_s, M^{1,\gamma}_{\mu,s}(X_s,U), M^{2,\gamma}_{\mu,s}(X_s,U), \alpha(s,X_s,U) \right) \mathrm{d}s + \sigma(s,X_s)\,\mathrm{d}W_s.
\end{align}

The auxiliary measures are given by:
\[
    \Rr^{\gamma,\alpha}_s(u) := \Lc\!\left( X_s,\,\Gr(u,U) \, \right),
    \qquad 
    \mu_s := \Lc\!\left(X_s,\,U \, \right),
\]
and
\begin{align*}
    M^{1,\gamma}_{\mu,s}(s,x,u)
    &:= \Lc\!\left( \gamma(s,x,u,X_s,U),\, X_s,\,U,\, \Gr(u,U) \, \right), \\
    M^{2,\gamma}_{\mu,s}(s,x,u)
    &:= \Lc\!\left( \gamma(s,X_s,U,x,u),\, X_s,\,U,\, \Gr(u,U) \, \right).
\end{align*}

\medskip

The corresponding mean--field control problem is 
\[
    V_{\mathrm{MFC}}(\nu)
    := \sup_{\eta \,\in\, \Pc_\nu} \sup_{\gamma \in \Ac_{\mathrm{int}},\, \alpha \in \Ac_{\rm reg}} J(\eta,\gamma,\alpha),
\]
with performance functional
\begin{align*}
    J(\eta,\gamma,\alpha)
    := \E \Bigg[ \int_0^T L \!\Big( s, X_s, M^{1,\gamma}_{\mu,s}(X_s,U), M^{2,\gamma}_{\mu,s}(X_s,U), \alpha(s,X_s,U) \Big)\,\mathrm{d}s 
    + g\!\left( X_T, \Rr^{\gamma,\alpha}_T(U) \right) \Bigg].
\end{align*}

We now introduce the standing assumptions ensuring the well--posedness of the control problem.

\begin{assumption}\label{assum:main1_MF_CI}
\begin{enumerate}
    \item[\textnormal{(i)}] 
    The drift $b(t,x,m^1,m^2,a)$ is Lipschitz continuous in $(m^1,m^2)$ uniformly in $(t,x,a)$, and the functions 
    \[
        (t,x,m^1,m^2,a,m) \longmapsto \bigl((b,L)(t,x,m^1,m^2,a),\, g(x,m)\bigr)
    \]
    are continuous in $(x,m^1,m^2,a,m)$ for each fixed $t \in [0,T]$.

    \item[\textnormal{(ii)}] 
    \textbf{Non--degeneracy:} there exists $\theta>0$ such that 
    \[
        \theta I_d \le \sigma\sigma^\top(t,x), \qquad \forall (t,x)\in[0,T]\times\R^d.
    \]
\end{enumerate}
\end{assumption}

\begin{remark}
    $(i)$ Under {\rm \Cref{assum:main1_MF_CI}}, if $b$ is also uniformly Lipschitz in $x$, the McKean--Vlasov equation \eqref{eq:general_mck_ci} admits a unique strong solution whenever the control functions $\gamma$ and $\beta$ are Lipschitz continuous $($see for instance {\rm \cite[ TheoremA.3.]{djete2019mckean}}$)$. 
However, since our objective is to analyze the control problem under the more general framework of {\rm \Cref{assum:main1_MF_CI}}, we shall not rely on this Lipschitz regularity. 
In this setting, only the \emph{weak uniqueness} of \eqref{eq:general_mck_ci} will in general be available $($strong existence remains true however, see {\rm \citeauthor*{AJVeretennikov_1981} \cite{AJVeretennikov_1981}}$)$. 
As we shall see in {\rm \Cref{prop:relaxed_uniqueness}}, such property can still be established under mild measurability and integrability conditions on the controls $(\gamma,\beta)$.

\medskip
$(ii)$ We emphasize that the boundedness of the coefficients and the compactness of the sets $A_{\rm int}$, $A_{\rm reg}$, and $\Er$ are not essential assumptions. 
These conditions are imposed only to streamline the exposition and highlight the core arguments of the paper. 
All results could be extended to more general unbounded or non--compact settings under standard growth and integrability conditions, at the cost of additional technicalities in the proofs.

\end{remark}

\medskip
\subsection{The relaxed/randomized formulation} One of the main objectives of this work is to establish fundamental properties of the above control problem---in particular, the existence of optimal controls and the continuity of the value function $\nu \mapsto V_{\mathrm{MFC}}(\nu)$---under the general framework of {\rm \Cref{assum:main1_MF_CI}}. 
To handle this level of generality, it is more convenient to adopt a {\it randomized} (or {\it relaxed}) formulation of the control problem, which will provide the necessary compactness and stability properties for the subsequent analysis. 
We now introduce this relaxed formulation in detail.

\medskip
Let us denote by $\Acb_{\rm int}$ the set of Borel measurable maps
\[
    \Acb_{\rm int} \;:=\; 
    \Bigl\{\, 
        \gammabb : [0,T] \times \bigl( \R^d \times [0,1]^2 \bigr)^2  \times [0,1]^2 
        \;\longrightarrow\; A_{\rm int} 
    \,\Bigr\},
\]
and by $\Acb_{\rm reg}$ the set of Borel measurable maps
\[
    \Acb_{\rm reg} \;:=\; 
    \Bigl\{\, 
        \betabb : [0,T] \times \bigl( \R^d \times [0,1]^2 \bigr)  \times [0,1]^2 
        \;\longrightarrow\; A_{\rm reg} 
    \,\Bigr\}.
\]
It is straightforward to see $ \Ac_{\rm int} \subset \Acb_{\rm int}$ and $ \Ac_{\rm reg} \subset \Acb_{\rm reg}$.  The class $\Acb_{\rm int}$ (resp. $\Acb_{\rm reg}$) thus represents a {\it randomized} version of the control set $\Ac_{\rm int}$ (resp. $\Ac_{\rm reg}$), where the additional argument in $[0,1]$ serves as an auxiliary randomization variable. For $\eta \in \Pc_\nu$ and any pair of controls $(\gammabb,\alphabb) \in \Ac_{\rm int} \times \Ac_{\rm reg}$, 
we denote by 
\[
    \left(\mu^{\gammabb,\betabb}_t := \Lc\bigl(\Xb^{\gammabb,\betabb}_t,U \bigr)\right)_{t \in [0,T]}
\]
a weak solution to the McKean--Vlasov dynamics with initial distribution 
$\Lc\bigl(\Xb^{\gammabb,\betabb}_0,U \bigr)=\eta$ and dynamics given by
\begin{align} \label{eq:relaxed_general_mck_ci}
    \mathrm{d}X_s
    &= \int_{[0,1]^3} 
        b\!\left( 
            s, X_s,\,
            \Mb^{1,\gammabb}_{\mu,s}(X_s,U,v,\pi),\,
            \Mb^{2,\gammabb}_{\mu,s}(X_s,U,v,\pi),\,
            \alphabb(s,X_s,U,v,\widetilde{v},\pi)
        \right)
        \,\mathrm{d}v\,\mathrm{d}\widetilde{v}\,\mathrm{d}\pi\, \mathrm{d}s
        + \sigma(s,X_s)\,\mathrm{d}W_s,
\end{align}
where the conditional distributions appearing in the interaction terms are defined by
\[
    \overline{\Rr}^{\gammabb,\alphabb}_s(u) := \Lc\!\bigl(X_s,\,\Gr(u,U)\bigr),
    \qquad 
    \mu_s := \Lc\!\bigl(X_s,\,U\bigr),
\]
and, for each $(s,x,u,v,\pi) \in [0,T]\times\R\times[0,1]^2$,
\begin{align*}
    \Mb^{1,\gammabb}_{\mu,s}(x,u,v,\pi)
    &:= \Lc\!\Bigl( 
        \gamma\bigl(s,x,u,v,X_s,U,V,\Vb,\pi\bigr),\, 
        X_s,\,U,\,\Gr(u,U)
    \Bigr), \\
    \Mb^{2,\gammabb}_{\mu,s}(x,u,v,\pi)
    &:= \Lc\!\Bigl( 
        \gamma\bigl(s,X_s,U,V,x,u,v,\Vb,\pi\bigr),\, 
        X_s,\,U,\,\Gr(u,U)
    \Bigr).
\end{align*}
Here, the auxiliary variables $V$ and $\Vb$ are independent, uniformly distributed over $[0,1]$, and independent of $(X,U,W)$. 

The associated {\it relaxed} mean--field control problem is 
\[
    \Vb_{\mathrm{MFC}}(\nu)
    := \sup_{\eta \,\in\, \Pc_\nu} \,\sup_{\gammabb \in \Acb_{\mathrm{int}},\, \alphabb \in \Acb_{\rm reg}} \Jb(\eta,\gammabb,\alphabb),
\]
with the reward
\begin{align*}
    &\Jb(\eta,\gammabb,\alphabb)
    \\
    &:= \E \Bigg[ \int_0^T \int_{[0,1]^3} L \!\Big( s, X_s, \Mb^{1,\gammabb}_{\mu,s}(X_s,U,v,\pi), \Mb^{2,\gammabb}_{\mu,s}(X_s,U,u,\pi), \alpha(s,X_s,U,v,\widetilde{v},\pi) \Big)\,\mathrm{d}v\,\mathrm{d}\widetilde{v}\,\mathrm{d}\pi\,\mathrm{d}s 
    + g\!\left( X_T, \Rrb^{\gamma,\alpha}_T(U) \right) \Bigg].
\end{align*}

\begin{remark}
$(i)$ Readers accustomed to relaxed controls may find our parametrization unusual.  
In the classical approach, a predictable $K$--valued control $(\beta_t)_{t\in[0,T]}$ is replaced by a predictable
kernel $\bigl(\Lambda_t(\mathrm{d}u)\bigr)_{t\in[0,T]}$ on $K$ $($often written $\Lambda_t(\mathrm{d}u)\mathrm{d}t$ $)$; see, e.g., {\rm \cite{el1987compactification}} for the standard formulation.

\smallskip
Our definition is \emph{equivalent} to the classical one.  
Indeed, since $K$ is a Polish space, there exists a Borel “sampling” map
\[
S:\Pc(K)\times[0,1]\longrightarrow K
\]
with the property that for any $\mu\in\Pc(K)$ and $V\sim{\rm Unif}[0,1]$,
\[
\Lc\!\big(S(\mu,V)\big)=\mu.
\] 
Therefore, given a predictable kernel $\Lambda_t(\omega)(\mathrm{d}u)$, the process
\[
\alpha_t(\omega,V)\;:=\;S\bigl(\Lambda_t(\omega),V\bigr)
\]
is a $K$--valued predictable control such that
$\Lc\!\big(\alpha_t(\omega,V)\big)=\Lambda_t(\omega)$ for a.e.\ $(\omega,t)$.  
Conversely, any predictable selector $\alpha_t(\omega,V)$ induces the relaxed kernel
$\Lambda_t(\omega):=\Lc\!\big(\alpha_t(\omega,V)\,|\,\omega,t\big)$.
Hence our formulation simply \emph{parametrizes} randomized/relaxed controls by measurable selectors driven by i.i.d.\ uniforms, rather than by measure--valued processes.

\smallskip
This representation is particularly convenient here because it mirrors the structure of the interaction class $\Ac_{\rm int}$ through explicit auxiliary uniforms:
\begin{itemize}
\item $v$ randomizes the control conditioned on the current pair $(x,u)$,
\item $v'$ randomizes the component associated with $(x',u')$,
\item $\overline v$ randomizes jointly over $(x,u,v,x',u',v')$ when needed (e.g., to couple components),
\item $\pi$ randomizes across the entire augmented tuple $(t,x,u,v,x',u',v',\overline v)$ (e.g., to form mixtures).
\end{itemize}
These auxiliary variables provide a canonical sampling device that replaces abstract kernels by concrete, measurable control maps without loss of generality.

\medskip
$(ii)$ We will show that, for every pair of randomized controls $(\gammabb,\alphabb) \in \Acb_{\rm int} \times \Acb_{\rm reg}$, 
there exists a unique weak solution to the McKean--Vlasov equation \eqref{eq:relaxed_general_mck_ci}.  
In particular, the corresponding flow of marginal laws 
\[
    \bigl( \Lc\bigl(\Xb^{\gammabb,\betabb}_t, U \bigr) \bigr)_{t \in [0,T]}
\]
is uniquely determined by $(\gammabb,\alphabb)$ and the initial distribution $\eta$.  
A detailed proof of this well--posedness result can be found in  {\rm \Cref{prop:relaxed_uniqueness}}.

\end{remark}
For any pair of randomized controls $(\gammabb,\alphabb) \in \Acb_{\rm int} \times \Acb_{\rm reg}$ 
and controls $(\gamma,\alpha) \in \Ac_{\rm int} \times \Ac_{\rm reg}$, 
we define the associated joint measures
\begin{align*}
    &\Lambdab^{\gammabb,\betabb}_t(\mathrm{d}x,\mathrm{d}u,\mathrm{d}r_1,\mathrm{d}r_2,\mathrm{d}a)\,\mathrm{d}t
    \\
    &:= 
    \Lc\!\left( 
        X_t,\,U,\,
        \Mb^{1,\gamma}_{\mu,t}(X_t,U,V,\Vf),\,
        \Mb^{2,\gamma}_{\mu,t}(X_t,U,V,\Vf),\,
        \alpha(t,X_t,U,V,\Vt,\Vf)
    \right)
    (\mathrm{d}x,\mathrm{d}u,\mathrm{d}r_1,\mathrm{d}r_2,\mathrm{d}a)\,\mathrm{d}t,
\end{align*}
and
\begin{align*}
    \Lambda^{\gamma,\alpha}_t(\mathrm{d}x,\mathrm{d}u,\mathrm{d}r_1,\mathrm{d}r_2,\mathrm{d}a)\,\mathrm{d}t
    &:=
    \Lc\!\left(
        X_t,\,U,\,
        M^{1,\gamma}_{\mu,t}(X_t,U),\,
        M^{2,\gamma}_{\mu,t}(X_t,U),\,
        \alpha(t,X_t,U)
    \right)
    (\mathrm{d}x,\mathrm{d}u,\mathrm{d}r_1,\mathrm{d}r_2,\mathrm{d}a)\,\mathrm{d}t
\end{align*}
where $(\Vt,\Vf) \sim {\rm Unif}([0,1])^{\otimes 2}$ and $(\Vt,\Vf) \perp (X_t,U,V,\Vb)$.
We then introduce the corresponding collections of admissible pairs:
\begin{align*}
    \Pib(\nu)
    &:= 
    \Bigl\{
        \bigl( \mu^{\gammabb,\alphabb},\, \Lambdab^{\gammabb,\alphabb} \bigr)
        \;:\;
        (\gammabb,\alphabb) \in \Acb_{\rm int} \times \Acb_{\rm reg}
    \Bigr\},
    \qquad
    \Pi(\nu)
    :=
    \Bigl\{
        \bigl( \mu^{\gamma,\alpha},\, \Lambda^{\gamma,\alpha} \bigr)
        \;:\;
        (\gamma,\alpha) \in \Ac_{\rm int} \times \Ac_{\rm reg}
    \Bigr\}.
\end{align*}

The next theorem shows that introducing randomization in the control variables does not modify the attainable set of laws, and hence both formulations of the mean--field control problem are equivalent.

\begin{theorem}
    \label{thm:equivalence_relaxed_strong}
    Let {\rm\Cref{assum:main1_MF_CI}} hold, and let $\nu \in \Pc_p(\R^d)$ with $p \in \{0\} \cup [1,\infty)$. 
    Then the set $\Pib(\nu)$ is closed in the Wasserstein topology $\Wc_p$, 
    and moreover, it coincides with the topological closure $($in $\Wc_p$ $)$ of the set $\Pi(\nu)$. 
    Consequently, the relaxed and strong formulations of the mean--field control problem are equivalent, i.e.,
    \begin{align*}
        \Vb_{\mathrm{MFC}}(\nu)
        \;=\;
        V_{\mathrm{MFC}}(\nu).
    \end{align*}
\end{theorem}

\begin{remark}
    $(i)$ {\rm\Cref{thm:equivalence_relaxed_strong}} highlights a fundamental structural property of the mean--field control problem: 
    the introduction of randomization through $(\gammabb,\alphabb)$ does not enlarge the attainable set of law/control pairs. 
    In other words, every relaxed control can be approximated, in the Wasserstein sense, by a sequence of strong controls. 
    This equivalence allows us to work indifferently within either formulation depending on the analytical setting --- 
    for instance, the relaxed formulation is better suited for compactness and stability arguments.

    \medskip
    $(ii)$ It is also worth emphasizing that {\rm\Cref{thm:equivalence_relaxed_strong}} highlights a key structural feature of the problem: 
what ultimately matters is the law-valued pair $\bigl(\mu^{\gamma,\alpha},\,\Lambda^{\gamma,\alpha}\bigr)$ and its associated {\rm Fokker--Planck} equation, rather than the individual sample paths of the controlled process $X^{\gamma,\alpha}$ itself.  
In other words, the dynamics and optimization depend only on the evolution of the joint distribution of states and controls, which fully encodes the system’s behavior at the mean-field level.  
This viewpoint often leads to a more tractable analysis, as illustrated in {\rm\cite{MFD-2020,MFD-2020-closed}}, where focusing directly on the measure dynamics provides a streamlined route to existence, stability, and convergence results.

\end{remark}

The previous result implies both the existence of an optimal control and 
the regularity of the value function with respect to the initial law.

\begin{proposition}
    \label{prop:existence_continuity}
    Under the assumptions of {\rm\Cref{thm:equivalence_relaxed_strong}}, for any $\nu \in \Pc_p(\R^d)$ with $p \in \{0\} \cup [1,\infty)$, 
    there exists an optimal pair of randomized controls 
    $(\gammabb^\star,\alphabb^\star) \in \Acb_{\rm int} \times \Acb_{\rm reg}$ and $\eta^\star \in \Pc_\nu$
    such that
    \begin{align*}
        V_{\mathrm{MFC}}(\nu)
        \;=\;
        \Jb\bigl(\eta^\star, \gammabb^\star, \alphabb^\star \bigr).
    \end{align*}
    Moreover, the value function
    \[
        V_{\mathrm{MFC}} : \Pc_p(\R^d) \longrightarrow \R
    \]
    is continuous with respect to the Wasserstein topology.
\end{proposition}

\begin{remark}
    $(i)$ The existence of an optimal control and the continuity of the value function obtained in 
    {\rm\Cref{prop:existence_continuity}}
    are direct consequences of the topological closure property established in {\rm\Cref{thm:equivalence_relaxed_strong}}. 
    The compactness of $\Pib(\nu)$ in the Wasserstein topology ensures the existence of minimizers,
    while the continuity of the coefficients in {\rm\Cref{assum:main1_MF_CI}} 
    implies the upper semicontinuity $($and thus continuity $)$ of the value function 
    $\nu \mapsto V_{\mathrm{MFC}}(\nu)$.
    This provides a complete well--posedness framework for the mean--field control problem 
    under minimal regularity assumptions.

    \medskip
    $(ii)$ It is worth noting that, under suitable convexity assumptions on the coefficients of the problem, one can recover an optimal \emph{strong} $($or \emph{closed--loop}$)$ control from an optimal \emph{relaxed} control.  
This reconstruction typically relies on measurable selection arguments and convexity of the control--dependence of the drift and cost functions $($see in particular  {\rm \cite[ Theorem 2.13.]{djete2023stackelbergmeanfieldgames}}, where similar difficulties arise due to the dependence on the law of the control$)$.  
To keep the exposition focused and accessible, we do not present this general recovery result here. 
Nevertheless, in a specific example discussed later in {\rm \Cref{prop_example2}}, we explicitly illustrate how such a procedure can be carried out in our framework.
\end{remark}

\subsection{The strong (\emph{open--loop}) formulation}
The formulations introduced so far are \emph{closed--loop}: controls depend (measurably) on the current state. 
We now consider an \emph{open--loop} setting in which controls depend only on the primitive noises (and initial data) of the system.
For this formulation, we assume that
\begin{assumption} \label{assum:Lip_b}
    The map $(t,x,m^1,m^2,a) \mapsto b(t,x,m^1,m^2,a)$ is Lipchitz in $x$ uniformly in $(t,m^1,m^2,a)$.
\end{assumption}
Let us denote by $\Ac^o_{\rm int}$ the set of Borel maps $\gamma: [0,T] \x \left(\R^d \x [0,1] \x \Cc^d \right)^2 \to A_{\rm int}$ and $\Ac^o_{\rm reg}$ the set of  Borel maps $\alpha: [0,T] \x \R^d \x [0,1] \x \Cc^d  \to A_{\rm reg}$.
We consider $(W^u)_{u \in [0,1]}$ a collection of Brownian motions.
Given $\eta \in \Pc_\nu$ and, $(\gamma,\alpha) \in \Ac^o_{\rm int} \x \Ac^o_{\rm reg}$, for a.e. $u \in [0,1]$, we define $ X^{\gamma,\alpha,u}$ by: $\Lc\bigl( X^{\gamma,\alpha,u}_0 \bigr)(\mathrm{d}x)\mathrm{d}u=\eta(\mathrm{d}x,\mathrm{d}u)$ and 
\begin{align*}
    \mathrm{d}X^{\gamma,\alpha,u}_s
    = b\!\left( s, X^{\gamma,\alpha,u}_s, M^{1,\gamma,u}_{\mu,s}, M^{2,\gamma,u}_{\mu,s}, \alpha^u_s \right) \mathrm{d}s + \sigma(s,X^{\gamma,\alpha,u}_s)\,\mathrm{d}W^u_s
\end{align*}
with $\alpha_s^u=\alpha \left(s,  X^{\gamma,\alpha,u}_0,\;u,\,W^u_{s \wedge \cdot} \right)$, $\Rr^{\gamma,\alpha,u}_s := \Lc\!\bigl(X^{\gamma,\alpha,v}_s,\,\Gr(u,v)\bigr)(\mathrm{d}x)\mathrm{d}v$, $\mu^{o,\gamma,\alpha}_s := \Lc\!\bigl(X^{\gamma,\alpha,u}_s\bigr)(\mathrm{d}x)\mathrm{d}u$ and for a.e. $\om$,
\begin{align*}
    M^{1,\gamma,u}_{\mu,s}(\om)
    &:=
     \int_0^1\Lc\!\left( \gamma\bigr(s, X^{\gamma,\alpha,u}_0(\om),u,W^u_{s \wedge \cdot}(\om), X^{\gamma,\alpha,v}_0,v,W^v_{s \wedge \cdot}   \bigl),\, X^{\gamma,\alpha,v}_s,\, \Gr(u,v) \, \right) \mathrm{d}v
     \\
     M^{2,\gamma,u}_{\mu,s}(\om)
    &:=
     \int_0^1\Lc\!\left( \gamma\bigr(s, X^{\gamma,\alpha,v}_0,v,W^v_{s \wedge \cdot}, X^{\gamma,\alpha,u}_0(\om),u,W^u_{s \wedge \cdot}(\om)   \bigl),\, X^{\gamma,\alpha,v}_s,\, \Gr(u,v) \, \right) \mathrm{d}v.
\end{align*}
The previous equations are well--defined under {\rm\Cref{assum:main1_MF_CI}} and \Cref{assum:Lip_b} by standard fixed point techniques (see for instance \cite{djete2019mckean}, \cite{decrescenzo2024meanfieldcontrolnonexchangeable} or \cite{cao2025probabilisticanalysisgraphonmean}). The {\it open--loop} formulation of mean--field control problem is then defined by 
\[
    V^o_{\mathrm{MFC}}(\nu)
    := \sup_{\eta \,\in\, \Pc_\nu} \sup_{\gamma \in \Ac^o_{\mathrm{int}},\, \alpha \in \Ac^o_{\rm reg}} J^o(\eta,\gamma,\alpha),
\]
with
\begin{align*}
    J^o(\eta,\gamma,\alpha)
    :=  \int_0^1\E \Bigg[ \int_0^T L \!\Big( s, X^{\gamma,\alpha,u}_s, M^{1,\gamma,u}_{\mu,s}, M^{2,\gamma,u}_{\mu,s}, \alpha_s^u \Big)\,\mathrm{d}s 
    + g\!\left( X^{\gamma,\alpha,u}_T, \Rr^{\gamma,\alpha,u}_T \right) \Bigg] \mathrm{d}u.
\end{align*}
We will now give an equivalence result between this {\it open--loop} formulation and our initial formulation. For this purpose, we set  
\begin{align*}
    \Lambda^{o,\gamma,\alpha}_t(\mathrm{d}x,\mathrm{d}u,\mathrm{d}r_1,\mathrm{d}r_2,\mathrm{d}a)\,\mathrm{d}t
    &:=
    \Lc\!\left(
        X^{\gamma,\alpha,u}_t,\,
        M^{1,\gamma,u}_{\mu,t},\,
        M^{2,\gamma,u}_{\mu,t},\,
        \alpha^u_t
    \right)
    (\mathrm{d}x,\mathrm{d}r_1,\mathrm{d}r_2,\mathrm{d}a)\,\mathrm{d}u\,\mathrm{d}t
\end{align*}
and
\begin{align*}
    \Pi^o(\nu)
    :=
    \Bigl\{
        \bigl( \mu^{o,\gamma,\alpha},\, \Lambda^{o,\gamma,\alpha} \bigr)
        \;:\;
        (\gamma,\alpha) \in \Ac^o_{\rm int} \times \Ac^o_{\rm reg}
    \Bigr\}.
\end{align*}
The next result shows that, under mild regularity assumptions, the \emph{open--loop} and \emph{closed--loop} formulations describe the same admissible laws and therefore yield identical value functions.

\begin{proposition} \label{prop:equiv_open}
    Let {\rm\Cref{assum:main1_MF_CI}} and {\rm \Cref{assum:Lip_b}} be true. For any $\nu \in \Pc(\R^d)$, we have $\Pi^o(\nu) \subset \Pib(\nu)$. Consequently,
    \begin{align*}
        V^o_{\mathrm{MFC}}(\nu)
        =
        V_{\mathrm{MFC}}(\nu).
    \end{align*}
\end{proposition}

\begin{remark}
    This result confirms that the formulation through the driving noise does not enlarge the set of admissible mean--field dynamics. 
    In particular, the optimization over open--loop controls leads to the same attainable laws and optimal value as in the closed--loop case.
    Hence, both viewpoints can be used interchangeably depending on analytical or numerical convenience.
\end{remark}

\subsection{The \emph{n--particle} formulation}
We now introduce the associated \emph{\(n\)--particle} formulation. Let \( n \ge 1 \). We define \( \mathcal{A}_{n,\rm reg} \) as the set of progressively Borel measurable maps \( \beta^n : [0,T] \times (\Cc^d)^n \to A_{\rm reg} \) and  \( \mathcal{A}_{n,\rm int} \) as the set of progressively Borel measurable maps \( \beta^n : [0,T] \times (\Cc^d)^n \to A_{\rm int} \).
Let \( (\xi^n_{ij})_{1 \le i,j \le n} \subset \Er \) be a given \( n \times n \) matrix. Let  $(W^i)_{1 \le i \le n}$ be a sequence of independent $\R^d$--valued $\H$--Brownian motions and the initial distribution be $\nu^n \in \Pc\bigr((\R^d)^n \bigl)$. Given interaction controls \( \boldsymbol{\gamma}^n := (\gamma^{n}_{ij})_{1 \le i,j \le n} \subset \mathcal{A}_{n,\rm int} \) and regular controls \( \boldsymbol{\alpha}^n := (\alpha^{1,n}, \dots, \alpha^{n,n}) \subset \mathcal{A}_{n,\rm reg} \), we define the interacting particle system \( \Xbb^n := (X^{1,n}, \dots, X^{n,n}) \) via: $\Lc(\Xbb^n_0)=\nu^n$ and
\begin{align} \label{eq:n_particle}
    \mathrm{d}X^i_t = b\left(t,X^i_t,M^{1,i,n}_{\gammab^n,t}, M^{2,i,n}_{\gammab^n,t},\alpha^{i,n}(t,\Xbb^n) \right)\mathrm{d}t + \sigma(t, X^i_t)\mathrm{d}W^i_t,
\end{align}
where the associated empirical measures are given by
\begin{align*}
    M^{1,i,n}_{\gammab^n,t} 
    = \frac{1}{n} \sum_{j=1}^n \delta_{\left( \gamma^n_{ij}(t, \Xbb^n),\, X^j_t,\,u^j_n,\, \xi^n_{ij} \right)},\qquad 
    M^{2,i,n}_{\gammab^n,t} 
    = \frac{1}{n} \sum_{j=1}^n \delta_{\left( \gamma^n_{ji}(t, \Xbb^n),\, X^j_t,\,u^j_n,\, \xi^n_{ij} \right)},
\end{align*}
for all \(u^i_n := \frac{i}{n}\) with \(i = 1, \dots, n\).
For each \( 1 \le i \le n \), the collection \( (\gamma^n_{ij})_{1 \le j \le n} \) represents the controls used by (or assigned to) player \( i \) to adjust their interactions with the rest of the population.

\medskip
The $n$--player optimization is defined by
\begin{align*}
    V_n(\nu^n):=\sup_{\gammab^n,\alphab^n} J_n(\nu^n,\gammab^n,\alphab^n),\;J_n(\nu^n,\gammab^n,\alphab^n):=\frac{1}{n} \sum_{i=1}^n\E \left[ \int_0^T L\left(t, X^i_t, M^{1,i,n}_{\gammab^n,t}, M^{2,i,n}_{\gammab^n,t}, \alpha^{i,n}(t, \Xbb) \right)\mathrm{d}t + g \left( X^i_T, \Rr^{i,n}_T \right)  \right]
\end{align*}
with $\Rr^{i,n}_t:=\frac{1}{n} \sum_{j=1}^n \delta_{\left( X^j_t ,\, \xi^n_{ij}  \right)}$.

\medskip
The step kernel associated with $(\xi^n_{ij})_{1 \le i,j \le n}$ is defined by the map \(\Gr^n : [0,1]^2 \to \Er\) as follows:
\begin{align*}
    \Gr^n(u,v) &:= \sum_{1 \le i,j \le n} \xi^n_{ij} \, \mathbf{1}_{(u^i_n - \frac{1}{n},\, u^i_n]}(u)\, \mathbf{1}_{(u^j_n - \frac{1}{n},\, u^j_n]}(v),
\end{align*}
for all \(u^i_n := \frac{i}{n}\) with \(i = 1, \dots, n\).
We assume that
\begin{align} \label{eq:cond_kernel}
    \lim_{n \to \infty} \| f \circ \Gr^n - f \circ \Gr \|_{\Box} = 0,
    \qquad \mbox{ for all Lipschitz map } f:\Er \to\R,
\end{align}
where the \emph{cut--norm} of a kernel $T:[0,1]^2 \to \R$ is defined by
\[
    \|T\|_{\Box} := \sup_{A,B \subset [0,1]} 
    \Bigg|\int_{A \times B} T(x,y)\,\mathrm{d}x\,\mathrm{d}y\Bigg|.
\]

\begin{remark}
    This assumption ensures that the sequence of step--kernels $\Gr^n$ converges to the limiting kernel $\Gr$ in the \emph{cut--norm} topology, when tested against bounded Lipschitz functions.  
Depending on the regularity of the coefficients $b$, $L$, and $g$, this requirement can in fact be relaxed.  
Notice that, when $\Er \subset \R^\ell$ for some $\ell \ge 1$, convergence in $\L^1$---that is,
\[
    \lim_{n \to \infty} \|\Gr^n - \Gr\|_{\L^1([0,1]^2)} = 0,
\]
implies convergence in the cut--norm sense stated in \eqref{eq:cond_kernel}.  
Hence, our assumption \eqref{eq:cond_kernel} is strictly weaker than the usual $\L^1$--type convergence commonly imposed in the literature on non--exchangeable $($or graphon--based$)$ mean--field systems $($see for instance {\rm \cite{cao2025probabilisticanalysisgraphonmean}}$)$. We also impose no continuity assumptions on the map $(u,v) \mapsto \Gr (u,v)$.
\end{remark}

Our objective now is to describe the connection between the {\it n}--player formulation and the mean--field control problem. For each $(\gammab^n,\alphab^n) \in (\Ac_{n,\,\rm int})^{n^2} \x (\Ac_{n,\,\rm reg})^n$, we define 
$$
    \mu^{n,\gammab^n,\alphab^n}_t:=\frac{1}{n} \sum_{i=1}^n \delta_{\left( X^i_t,\,u^i_n \right)}\quad\mbox{for all }t \in [0,T],
$$ 
with 
\begin{align*}
    \Lambda^{n,\gammab^n,\alphab^n}_t(\mathrm{d}x,\mathrm{d}u,\mathrm{d}r_1,\mathrm{d}r_2,\mathrm{d}a)\,\mathrm{d}t
    &:= \frac{1}{n} \sum_{i=1}^n \delta_{\left( X^i_t,\,u^i_n,\,M^{1,i,n}_{\gammab^n,\,t},\,M^{2,i,n}_{\gammab^n,\,t},\,\alpha^{i,n}(t,\Xbb^n) \right)}(\mathrm{d}x,\mathrm{d}u,\mathrm{d}r_1,\mathrm{d}r_2,\mathrm{d}a)\,\mathrm{d}t
\end{align*}
and
\begin{align*}
    \Pr^n
    :=
    \P \circ\left( \mu^{n,\gammab^n,\alphab^n},\, \Lambda^{n,\gammab^n,\alphab^n}  \right)^{-1}.
\end{align*}
For each $m \in \Pc(\R^d \x [0,1])$ and $u \in [0,1]$, we set  $\Rr(m,u):=\Lc(X^m, \Gr (u,U^m))$ where $(X^m,U^m)$ are random variables s.t. $m=\Lc(X^m,U^m)$.
\begin{theorem} \label{thm:from_n_to_limit}
    Let {\rm\Cref{assum:main1_MF_CI}} hold. 
    Assume that the sequence of initial laws $(\nu^n)_{n \ge 1} \subset \Pc\bigl((\R^d)^n\bigr)$ is s.t. the sequence
    \begin{align} \label{eq:cond_initial}
        \left(\Lc\left( \frac{1}{n} \sum_{j=1}^n \delta_{X^i_0}  \right) \right)_{n \ge 1}
    \end{align}
    is relatively compact for the weak topology. 
    Then, the sequence of laws
    \[
        \left(\Pr^n := \P \circ\!\left(\mu^{n,\gammab^n,\alphab^n},\, \Lambda^{n,\gammab^n,\alphab^n}\right)^{-1}\right)_{n \ge 1}
    \]
    is relatively compact for the weak topology.
    Moreover, every limit point $\Pr=\P \circ (\mu,\Lambda)^{-1}$ of a convergent subsequence $(\Pr^{n_k})_{k \ge 1}$ is supported on the set
    \[
        \bigcup_{\nu' \in \Pc(\R^d)} \Pib(\nu')
    \]
    and
    \begin{align*}
        \lim_{k \to \infty} J_{n_k}(\nu^{n_k},\gammab^{n_k},\alphab^{n_k}) = \E \left[ \int_0^T \langle L(t,\cdot),\Lambda_t \rangle \mathrm{d}t +\int_{\R^d \x [0,1]} g\left( x, \Rr(\mu_T,u) \right) \mu_T(\mathrm{d}x,\mathrm{d}u) \right].
    \end{align*} 
\end{theorem}

\begin{remark}
    $(i)$ This result provides the compactness foundation linking the finite $n$--particle formulation to the mean--field control problem. 
    In particular, any accumulation point of the sequence $(\Pr^n)_{n \ge 1}$ represents the law of a relaxed mean--field control satisfying the limiting McKean--Vlasov dynamics. 
    Hence, the theorem ensures that the mean--field formulation is the natural asymptotic limit of the $n$--player optimization problems, and that no loss of admissible dynamics occurs in the passage from the finite system to its continuum counterpart.

    \medskip
    $(ii)$ It is important to note that the assumptions on the initial distributions concern only the sequence specified in {\rm\eqref{eq:cond_initial}}.  
In particular, we neither require the random variables $(X^i_0)_{1 \le i \le n}$ to be independent nor assume that they share a common distribution.  
All convergence statements are understood in the weak topology, and it is unnecessary to work with Wasserstein distances $\Wc_p$ for $p>1$, since the coefficients are uniformly bounded and the sets $A_{\rm int}$, $A_{\rm reg}$, and $\Er$ are compact.  
While our analysis could be extended to unbounded settings by suitably adapting the estimates and tightness arguments, we restrict attention to this bounded framework for clarity and simplicity of exposition.

\end{remark}

Next, let us provide an approximation of any McKean--Vlasov limit through a $n$--particle system. Let $(\gamma, \alpha) \in \Ac_{\rm int} \x \Ac_{\rm reg}$. We consider a sequence of Lipschitz map $(\gamma^\ell,\alpha^\ell)_{\ell \ge 1}$ s.t. $\Lim_{\ell \to \infty} (\gamma^\ell, \alpha^\ell)=(\gamma,\alpha)$ a.e.  For any $(t,x_1,\cdots,x_n) \in [0,T] \x (\Cc^d)^n$, we set
\begin{align*}
    \gamma^{\ell,n}_{ij}(t,x_1,\dots,x_n) &:= \gamma^\ell(t, x_i(t), u^i_n, x_j(t), u^j_n), \\
    \alpha^{\ell,i,n}(t,x_1,\dots,x_n) &:= \alpha^\ell(t, x_i(t), u^i_n).
\end{align*}

\begin{theorem} \label{thm:strong_to_n}
    Under {\rm\Cref{assum:main1_MF_CI}}, if 
    \begin{align*}
        \Lim_{\ell \to \infty}\Lim_{n \to \infty} \Lc \bigl( \mu^{n,\gammab^{\ell,n},\alphab^{\ell,n}}_0 \bigr)=\delta_{\eta}\quad\mbox{in }\Wc_p\mbox{ for some }\eta \in \Pc(\R^d \x [0,1])\mbox{ and } p \in \{ 0\} \cup [1,\infty),
    \end{align*}
    we have
    \begin{align*}
        \Lim_{\ell \to \infty}\Lim_{n \to \infty} \Lc \bigl( \mu^{n,\gammab^{\ell,n},\alphab^{\ell,n}},\,\Lambda^{n,\gammab^{\ell,n},\alphab^{\ell,n}} \bigr)=\delta_{\left( \mu^{\gamma,\alpha},\,\Lambda^{\gamma,\alpha} \right)}\quad\mbox{in }\Wc_p\quad \mbox{and}\quad \Lim_{\ell \to \infty}\Lim_{n \to \infty} J_n(\nu^n,\gammab^{\ell,n},\alphab^{\ell,n})=J(\eta,\gamma,\alpha).
    \end{align*}
\end{theorem}
\begin{remark}
    The theorem establishes a \emph{propagation of chaos} property adapted to the controlled McKean--Vlasov setting: 
    under consistent initialization, the empirical laws of the $n$--particle system asymptotically behave as independent copies of the mean--field limit driven by the same control pair $(\gamma,\alpha)$.
\end{remark}

Combining the previous convergence theorem with the structural equivalence result of {\rm\Cref{thm:equivalence_relaxed_strong}}, 
    the next proposition shows that the mean--field control value $V_{\mathrm{MFC}}(\nu)$ accurately captures the asymptotic performance of the optimal $n$--player systems. 
    Hence, the mean--field formulation provides a consistent and rigorous limit theory for large controlled populations.

\begin{proposition} \label{prop:conv_value_function}
    Under {\rm\Cref{assum:main1_MF_CI}}, assume that the sequence of initial laws $(\nu^n)_{n \ge 1}$ satisfies
    \[
        \Lim_{n \to \infty}\Lc\left( \frac{1}{n} \sum_{j=1}^n \delta_{X^i_0} \right)
        = \delta_{\nu}
        \quad \text{in } \Wc_p\mbox{ for some }p \in \{ 0\} \cup [1,\infty).
    \]
    Then, the sequence of $n$--player value functions converges to the mean--field value:
    \[
        \lim_{n \to \infty} V_n(\nu^n)
        = V_{\mathrm{MFC}}(\nu).
    \]
\end{proposition}
An immediate consequence of {\rm\Cref{thm:equivalence_relaxed_strong}}, {\rm\Cref{thm:from_n_to_limit}}, {\rm\Cref{thm:strong_to_n}}, and {\rm\Cref{prop:conv_value_function}} is that the optimal controls of the finite--player problems $V_n(\nu^n)$ asymptotically coincide with those of the mean--field control problem $V_{\rm MFC}(\nu)$.  
This correspondence between the finite--dimensional and the mean--field formulations is summarized in the following corollary.

\begin{corollary} \label{cor_contruction_cong_optimal_control}
    Let us stay in the context of {\rm\Cref{prop:conv_value_function}}.  
    Suppose $(\gammabb,\alphabb)$ is an optimal relaxed control for $V_{\mathrm{MFC}}(\nu)$.  
    Then, there exists a sequence of $n$--player controls 
    \[
        (\gammab^n,\alphab^n) \in \Ac_{n,{\rm int}}^{n^2} \times \Ac_{n,{\rm reg}}^n,
        \qquad n \ge 1,
    \]
    such that, for each $n$, the pair $(\gammab^n,\alphab^n)$ is $\varepsilon_n$--optimal for $V_n(\nu^n)$ and satisfies
    \begin{align*}
        \lim_{n \to \infty} J_n(\nu^n,\gammab^n,\alphab^n) = V_{\mathrm{MFC}}(\nu),
        \qquad 
        \varepsilon_n := \big|\,V_n(\nu^n) - V_{\mathrm{MFC}}(\nu)\,\big|.
    \end{align*}
    In particular, if the relaxed optimal control $(\gammabb,\alphabb)$ happens to be \emph{strong} $($i.e., closed--loop$)$, 
    the approximating sequence $(\gammab^n,\alphab^n)$ can be explicitly constructed as described in {\rm\Cref{thm:strong_to_n}}.

    \medskip
    Conversely, consider a sequence of controls 
    \[
        (\gammab^n,\alphab^n) \in \Ac_{n,{\rm int}}^{n^2} \times \Ac_{n,{\rm reg}}^n, \qquad n \ge 1,
    \]
    such that each $(\gammab^n,\alphab^n)$ is $\varepsilon_n$--optimal for $V_n(\nu^n)$, with $\varepsilon_n \to 0$ as $n \to \infty$.  
    If there exists a convergent subsequence for which
    \[
        \P \circ \bigl(\mu^{n,\gammab^n,\alphab^n},\, \Lambda^{n,\gammab^n,\alphab^n}\bigr)^{-1}
        \;\Longrightarrow\;
        \delta_{(\mu^{\gammabb,\alphabb},\,\Lambdab^{\gammabb,\alphabb})},
    \]
    then the limiting control $(\gammabb,\alphabb)$ is optimal for the mean--field control problem $V_{\mathrm{MFC}}(\nu)$.
\end{corollary}

\begin{remark}
    These results rigorously establish the convergence of the finite--player cooperative optimization problem toward its mean--field counterpart.  
In other words, as the number of agents \(n\) becomes large, the optimal performance and control strategies of the \(n\)--player system are asymptotically captured by those of the mean--field control problem with controlled interactions.

\medskip
This convergence has two major implications. 
First, it provides a solid theoretical foundation for using the mean--field control framework as a tractable approximation of large cooperative systems: solving the mean--field problem yields an asymptotically optimal strategy for the finite system.  
Second, it confirms the internal consistency of our extended framework with interaction controls.  
Despite the additional layer of complexity introduced by allowing agents $($ or a planner $)$ to control the structure of their interactions, the collective behavior remains stable in the large--population limit and converges to a well--defined mean--field model.

\medskip
From an applied standpoint, this result ensures that policies or strategies designed at the mean--field level remain meaningful and near--optimal in large but finite systems—such as financial networks, communication infrastructures, or large economic or social systems—where interaction patterns can be influenced or strategically adjusted.

\end{remark}

\section{Examples and discussions} \label{sec:examples}

\paragraph*{Example 1: A bang--bang situation.}

We begin with a minimal example illustrating the structure and interpretation of the controlled interaction framework introduced above.  
To keep the exposition simple, we work in the one--dimensional setting $d=1$ and consider $A_{\rm int} = [0,1], 
\, L \equiv 0,
\, \sigma(t,x) \equiv 1.$
The drift of the controlled state dynamics is defined by
\[
b(t,x,m^1,m^2,a)
:=
\int_{A_{\rm int} \x \R}
    z\,\Phi\!\left(t,x,y,\,m^1(A_{\rm int},\mathrm{d}y',\Er)\right)
    m^1(\mathrm{d}z,\mathrm{d}y,\Er),
\]
and the terminal reward takes the form
\[
g(x,\overline{m}) = G\!\left(\overline{m}(\mathrm{d}y',\Er)\right),
\]
where $\Phi:[0,T]\x\R\x\R\x\Pc(\R)\to\R$ and $G:\Pc(\R)\to\R$ are bounded Borel maps.  
In this specification, the drift $b$ represents the aggregate effect of pairwise interactions modulated by the sign and magnitude of $\Phi$, while $G$ encodes the terminal evaluation of the overall distribution of states.

\medskip
The associated mean--field control problem thus reads
\[
    V_{\mathrm{MFC}}(\nu)
    :=
    \sup_{\gamma \in \Ac_{\mathrm{int}}} 
    \widehat{J}(\gamma),
    \qquad
    \widehat{J}(\gamma)
    :=
    \E\!\left[ G\!\left( \mu_T^{\gamma} \right) \right],
\]
where the controlled state process satisfies
\[
    \mathrm{d}X_t^{\gamma}
    =
    \left( 
        \int_{\R} 
        \gamma(t,X_t^{\gamma},y)\,
        \Phi\!\left(t,X_t^{\gamma},y,\mu_t^{\gamma}\right)
        \mu_t^{\gamma}(\mathrm{d}y)
    \right)
    \mathrm{d}t
    + \mathrm{d}W_t,
    \qquad 
    \mu_t^{\gamma} = \Lc(X_t^{\gamma}),
    \qquad
    \Lc(X_0^{\gamma}) = \nu.
\]

\medskip
We denote by
\[
    \Phih(t,x,m) := \int_{\R} \Phi(t,x,y,m)^{+}\,m(\mathrm{d}y)
\]
the positive part of the interaction kernel averaged with respect to the population law.

\medskip
Given a map $U: \Pc(\R) \to \R$. we will say that a Borel map $F:\Pc(\R) \x \R \to \R$ is a linear functional derivative of $U$ if: for each $m$ and $m'$, we have $ \int_0^1 \int_\R \left|F\left(e\; m + (1-e)\; m',x\right) \right| \left(m+m' \right)(\mathrm{d}x)\; \mathrm{d}e< \infty$ and$ U(m)-U(m')=\int_0^1 \int_\R F\left(e\; m + (1-e)\; m',x\right) \left(m-m' \right)(\mathrm{d}x) \mathrm{d}e.$
    We will denote $F$ by $\delta_m U$. Since $\delta_m U$ is defined up to a constant, we use the convention, $\int_\R \delta_m U(m,x)m(\mathrm{d}x)=0$ whenever $\int_\R |\delta_m U(m,x)|m(\mathrm{d}x) < \infty$.

\medskip
We assume throughout that $\Phi(t,x,y,m)$ is Lipschitz in $(x,y,m)$ uniformly in $t$, admits a linear functional derivative $\delta_m\Phi(t,x,y,m)(z)$ differentiable in $z$, and that the functional $G$ possesses a linear derivative $\delta_m G(m)(x)$ differentiable in $x$.

\begin{proposition}
    \label{prop_example1}
    Suppose that for all $(t,x,y,m)$,
    \[
        \partial_x \delta_m G(m)(x) \ge 0,
        \qquad
        \partial_y \Phi(t,x,y,m)(z) \ge 0,
        \qquad
        \partial_z \delta_m \Phi(t,x,y,m)(z) \ge 0.
    \]
    Then the optimal interaction control in $V_{\mathrm{MFC}}(\nu)$ is given explicitly by the \emph{bang--bang} rule
    \[
        \widehat{\gamma}(t,x,y)
        =
        \mathbf{1}_{\{\Phi(t,x,y,\mu^\star_t) \ge 0\}},
    \]
    where the optimal state process satisfies
    \[
        \mathrm{d}Y_t
        =
        \Phih(t,Y_t,\mu^\star_t)\,\mathrm{d}t
        + \mathrm{d}W_t,
        \qquad 
        \mu_t^{\star} = \Lc(Y_t),
        \qquad
        \Lc(Y_0) = \nu.
    \]
\end{proposition}

\begin{remark}
    $(i)$ The above result shows that, under suitable monotonicity and positivity conditions, the optimal interaction policy exhibits a \emph{threshold} $($or bang--bang $)$ structure:  
    each agent/particle interacts with another only when the contribution of the interaction kernel $\Phi$ to the collective drift is nonnegative.  
    In other words, agents/particles optimally ``activate'' connections that enhance the collective performance and ``deactivate'' those that are detrimental.  
    The resulting dynamics are thus driven by the regions of the state space where interactions are mutually beneficial.  
    This simple form provides a clear intuition on how structural properties of the mean--field functional $G$ and the kernel $\Phi$ translate into interpretable control laws, even in potentially high--dimensional systems.

    \medskip
    $(ii)$ The proof of {\rm \Cref{prop_example1}} is based on the $n$--player approximation established in {\rm \Cref{prop:conv_value_function}}. 
Alternatively, the result could also be derived directly by using the framework of backward stochastic differential equations on the Wasserstein space developed in {\rm \cite{djete2025notionbsdewassersteinspace}}. 
For the sake of completeness and to keep the exposition self--contained, we adopt here the $n$--player approach.

\end{remark}

\paragraph*{Example 2: A simple model of interdependent behavior on social media.} \label{para_example2} 

The phenomenon of social media provides a natural and intuitive setting in which individuals strategically choose how to interact with others.  
Each user decides whom to follow, whose content to engage with, and how actively to maintain those connections—while simultaneously being affected by the attention and influence they receive from others.  
This interplay between outgoing and incoming connections is inherently asymmetric, and serves as a canonical example of a system where agents \emph{control their interaction structure} rather than merely their state.

\medskip
Let $n \ge 1$, and consider a population of $n$ users represented by the state processes $\Xbb^n = (X^{1,n}, \dots, X^{n,n})$.  
For each user $i$, the process $X^{i,n}_t$ represents a quantitative measure of their visibility or reputation in the network—such as the number of followers, the average engagement with their posts, or a broader score of social influence.  
The evolution of these variables captures the dynamics of online attention and is described, for $1 \le i \le n$, by the system
\begin{align*}
    \mathrm{d}X^{i,n}_t
    &=
    \Bigg(
        \frac{1}{n} \sum_{j=1}^n 
        \gamma^n_{ij}(t,\Xbb^n)\,b_1\!\left(t,\, \xi^n_{ij},\, X^{j,n}_t \right)
        +
        \frac{1}{n} \sum_{j=1}^n 
        \gamma^n_{ji}(t,\Xbb^n)\,b_2\!\left(t,\, \xi^n_{ij},\, X^{j,n}_t \right)
    \Bigg)\mathrm{d}t
    + \sigma\,\mathrm{d}W^i_t.
\end{align*}

\medskip
The interpretation is straightforward yet rich.  
Each function $\gamma^n_{ij}$ represents the decision of user $i$ about the \emph{intensity or quality of their connection} with user $j$—for instance, how much attention or engagement $i$ devotes to $j$.  
This decision enters the dynamics through the first term, modulated by the weight $\xi^n_{ij}$, which encodes the preexisting affinity or structural link between $i$ and $j$ (for example, shared interests or proximity in a latent social graph).  
Conversely, the second term captures the reverse influence: how the choices of other users toward $i$, namely $(\gamma^n_{ji})_{1 \le j \le n}$, affect $i$’s visibility or reputation.  
In social media, one can follow others without being followed back—so the impact of one’s own connection choices and that of others’ choices toward you are not necessarily balanced.  
This asymmetry is an essential aspect of modern social platforms.

\medskip
A social planner—think of a platform designer or regulator—seeks to promote the overall health and engagement of the system by coordinating or incentivizing the connection patterns.  
The planner’s objective is to maximize the average welfare of the users, defined as
\begin{align*}
    \frac{1}{n} \sum_{i=1}^n
    \E\!\left[
        \int_0^T 
        \frac{1}{n}\sum_{j=1}^n 
        L\!\left(t, \gamma^n_{ij}(t,\Xbb^n),\, \xi^n_{ij},\, X^{j,n}_t \right)\mathrm{d}t
        + g\!\left(X^{i,n}_T\right)
    \right].
\end{align*}
The instantaneous reward $L$ measures the benefit of the chosen connections—such as increased exposure, mutual engagement, or advertising impact—while $g$ quantifies the terminal visibility or influence level of each user.  

\medskip
As the population becomes very large, the system admits a mean–field limit in which each user interacts with the overall distribution of others rather than with finitely many individuals.  
The corresponding mean–field control problem is
\[
    V_{\mathrm{MFC}}(\nu)
    :=
    \sup_{\gamma \in \Ac_{\mathrm{int}}}
    \widehat{J}(\gamma),
    \qquad
    \widehat{J}(\gamma)
    :=
    \E\!\left[
        \int_0^T \int_{(\R \x [0,1])^2}
        L\!\left(\gamma(t,x,x'),\,\Gr(u,u'),\,x'\right)
        \mu^\gamma_t(\mathrm{d}x,\mathrm{d}u)
        \mu^\gamma_t(\mathrm{d}x',\mathrm{d}u')\,\mathrm{d}t
        + g\!\left(X^\gamma_T\right)
    \right],
\]
where $\Gr:[0,1]^2\to\Er$ denotes the limiting interaction kernel representing the large–scale structure of the platform.  
The controlled state process satisfies $\Lc(X_0^\gamma)=\nu$ and evolves as
\[
    \mathrm{d}X_t^{\gamma}
    =
    \Bigg(
        \int_{\R\x[0,1]}
            \gamma(t,X_t^{\gamma},x)\,b_1\!\left(t,\,\Gr(U,u),\,x\right)
            +
            \gamma(t,x,X_t^{\gamma})\,b_2\!\left(t,\,\Gr(U,u),\,x\right)
        \mu_t^{\gamma}(\mathrm{d}x,\mathrm{d}u)
    \Bigg)\mathrm{d}t
    + \sigma\,\mathrm{d}W_t,
    \qquad
    \mu_t^{\gamma}=\Lc(X_t^{\gamma},U).
\]

\medskip
This formulation highlights how the control $\gamma$ now regulates the \emph{interaction intensity} between typical pairs of agents in the mean–field limit.  
The kernel $\Gr$ captures the exogenous structure of potential links (for instance, similarity or recommendation weights), while $\gamma$ describes how those potential links are exploited in equilibrium or by a social planner.

\medskip
In addition to the standing assumptions of \Cref{sec:main}, we impose the convexity of $A_{\rm int} \subset \R$ and for $(t,x,r)\in [0,T]\x\R\x\Er$, the map $e \mapsto L(t,e,r,x)$ is concave 
which guaranties the recovery of strong \emph{closed--loop} control but relaxed one, enabling the existence of an optimal \emph{closed--loop} control.

\begin{proposition} \label{prop_example2}
    Under the above assumptions, there exists an optimal \emph{closed--loop} control for the mean--field social network model.  
    Moreover, this optimal control can be used to construct an asymptotically optimal sequence of controls for the finite $n$--agent systems, as stated in {\rm\Cref{cor_contruction_cong_optimal_control}}.
\end{proposition}

\medskip
This example illustrates how mean–field control with controlled interactions can model strategic behavior in online social environments.  
In such systems, visibility and influence are both consequences and determinants of connection decisions, leading to a feedback loop between individual strategies and collective dynamics.  
The mean–field perspective provides a principled way to analyze and optimize such large, asymmetric networks—bridging ideas from social interaction theory, stochastic control, and network science.

\section{Proofs of the main results}  \label{sec:proof}

The remainder of the paper is devoted to the proofs of the main results stated above. We first establish several auxiliary results that will play a key role in the analysis. In the next section, we introduce an alternative formulation of the relaxed problem on a canonical space, which provides a convenient framework for handling the sequences of probability measures involved in the proofs.

\subsection{A relaxed formulation} \label{sec:relaxed}

\paragraph*{Admissible pair} 
Before proceeding, we clarify what constitutes an admissible pair of controls by setting out the minimal structural and measurability assumptions required in our framework.
For any Polish space \(E\), we recall that we denote by \(\M(E)\) the set of measures \(m(\mathrm{d}e,\mathrm{d}t)\) such that $\tfrac{1}{T} m(\mathrm{d}e,\mathrm{d}t) \in \Pc(E \times [0,T]) 
    \quad \text{and} \quad 
    m(E \times \mathrm{d}t) = \mathrm{d}t .$ 
We set 
\[
    \Hc := \Pc(A_{\mathrm{int}} \times \R^d \x [0,1] \x \Er), 
    \qquad 
    \Rc := \Hc \times \Hc \times A_{\mathrm{reg}} .
\]

\medskip

A pair \((\mu, \Lambda) \in \Cc^d_{\mathrm{pr}} \times \M( \R^d \times [0,1] \times \Rc)\) is said to be \emph{admissible} if the following conditions hold:  
\[
    \mu_t = \Lc(X_t,U), 
    \qquad 
    \Lambda_t(\mathrm{d}x,\mathrm{d}u,\mathrm{d}r_1,\mathrm{d}r_2,\mathrm{d}a) 
    = \E\!\left[ \delta_{(X_t,U)}(\mathrm{d}x,\mathrm{d}u)\, \Gamma_t(\mathrm{d}r_1,\mathrm{d}r_2,\mathrm{d}a) \right] \quad \text{a.e. } t,
\] 
where on the probability space \((\Om,\F,\P)\):  

\begin{itemize}
    \item The random variables \( (X_0, U) \) and \( W \) are \( \P \)--independent;
    \item The process \( W \) is a \( (\P, \F) \)--Brownian motion;
    \item The random variable \( U \) is uniformly distributed on \([0,1]\), i.e. \(\Lc(U) = \mathrm{Unif}([0,1])\);
    \item The process \( X \) satisfies the stochastic differential equation:
    \[
        \mathrm{d}X_t 
        = \int_{\Rc} b(t, X_t, r_1, r_2, a)\,\Gamma_t(\mathrm{d}r_1, \mathrm{d}r_2, \mathrm{d}a)\,\mathrm{d}t 
        + \sigma(t, X_t)\,\mathrm{d}W_t .
    \]
\end{itemize}

\paragraph*{The relaxation of the set of controls of the interactions}
The kernel $\Gr$ is fixed.  
We denote by
\[
    \Uc := \Big\{\, \gamma : (\R^d \x [0,1])^2 \to A_{\rm int}\;\text{Borel map}\,\Big\},
    \qquad
    \overline{\Uc} := \Big\{\, \gammabb : \big((\R^d \x [0,1]^2)^2 \x [0,1]\big) \to A_{\rm int}\;\text{Borel map}\,\Big\}.
\]

\medskip
Let us consider random variables
\[
    (X,U) \;\perp\; V \;\perp\; \Vb,
    \qquad U,\,V,\,\Vb \sim {\rm Unif}([0,1]),
\]
all mutually independent. We denote by
\[
    \mu := \Lc(X,U) \in \Pc(\R^d \x [0,1]).
\]

\medskip
For any $\gamma \in \Uc$, we associate the probability kernels
\begin{align*}
    N_\mu^{1,\gamma}(x,u) 
    &:= \Lc\Bigl(\gamma(x,u,X,U),\,X,\,U,\,\Gr(u,U)\Bigr), 
    \\
    N_\mu^{2,\gamma}(x,u) 
    &:= \Lc\Bigl(\gamma(X,U,x,u),\,X,\,U,\,\Gr(u,U)\Bigr),
\end{align*}
which encode respectively the law induced by $\gamma$ when $(x,u)$ interacts with $(X,U)$ in the first or second position.

\medskip
Similarly, for any $\gammabb \in \Ucb$, we define
\begin{align*}
    \Nb^{1,\gammabb}_\mu(x,u,v) 
    &:= \Lc\Bigl(\gammabb(x,u,v,X,U,V,\Vb),\,X,\,U,\,\Gr(u,U)\Bigr), 
    \\
    \Nb^{2,\gammabb}_\mu(x,u,v) 
    &:= \Lc\Bigl(\gammabb(X,U,V,x,u,v,\Vb),\,X,\,U,\,\Gr(u,U)\Bigr),
\end{align*}
where the auxiliary uniforms $(V,\Vb)$ provide additional randomization in the interaction.

\medskip
With these notations, we introduce the subsets
\[
    \Mc := \Bigl\{\, 
        \Lc\bigl(X,\,U,\,N^{1,\gamma}_\mu(X,U),\,N^{2,\gamma}_\mu(X,U)\bigr)
        \;:\; \gamma \in \Uc,\;\Lc(X) \in \Pc(\R^d) 
    \Bigr\}
\]
and
\[
    \Mcb := \Bigl\{\, 
        \Lc\bigl(X,\,U,\,\Nb^{1,\gammabb}_\mu(X,U,V),\,\Nb^{2,\gammabb}_\mu(X,U,V)\bigr)
        \;:\; \gammabb \in \Ucb,\;\Lc(X) \in \Pc(\R^d) 
    \Bigr\},
\]
which are subsets of
\[
    \Pc\Bigl(\R^d \x [0,1] \x \Pc(A_{\rm int} \x \R^d \x [0,1] \x \Er)^2\Bigr).
\]
\medskip
We now provide an alternative characterization of the set $\Mcb$, which will be particularly useful 
when dealing with sequences in $\Mcb$. To this end, on the probability space $(\Om,\F,\P)$ and consider 
a tuple of random variables
\[
    \Gamma \;=\; \Big( \gamma,\; Z=(X,U),\; N=(N^1,N^2),\; \Zh=(\Xh,\Uh),\; \Nh=(\Nh^1,\Nh^2) \Big).
\]

\noindent
We say that $\Gamma$ is \emph{compatible} if the following conditions are satisfied:
\begin{itemize}
    \item[(i)] The pairs $(Z,N)$ and $(\Zh,\Nh)$ are independent and identically distributed, that is
    \[
        \Lc\bigl(Z,\,N,\,\Zh,\,\Nh\bigr) \;=\; \Lc\bigl(Z,\,N\bigr) \otimes \Lc\bigl(Z,\,N\bigr)
    \]
    and $\Lc(U)={\rm Unif}([0,1])$.

    \item[(ii)] For $\P$--a.e. realization of $\gamma \in A_{\rm int}$ and $Z=(X,U)\in \R^d\x[0,1]$, the kernels $N^1$ and $\Nh^2$ are consistent with $\gamma$ in the sense that
    \begin{align*}
        N^1 &= \Lc\bigl(\gamma,\,\Zh,\,\Gr(U,\Uh)\;\mid\;Z,N\bigr),
        \\
        \Nh^2 &= \Lc\bigl(\gamma,\,Z,\,\Gr(\Uh,U)\;\mid\;\Zh,\Nh\bigr).
    \end{align*}
\end{itemize}

\noindent
Given this notion, we define the auxiliary set
\[
    \Mcb_{\rm aux} := \Bigl\{\, \P \circ (Z,N)^{-1}\;:\;\Gamma \text{ is compatible} \Bigr\}.
\]

\begin{proposition} \label{prop:reformulation_of_relaxed}
    The two constructions of randomized interaction laws coincide, i.e.
    \[
        \Mcb = \Mcb_{\rm aux}.
    \]
\end{proposition}

\noindent
In words, $\Mcb_{\rm aux}$ characterizes elements of $\Mcb$ as laws of $(Z,N)$ that can be extended 
to an admissible system $(Z,N,\Zh,\Nh)$ consisting of two independent and identically distributed copies, 
linked consistently through the common action $\gamma$ and the coupling mechanism $\Gr$. 
This formulation will be particularly useful when manipulating sequences of elements in $\Mcb$, 
since admissibility provides a symmetric extension that is stable under weak limits.

\begin{proof}
Let $M \in \Mcb$ be associated to $\gammabb \in \Ucb$. We take the random variables $\left( X',U',V'\right)$ independent of $(X,U,V,\Vb)$ s.t. $\Lc(X',U',V)=\Lc(X,U,V)$. Remember that $\Lc(X,U)=\mu$.  We define 
\begin{align*}
    \Gamma^M
    =
    \left( \gammabb (X,U,V,X',U',V',\Vb),\,(X,U),\, \Nb^{1,\gammabb}_\mu(X,U,V),\,\Nb^{2,\gammabb}_\mu(X,U,V),\,(X',U'),\,\Nb^{1,\gammabb}_\mu(X',U',V'),\,\Nb^{2,\gammabb}_\mu(X',U',V') \right).
\end{align*}
Simple computations allow us to check that $\Gamma^M$ is \emph{compatible} in the sense of the definition of $\Mcb_{\rm aux}$. We can deduce that $M \in \Mcb_{\rm aux}$.

\medskip
Now, let $M \in \Mcb_{\rm aux}$ be associated to a \emph{compatible} $\Gamma$.
By standard measurable arguments, there exist a Borel map 
\[
    \gammabb: \left(\Hc^2 \times \R^d \x [0,1] \right)^2  \times [0,1] \to A_{\rm int}
\] 
and a uniform random variable $\Vb$, independent of $(Z,\,N,\, \Zh,\,  \Nh)$, such that
\begin{align*}
    \Lc \bigl( \gamma,\, N,\, Z,\, \Nh,\, \Zh \bigr)
    = \Lc \Bigl( \gammabb \bigl(N,\, Z,\, \Nh,\, \Zh,\, \Vb \bigr),\, N,\, Z,\, \Nh,\, \Zh \Bigr).
\end{align*}

Similarly, there exist a Borel map 
\[
    \Nb=\bigl(\Nb^1,\,\Nb^2 \bigr): \R^d \x [0,1] \times [0,1] \to \Hc^2
\] 
and a uniform random variable $V$, independent of $(Z,\, \Zh)$, such that
\begin{align*}
    \Lc \bigl( N,\, Z \bigr) = \Lc \bigl( \Nb(Z,\, V),\, Z \bigr).
\end{align*}

\noindent
Here, the uniform random variables $\Vb$ and $V$ serve as auxiliary randomness that allows us to represent distributions in terms of measurable functions of the underlying variables. In other words, they “encode” the randomness needed to realize the laws as Borel functions.

\medskip
We assume that $V$ and $\Vb$ are independent.  
In addition, let $\Vh$ be an independent copy of $V$, also independent of $\Vb$. We then define
\begin{align*}
    \Gammabb(Z,\, V,\, \Zh,\, \Vh,\, \Ub) 
    := \gammabb \Bigl( \Nb(Z, V),\, Z,\, \Nb(\Zh, \Vh),\, \Zh,\, \Vb \Bigr).
\end{align*}
It immediately follows from the point $(i)$ verified by a \emph{compatible} $\Gamma$ that
\begin{align*}
    \Lc \Bigl( \Gammabb(Z,\, V,\, \Zh,\, \Vh,\, \Vb),\, \Nb(Z, V),\, Z,\, \Nb(\Zh, \Vh),\, \Zh \Bigr) 
    = \Lc \bigl( \gamma,\, N,\, Z,\, \Nh,\, \Zh \bigr).
\end{align*}

\noindent
The function $\Gammabb$ provides an explicit measurable representation of $\gamma$ in terms of the auxiliary randomness $V$, $\Vh$, and $\Vb$. In particular, it shows that the conditional distributions $\Nb(Z,V)$ and $\Nb(\Zh,\Vh)$, together with the independent uniform variable $\Vb$, are sufficient to reconstruct the law of $\gamma$.

\medskip
By combining the law identity just proved together with the point $(ii)$, for any smooth maps $g$ and $\varphi$, we have
\begin{align*}
    &\E \Bigl[ \langle g, \Nb^1(Z,V) \rangle \, \varphi \bigl( \Nb(Z,V),\, Z \bigr) \Bigr] 
    = \E \Bigl[ \langle g, N^1 \rangle \, \varphi \bigl( N,\, Z \bigr) \Bigr] 
    = \E \Bigl[ g \bigl( \gamma, \Zh,\,\Gr(U,\Uh) \bigr) \, \varphi \bigl( N,\, Z \bigr) \Bigr] 
    \\
    &= \E \Bigl[ g \bigl( \Gamma(Z,\, V,\, \Zh,\, \Vh,\, \Vb),\, \Zh,\,\Gr(U,\Uh) \bigr) \, \varphi \bigl( \Nb(Z,V),\, Z \bigr) \Bigr] 
    \\
    &= \E \Bigl[ \E \Bigl[ g \bigl( \Gamma(Z,\, V,\, \Zh,\, \Vh,\, \Vb),\, \Zh,\,\Gr(U,\Uh) \bigr) \mid Z, V \Bigr] \, \varphi \bigl( \Nb(Z,V),\, Z \bigr) \Bigr].
\end{align*}
Since this holds for arbitrary $g$ and $\varphi$, we deduce that $\P$--a.e.
\begin{align*}
    \Nb^1(Z,V) = \Lc \Bigl( \Gammabb(Z,\, V,\, \Zh,\, \Vh,\, \Vb),\, \Zh,\,\Gr(U,\Uh) \;\big|\; Z, V \Bigr).
\end{align*}
By the same argument, we also obtain
\begin{align*}
    \Nb^2(\Zh,\Vh) = \Lc \Bigl( \Gammabb(Z,\, V,\, \Zh,\, \Vh,\, \Vb),\, Z,\,\Gr(\Uh,U) \;\big|\; \Zh, \Vh \Bigr).
\end{align*}

\noindent
These identities show that the conditional laws $\Nb^1(Z,V)$ and $\Nb^2(\Zh,\Vh)$ can be explicitly represented using the measurable map $\Gammabb$ and the auxiliary randomness $V$, $\Vh$, and $\Vb$. In other words, $\Gammabb$ realizes the conditional distributions of $\gamma$ given $(Z,N)$ and $(\Zh,\Nh)$, respectively.

\medskip
By the law identity of point $(i)$, we deduce that
\begin{align*}
    \Nb^2(Z,V) &= \Lc \Bigl( \Gamma(\Zh,\, \Vh,\, Z,\, V,\, \Vb),\, \Zh,\,\Gr(U,\Uh) \;\big|\; Z, V \Bigr),
    \\
    \Nb^1(\Zh,\Vh) &= \Lc \Bigl( \Gamma(\Zh,\, \Vh,\, Z,\, V,\, \Vb),\, Z,\,\Gr(\Uh,U) \;\big|\; \Zh, \Vh \Bigr).
\end{align*}
These results lead to deduce that $M \in \Mcb$ associated to $\Gammabb$.

\end{proof}

\paragraph*{The set of relaxed controls} 
Building upon the relaxation of interaction controls introduced earlier, we now formalize the notion of \emph{relaxed controls}, which will play a central role in the forthcoming analysis and proofs.  
This concept provides a flexible probabilistic framework that extends the admissible control class by allowing convex combinations of admissible interaction structures, while preserving the marginal consistency in the state–label variables. 
\begin{definition}
    A pair $\left( \mu,\Lambda \right)$ is called a \emph{relaxed control} if it is admissible and satisfies, for almost every $t \in [0,T]$,
    \[
        \Lambda_t(\mathrm{d}x,\mathrm{d}u,\mathrm{d}r_1,\mathrm{d}r_2,A_{\rm reg})
        =
        \int_{\Vc} 
            \Delta_{\pi}(\mathrm{d}x,\mathrm{d}u,\mathrm{d}r_1,\mathrm{d}r_2)
        \,\Ir(\mathrm{d}\pi),
    \]
    where $\Vc$ is a Polish space, $\Ir \in \Pc(\Vc)$ is a probability measure, and the map 
    $\pi \mapsto \Delta_{\pi}(\mathrm{d}x,\mathrm{d}u,\mathrm{d}r_1,\mathrm{d}r_2)$
    is Borel measurable such that, $\Ir$--a.e.,
    \begin{align*}
        \Delta_{\pi}(\mathrm{d}x,\mathrm{d}u,\mathrm{d}r_1,\mathrm{d}r_2)
        &\in \Mcb, 
        \quad \text{and}\\
        \Delta_{\pi}(\mathrm{d}x,\mathrm{d}u,\Hc,\Hc)
        &= \Lambda_t(\mathrm{d}x,\mathrm{d}u,\Hc,\Hc)
        = \mu_t(\mathrm{d}x,\mathrm{d}u).
    \end{align*}
\end{definition}

\medskip
We denote by $\Pc_R$ the set of all relaxed controls $\left( \mu,\Lambda \right)$.  
The above condition expresses that, at each time $t$, the conditional measure 
$\Lambda_t(\mathrm{d}x,\mathrm{d}u,\mathrm{d}r_1,\mathrm{d}r_2,A_{\rm reg})$
can be viewed as a probabilistic mixture (or convex combination) of admissible configurations $\Mcb$, representing possible interaction structures.  
However, the marginal distribution in the variables $(x,u)$ is kept fixed and identical across all realizations of $\pi$, i.e.,
\[
    \Delta_{\pi}(\mathrm{d}x,\mathrm{d}u) = \mu_t(\mathrm{d}x,\mathrm{d}u),
\]
which ensures that the randomization only affects the interaction component, while preserving the underlying population distribution.


\medskip
The next result provides a representation of any relaxed control. 
It shows that every pair $(\mu,\Lambda)$ can be realized through deterministic measurable maps 
together with auxiliary uniform randomizations, and that the corresponding state dynamics 
admits an explicit stochastic differential form.

\begin{proposition} \label{prop:charac_relaxed}
    Let $(\mu,\Lambda)$ be a relaxed control. Then there exist   measurable maps 
\begin{align*}
    \gammabb &: [0,T] \x \bigl( \R^d \x [0,1]^2 \bigr)^2 \x [0,1]^2 \longrightarrow A_{\rm int},
    \\
    \beta &: [0,T] \x \bigl( \R^d \x [0,1]^2 \bigr) \x [0,1]^2 \longrightarrow A_{\rm reg}.
\end{align*}
    and an auxiliary random variable $\bigl(\Vt,\Vf \bigr) \sim {\rm Unif}([0,1])^{\otimes^2}$, independent of $(X_t,U,V,\Vb)$, such that:
    \begin{itemize}
        \item[(i)] For almost every $(t,\pi) \in [0,T] \x [0,1]$, $\gammabb_t(\pi) \in \Ucb$, the relaxed control component $\Lambda_t$ admits the representation
        \begin{align*}
            \Lambda_t
            &= \Lc\Bigl(
                X_t,\,
                U,\,
                \Nb^{1,\gammabb_t(\Vf )}_{\mu_t}(X_t,U,V),\,
                \Nb^{2,\gammabb_t(\Vf)}_{\mu_t}(X_t,U,V),\,
                \beta\bigl(t,X_t,U,V,\Vt,\Vf\bigr)
            \Bigr),
        \end{align*}
        where $\mu_t=\Lc(X_t,U)$ and $\gammabb_t(\pi)(x,u,v,\hat x,\hat u,\hat v, \overline{v}):=\gammabb(t,x,u,v,\hat x,\hat u,\hat v,\overline{v},\pi)$.

        \item[(ii)] The controlled state process $(X_t)_{t \in [0,T]}$ satisfies the stochastic dynamics: $U \perp W$ and
        \begin{align}
        \label{eq:general_relaxed}
            \mathrm{d}X_t 
            &= \int_{[0,1]^3}  
                b\Bigl(
                    t,\,
                    X_t,\,
                    \Nb^{1,\gammabb_t(\pi)}_{\mu_t}(X_t,U,v),\,
                    \Nb^{2,\gammabb_t(\pi)}_{\mu_t}(X_t,U,v),\,
                    \beta\bigl(t,X_t,U,v,\tilde v, \pi\bigr) 
                \Bigr)
            \,\mathrm{d}v\,\mathrm{d}\tilde v \,\mathrm{d}\pi \,\mathrm{d}t 
            + \sigma(t,X_t)\,\mathrm{d}W_t.
        \end{align}
    \end{itemize}
\end{proposition}

\begin{proof}
We proceed in 4 steps. 

\medskip
\noindent \emph{Step 1: Disintegration of the regular control component.}  
By the classical disintegration theorem, there exists a Borel measurable kernel
\[
    \Theta : [0,T] \x \R^d \x [0,1] \x \Hc^2 \;\longrightarrow\; \Pc(A_{\rm reg})
\]
such that, for all bounded measurable test functions $\varphi$, one has
\[
    \int_{A_{\rm reg}} \varphi(a)\,\Theta(t,x,u,r_1,r_2)(\mathrm{d}a)\;
    \Lambda_t(\mathrm{d}x,\mathrm{d}u,\mathrm{d}r_1,\mathrm{d}r_2,A_{\rm reg})
    =
    \int_{A_{\rm reg}} \varphi(a)\,
    \Lambda_t(\mathrm{d}x,\mathrm{d}u,\mathrm{d}r_1,\mathrm{d}r_2,\mathrm{d}a).
\]
In other words, $\Theta$ is the conditional distribution of the $A_{\rm reg}$--component given $(t,x,u,r_1,r_2)$.  

\medskip
\noindent \emph{Step 2: Measurable parametrization of $\Theta$.}  
By a measurable selection theorem, 
there exists a Borel map
\[
    \theta : [0,T] \x \R^d \x [0,1] \x \Hc^2 \x [0,1] \;\longrightarrow\; A_{\rm reg}
\]
and an auxiliary random variable $\Vt \sim {\rm Unif}([0,1])$, independent of $(X_t,U,V,\Vb)$, 
such that
\[
    \Theta(t,x,u,r_1,r_2)(\mathrm{d}a)
    = \Lc\Bigl(\theta \bigl(t,x,u,r_1,r_2,\Vt \bigl)\Bigr)(\mathrm{d}a).
\]

\medskip
\noindent \emph{Step 3: Representation of the interaction component.}  
By definition of $(\mu,\Lambda)$ being a relaxed control, 
the marginal distribution 
\[
    \Lambda_t(\mathrm{d}x,\mathrm{d}u,\mathrm{d}r_1,\mathrm{d}r_2,A_{\rm reg}) = \int_{\Vc_t} 
            \Delta_{t,\pi}(\mathrm{d}x,\mathrm{d}u,\mathrm{d}r_1,\mathrm{d}r_2)
        \,\Ir_t(\mathrm{d}\pi)
\]
where $\Vc_t$ is a Polish space, $\Ir_t$--a.e. $\pi$, $\Delta_{t,\pi}(\mathrm{d}x,\mathrm{d}u,\mathrm{d}r_1,\mathrm{d}r_2)$ belongs to $\Mcb$ and $\Delta_{t,\pi}(\mathrm{d}x,\mathrm{d}u,\Hc,\Hc)$ is independent of $\pi$. Without loss of generality, we can assume that $\Vc_t=[0,1]$. Hence, there exists $\gammabb_t(\pi) \in \Ucb$ such that
\[
    \Lambda_t(\mathrm{d}x,\mathrm{d}u,\mathrm{d}r_1,\mathrm{d}r_2,A_{\rm reg})
    = \int_{[0,1]}\Lc\Bigl(
        X_t,\,
        U,\,
        \Nb^{1,\gammabb_t(\pi)}_{\mu_t}(X_t,U,V),\,
        \Nb^{2,\gammabb_t(\pi)}_{\mu_t}(X_t,U,V)
    \Bigr) \Ir_t(\mathrm{d}\pi),
\]
where $U,V$ are independent uniforms, whose law does not depend on $t$ 
(and thus can be chosen consistently across time).  

\medskip
\noindent \emph{Step 4: Construction of the global feedback maps.}  
For each $t$, the measure $\Ir_t$ is a probability measure on a Polish space $\Vc_t$. We can find a map Borel map $[0,1] \ni s \mapsto c(t,s) \in [0,1]$ s.t. $\Lc(c(t,V))=\Ir_t$. 
Since the map $t \mapsto \Lambda_t$ is measurable, we can choose $t \mapsto \Pi_t$ measurable and therefore $[0,T] \x [0,1] \ni (t,s) \mapsto c(t,s) \in [0,1]$ measurable.
Again, by using the fact that the map $t \mapsto \Lambda_t$ is measurable, we can select a Borel measurable 
$\gammat : [0,T] \x (\R^d \x [0,1]^2)^2 \x [0,1]^2 \to A_{\rm int}$ such that
\[
    \gammat(t,x,u,v,\hat x,\hat u,\hat v,\overline{v},\pi) = \gammabb_t(\pi)(x,u,v,\hat x,\hat u,\hat v, \overline{v}),
    \qquad \text{for a.e. } t \in [0,T].
\]

We then define
\[
    \gammabb(t,x,u,v,\hat x,\hat u,\hat v,\overline{v},\pi)
    :=
    \gammat(t,x,u,v,\hat x,\hat u,\hat v,\overline{v},c(t,\pi))
\]
and
\[
    \beta(t,x,u,v,\widetilde{v},\pi)
    := \theta\Bigl(
        t,\,
        x,\,
        u,\,
        \Nb^{1,\gammabb_t(\pi)}_{\mu_t}(x,u,v),\,
        \Nb^{2,\gammabb_t(\pi)}_{\mu_t}(x,u,v),\,
        \widetilde{v}
    \Bigr).
\]

\medskip
\noindent \emph{Conclusion.}  
The above construction provides the representation of $\Lambda_t$ stated in Item (i) of the proposition.  
Finally, Item (ii) follows from this explicit representation together with the admissibility of $(\mu,\Lambda)$, 
which ensures that the dynamics of $X_t$ can be written in the differential form given in the statement.  

\end{proof}

\medskip
Let us recall that we denote by $\Acb_{\rm int}$ the set of Borel measurable \emph{interaction maps}
\[
    \Acb_{\rm int} \;:=\; 
    \Bigl\{\, 
        \gammabb : [0,T] \times \bigl( \R^d \times [0,1]^2 \bigr)^2 \times [0,1]^2 
        \;\longrightarrow\; A_{\rm int} 
    \,\Bigr\},
\]
and by $\Acb_{\rm reg}$ the set of Borel measurable \emph{regular control maps}
\[
    \Acb_{\rm reg} \;:=\; 
    \Bigl\{\, 
        \beta : [0,T] \times \bigl( \R^d \times [0,1]^2 \bigr) \times [0,1]^2 
        \;\longrightarrow\; A_{\rm reg} 
    \,\Bigr\}.
\]

\medskip
In the next result, we establish the uniqueness of the marginal law of the state process associated with a given pair of measurable controls.

\begin{proposition} \label{prop:relaxed_uniqueness}
    Let $(\gammabb,\beta) \in \Acb_{\rm int} \times \Acb_{\rm reg}$, and let $X$ be a solution to the {\rm equation~\eqref{eq:general_relaxed}} driven by the common sources of randomness $(X_0,U,W)$.  
    Then, for each $t \in [0,T]$, the joint distribution $\Lc(X_t,U)$ is uniquely determined.  
    In other words, if $\Xt$ is another process satisfying~\eqref{eq:general_relaxed} with $(\gammabb,\beta, \Xt_0,\Ut,\Wt)$ s.t. $\Lc\bigl(\Xt_0,\Ut,\Wt\bigr)=\Lc(X_0,U,W)$, then
    \[
        \mu_t=\Lc(X_t,U) = \Lc\bigl(\Xt_t,\Ut \bigr), 
        \qquad \text{for all } t \in [0,T].
    \]
\end{proposition}

\begin{proof}
Let $(\mu^1_t)_{t \in [0,T]}$ and $(\mu^2_t)_{t \in [0,T]}$ be two solutions associated with the processes $(X^{1},U^1,W^1)$ and $(X^{2}, U^2,W^2)$ respectively. 
By definition $\mu^1_t(\R \times \cdot)=\mu^2_t(\R \times \cdot)$ for each $t$, any difference between $\mu^1_t$ and $\mu^2_t$ must lie in their $u$--conditionals. 
More precisely, writing the (Borel) disintegrations
\[
  \mu^k_t(\mathrm{d}x,\mathrm{d}u)=\mu^k_t(u)(\mathrm{d}x)\,\mathrm{d}u,
  \qquad k=1,2,
\]
we have that $u \mapsto \mu^1_t(u)$ and $u \mapsto \mu^2_t(u)$ carry all the information about the possible discrepancy.

\medskip
\emph{Step 1: A change of measure aligning the drifts.}
Define the (vector) process $B$ by
\begin{align*}
    B(t,X^1_t,U^1,\mu^1_t,\mu^2_t)
    &:=
        -\int_{[0,1]^3} b\Bigl(t, X^1_t, \Nb^{1,\gammabb_t(\pi)}_{\mu^2_t}(X^1_t,U^1,v), \Nb^{2,\gammabb_t(\pi)}_{\mu^2_t}(X^1_t,U^1,v), \beta(t,X^1_t,U^1,v,\tilde v,\pi) \Bigr) \,\mathrm{d}v\,\mathrm{d}\tilde v\,\mathrm{d}\pi
        \\
    &\quad + \int_{[0,1]^3} b\Bigl(t, X^1_t, \Nb^{1,\gammabb_t(\pi)}_{\mu^1_t}(X^1_t,U^1,v), \Nb^{2,\gammabb_t(\pi)}_{\mu^1_t}(X^1_t,U^1,v), \beta(t,X^1_t,U^1,v,\tilde v,\pi) \Bigr) \,\mathrm{d}v\,\mathrm{d}\tilde v \,\mathrm{d}\pi.
\end{align*}
Consider $Z$ solving
\[
    Z_0=1,
    \qquad
    \mathrm{d}Z_t = Z_t\, \sigma^{-1}(t,X^1_t)\, B(t,X^1_t,U^1,\mu^1_t,\mu^2_t)\, \mathrm{d}W^1_t .
\]
Since $b$ is bounded and $\sigma$ is uniformly non--degenerate, Novikov’s criterion holds, and thus $Z$ is a true martingale with $\E[Z_T]=1$. 
Define a new probability $\Pb$ on $(\Omega,\F)$ by $\mathrm{d}\Pb := Z_T \mathrm{d}\P$. 
By Girsanov’s theorem and the fact that $(X_0^1,U^1) \perp W^1$,
\[
  \Wb_\cdot := W^1_\cdot - \int_0^\cdot \sigma^{-1}(t,X^1_t)\, B(t,X^1_t,U^1,\mu^1_t,\mu^2_t)\, \mathrm{d}t
\]
is a $\Pb$--Brownian motion (conditionally on each value of $U^1$). 
In particular, for a.e.\ $u$,
\[
  \Lc^{\Pb}\bigl(\Wb \mid U^1=u\bigr)=\Lc(W^1),
  \qquad\text{and hence}\qquad
  \Lc^{\Pb}\bigl(X^1 \mid U^1=u\bigr)=\Lc^{\P}\bigl(X^2 \mid U^2=u\bigr),
\]
because $\Lc(X_0^1,U^1,W^1)=\Lc(X_0^2,U^2,W^2)$ and, under $\Pb$ the drift of $X^1$ matches that of $X^2$ when the interaction kernel is evaluated at $\mu^2$ and $U^2=u$.

\medskip
\emph{Step 2: Relative entropy on path space and the conditional density.}
For a Polish space $\Vc$ and $m,m'\in\Pc(\Vc)$, recall
\[
  H(m\mid m') := 
  \begin{cases}
    \displaystyle \int_{\Vc} \frac{\mathrm{d}m}{\mathrm{d}m'} \log\!\Big(\frac{\mathrm{d}m}{\mathrm{d}m'}\Big)\, \mathrm{d}m' 
      & \text{if } m \ll m',\\[2mm]
    +\infty & \text{otherwise.}
  \end{cases}
\]
By the definition of $\Pb$ and the standard conditional density argument,
\[
  \Lc\bigl( X^2_{t\wedge\cdot} \mid U^2=u \bigr)(\mathrm{d}x)
  = \E\!\left[ Z_t \mid X^1_{t\wedge\cdot}=x,\, U^1=u \right]\,
    \Lc\bigl( X^1_{t\wedge\cdot} \mid U^1=u \bigr)(\mathrm{d}x)
  \quad \text{a.e.}
\]
Moreover, using that $\sigma$ is non--degenerate, the (completed) filtrations generated by $(X^1,U^1)$ and by $(X^1_0,W^1,U^1)$ coincide; hence the Cameron–Martin/Girsanov computation of entropy yields
\[
  H\!\left( \Lc(X^1_{t\wedge\cdot},U^1) \,\middle\vert\, \Lc(X^2_{t\wedge\cdot},U^2) \right)
  \;=\; -\,\E \!\left[ \log \E\!\left[ Z_t \,\middle\vert\, X^1_{t\wedge\cdot},\,U^1 \right] \right]
  \;=\; \frac{1}{2}\, \E \!\left[ \int_0^t 
        \big\lvert \sigma^{-1}(s,X^1_s) B(s,X^1_s,U^1,\mu^1_s,\mu^2_s) \big\rvert^2\, \mathrm{d}s \right].
\]

\medskip
\emph{Step 3: Lipschitz control in total variation and Grönwall.}
By the Lipschitz property of $b$ in its interaction arguments and the boundedness/measurability of $\gammabb$, there exists $C>0$ such that, for all $s$,
\[
  \bigl| B(s,X^1_s,U^1,\mu^1_s,\mu^2_s) \bigr|
  \;\le\; C\, \| \mu^1_s - \mu^2_s \|_{\rm TV},
\]
where $\|\cdot\|_{\rm TV}$ denotes total variation i.e.  $\|\nu-\nu'\|_{\rm TV}:=\sup_{|f| \le 1} |\langle f, \nu-\nu' \rangle|$, the supremum is being taken over Borel measurable maps. 
Consequently, for some $K>0$ and every $t\in[0,T]$,
\[
  H\!\left( \Lc(X^1_{t\wedge\cdot},U^1) \,\middle\vert\, \Lc(X^2_{t\wedge\cdot},U^2) \right)
  \;\le\; K\, \E\!\left[ \int_0^t \| \mu^1_s - \mu^2_s \|_{\rm TV}^2 \, \mathrm{d}s \right].
\]
By monotonicity of total variation,
\[
  \| \mu^1_t - \mu^2_t \|_{\rm TV}
  \;\le\; \big\| \Lc(X^1_{t\wedge\cdot},U^1) - \Lc(X^2_{t\wedge\cdot},U^2) \big\|_{\rm TV},
\]
and by Pinsker’s inequality,
\[
  \big\| \Lc(X^1_{t\wedge\cdot},U^1) - \Lc(X^2_{t\wedge\cdot},U^2) \big\|_{\rm TV}^2
  \;\le\; \tfrac12\, H\!\left( \Lc(X^1_{t\wedge\cdot},U^1) \,\middle\vert\, \Lc(X^2_{t\wedge\cdot},U^2) \right).
\]
Putting these together, we obtain
\[
  \| \mu^1_t - \mu^2_t \|_{\rm TV}^2
  \;\le\; \frac12\, H\!\left( \Lc(X^1_{t\wedge\cdot},U^1) \,\middle\vert\, \Lc(X^2_{t\wedge\cdot},U^2) \right)
  \;\le\; K\, \E\!\left[ \int_0^t \| \mu^1_s - \mu^2_s \|_{\rm TV}^2 \, \mathrm{d}s \right].
\]
By Grönwall’s lemma, it follows that $\| \mu^1_t - \mu^2_t \|_{\rm TV}=0$ for all $t\in[0,T]$, hence $\mu^1_t=\mu^2_t$ for every $t$. 
In particular, $\mu^1_t(u)=\mu^2_t(u)$ for a.e.\ $u\in[0,1]$ and all $t\in[0,T]$.

\medskip
This proves the claimed uniqueness of $\Lc(X_t,U)$ for each $t\in[0,T]$. 
Applying the same argument to any other solution $\Xt$ with $(\Xt_0,\Ut,\Wt)$ such that $\Lc(\Xt_0,\Ut,\Wt)=\Lc(X_0,U,W)$ yields
\[
  \mu_t = \Lc(X_t,U) = \Lc(\Xt_t,\Ut), \qquad t\in[0,T].
\]
\end{proof}

\medskip
In view of \Cref{prop:charac_relaxed}, every relaxed control $(\mu,\Lambda) \in \Pc_R$ can be represented through some pair $(\gammabb,\beta) \in \Acb_{\rm int} \x \Acb_{\rm reg}$ and the inital distribution $\mu_0=\Lc(X_0,U)$. 
By uniqueness in distribution of the corresponding state dynamics, this representation is unambiguous.  
It is therefore convenient to emphasize the dependence on $(\gammabb,\beta)$ by writing
\[
    (\mu,\Lambda) \equiv \bigl(\mu^{\gammabb,\beta},\,\Lambda^{\gammabb,\beta}\bigr).
\]
We will add the initial distribution in the index when it is relevant.

\medskip

\begin{proposition}[Closedness of relaxed controls] \label{prop:closed_relaxed}
    The set of relaxed controls $\Pc_R \subset \Cc^d_{\rm pr}\times \M(\R^d \x [0,1] \x \Rc)$ is closed with respect to the Wasserstein topology.
\end{proposition}

\begin{proof}
Fix a sequence $(\mu^k,\Lambda^k)_{k\ge1}\subset \Pc_R$ and suppose
\[
    (\mu^k,\Lambda^k)\;\Longrightarrow\;(\mu,\Lambda)
    \quad\text{in the Wasserstein topology on}\quad
    \Cc^d_{\rm pr}\times \M(\R^d \x [0,1] \x \Rc).
\]
We prove that $(\mu,\Lambda)\in\Pc_R$ by verifying (i) admissibility and (ii) the structural constraint on the interaction component, for a.e.\ $t\in[0,T]$.

\medskip
\noindent\emph{Step 1: Admissibility of the limit.}
By definition of relaxed controls, each $(\mu^k,\Lambda^k)$ solves the corresponding Fokker--Planck (or weak formulation) with test functions smooth in $x$ and bounded in the remaining variables. Passing to the limit in the Fokker--Planck equation under Wasserstein convergence : for any test smooth map $F$ and $\phi$,
\begin{align*}
   &\mathrm{d}\int_{\R^d \x [0,1]}\phi(x)F(u) \mu_t(\mathrm{d}x,\mathrm{d}u)
   \\
   &=\int_{\R^d \x [0,1] \x \Rc}\phi'(x) F(u) b(t,x,r_1,r_2,a) \Lambda_t(\mathrm{d}x,\mathrm{d}u,\mathrm{d}r_1,\mathrm{d}r_2,\mathrm{d}a)\mathrm{d}t + \frac{1}{2}\int_{\R^d } {\rm Tr}[ \phi''(x) \sigma(t,x) \sigma^\top(t,x)]F(u) \mu_t(\mathrm{d}x,\mathrm{d}u) \mathrm{d}t.
\end{align*}
Let $(t,x,u) \mapsto \Gamma(t,x,u)$ be the Borel map verifying $\Lambda= \Gamma(t,x,u)(\mathrm{d}r_1,\mathrm{d}r_2,\mathrm{d}a) \Lambda_t( \mathrm{d}x,\mathrm{d}u,\Rc) \mathrm{d}t$. We take $(X_0,U) \perp W$ with $\Lc(X_0,U)=\mu_0$, and we define $Y$ by : $\Lc(Y_0,U)=\mu_0$ and
\begin{align*}
    \mathrm{d}Y_t= \int_{\Rc} b \left(t,Y_t,r_1,r_2,a \right)\Gamma(t,Y_t,U)(\mathrm{d}r_1,\mathrm{d}r_2,\mathrm{d}a) \mathrm{d}t + \sigma(t,Y_t)\mathrm{d}W_t.
\end{align*}
We can then check that $\Lc(Y_t,U)=\mu_t$ for each $t$. This allows us to deduce that $(\mu,\Lambda)$ is \emph{admissible}.

\medskip
\noindent\emph{Step 2: The interaction marginal lies in $\Mcb$ for a.e.\ $(t,\pi)$.}
By the auxiliary characterization $\Mcb_{\rm aux}=\Mcb$, for each $k$ and a.e.\ $(t,\pi)$ there exists a \emph{compatible} tuple
\[
    \Gamma_{t,\pi}^k=\big(\gamma_{t,\pi}^k,\;Z_{t}^k,\;N_{t,\pi}^k,\;\Zh_{t}^k,\;\Nh_{t,\pi}^k\big),
    \quad
    (Z_t^k,N_{t,\pi}^k)\stackrel{d}{=}\Delta_{t,\pi}^k(\mathrm{d}x,\mathrm{d}u,\mathrm{d}r_1,\mathrm{d}r_2),
\]
such that 
\begin{align*}
    \Lambda_{t}^k(\mathrm{d}x,\mathrm{d}u,\mathrm{d}r_1,\mathrm{d}r_2,A_{\rm reg})=\int_0^1 \Delta_{t,\pi}^k(\mathrm{d}x,\mathrm{d}u,\mathrm{d}r_1,\mathrm{d}r_2)\Ir_t^k(\mathrm{d}\pi)
\end{align*}
and
$\Lc\!\big(Z_t^k,N_{t,\pi}^k,\Zh_t^k,\Nh_{t,\pi}^k\big)=\Sigma^k_{t,\pi}\otimes \Sigma^k_{t,\pi}$, $\Sigma^k_{t,\pi}:=\Lc(Z_t^k,N_{t,\pi}^k),$
\[
    N_{t,\pi}^{1,k}=\Lc\!\big(\gamma_{t,\pi}^k,\Zh_t^k,\Gr(U_{t,\pi}^k,\Uh_{t,\pi}^k)\mid Z_t^k,N^k_{t,\pi}\big),
    \quad
    \Nh_{t,\pi}^{2,k}=\Lc\!\big(\gamma_{t,\pi}^k,Z_t^k,\Gr(\Uh_{t,\pi}^k,U_{t,\pi}^k)\mid \Zh_t^k,\Nh^k_{t,\pi}\big).
\]
Define the time–space law $\overline{\Lambda}^k$
\[
    \overline{\Lambda}^k
    :=
    \int_0^1\Lc\left(\gamma_{t,\pi}^k,Z_t^k,N_{t,\pi}^k,\Zh_t^k,\Nh_{t,\pi}^k,\Sigma^k_{t,\pi}\right)\,\Ir^k_t(\mathrm{d}\pi)\,\mathrm{d}t
    \,\, \in \,\, \M\!\Big(A_{\rm int} \x \big(\R^d\x[0,1]\x\Hc^2\big)^2 \x \Pc\big(\R^d\x[0,1]\x\Hc^2 \big)\Big)
\]
and the flow $\mu^k$
\[
    \mub^k_t:=\Lc(Z^k_t),\,t \in [0,T].
\]
By Prokhorov and the uniform integrability encoded in the initial sequence $(\mu^k,\Lambda^k)_{k \ge 1}$, we deduce that $\bigl( \overline{\Lambda}^k,\mub^k \bigr)$  is relatively compact; take a convergent subsequence (not relabeled) with limit
\[
    \overline{\Lambda} \ = \
    \Lc\big(\gamma_t,Z_t,N_t,\Zh_t,\Nh_t, \Sigma_t\big)\,\mathrm{d}t\quad \mbox{and}\quad \left( \mub_t=\Lc(Z_t) \right)_{t \in [0,T]} .
\]

\smallskip
\noindent\underline{Compatibility is preserved in the limit.}
Let $\varphi,\widehat\varphi,G$ be bounded continuous. Using weak convergence and the compatibility of $\Gamma_{t,\pi}^k$,
\begin{align*}
    \int_0^T \E\!\left[\varphi(t,Z_t,N_t)\widehat\varphi(t,\Zh_t,\Nh_t) G(\Sigma_t)\right]\mathrm{d}t
    &= \lim_{k\to\infty}\int_{[0,T] \x [0,1]} \E\!\left[\varphi(t,Z_t^k,N_{t,\pi}^k)\widehat\varphi(t,\Zh_t^k,\Nh_{t,\pi}^k) G(\Sigma^k_{t,\pi})\right]\,\Ir^k_t(\mathrm{d}\pi)\,\mathrm{d}t
    \\
    &= \lim_{k\to\infty}\int_{[0,T] \x [0,1]} \langle \varphi(t,\cdot),\Sigma^k_{t,\pi}\rangle \langle \widehat\varphi(t,\cdot),\Sigma^k_{t,\pi}\rangle G(\Sigma^k_{t,\pi})\,\Ir^k_t(\mathrm{d}\pi)\,\mathrm{d}t
    \\
    &= \int_0^T \E \left[ \langle \varphi(t,\cdot),\Sigma_{t}\rangle \langle \widehat\varphi(t,\cdot),\Sigma_{t}\rangle G(\Sigma_{t})\right]\,\mathrm{d}t.
\end{align*}
Hence, $\mathrm{d}\P \otimes \mathrm{d}t$  a.e., 
\[
    \Lc\big(Z_t,N_t,\Zh_t,\Nh_t \mid \Sigma_t\big)=\Lc(Z_t,N_t\mid \Sigma_t)\otimes \Lc(Z_t,N_t \mid \Sigma_t)=\Sigma_t \otimes \Sigma_t .
\]

\smallskip
\noindent\underline{Kernel identities pass to the limit.}
Let $F,\Phi,G$ be bounded continuous. Using again the defining identities for $\Gamma_{t,\pi}^k$, and that $(U_t^k,\Uh_t^k)\sim{\rm Unif}([0,1])^{\otimes2}$ for all $k,t$ (hence independent of $k$ and $t$), we may combine weak with stable convergence to obtain
\begin{align*}
    \int_0^T \E\!\left[\langle F,N_t^{1}\rangle\,\Phi(t,Z_t,N_t)\,G(\Sigma_t)\right]\mathrm{d}t
    &= \lim_{k\to\infty}\int_{[0,T] \x [0,1]} \E\!\left[\langle F,N_{t,\pi}^{1,k}\rangle\,\Phi(t,Z_t^k,N_{t,\pi}^k)\,G(\Sigma^k_{t,\pi})\right]\,\Ir^k_t(\mathrm{d}\pi)\,\mathrm{d}t
    \\
    &= \lim_{k\to\infty}\int_{[0,T] \x [0,1]} \E\!\left[F\big(\gamma_{t,\pi}^k,\Zh_t^k,\Gr(U_t^k,\Uh_t^k)\big)\,\Phi(t,Z_t^k,N_{t,\pi}^k)\,G(\Sigma^k_{t,\pi})\right]\,\Ir^k_t(\mathrm{d}\pi)\,\mathrm{d}t
    \\
    &= \int_0^T \E\!\left[F\big(\gamma_t,\Zh_t,\Gr(U_t,\Uh_t)\big)\,\Phi(t,Z_t,N_t)\,G(\Sigma_t)\right]\mathrm{d}t.
\end{align*}
Since $F,\Phi,G$ are arbitrary, $\mathrm{d}\P \otimes \mathrm{d}t$  a.e.,
\[
    N_t^{1}=\Lc\big(\gamma_t,\Zh_t,\Gr(U_t,\Uh_t)\mid Z_t,N_t, \Sigma_t\big).
\]
A symmetric argument yields, $\mathrm{d}\P \otimes \mathrm{d}t$  a.e.,
\[
    \Nh_t^{2}=\Lc\big(\gamma_t,Z_t,\Gr(\Uh_t,U_t)\mid \Zh_t,\Nh_t, \Sigma_t\big).
\]
Therefore, given $\Sigma_t$, $(\gamma_t,Z_t,N_t,\Zh_t,\Nh_t)$ is \emph{compatible} for a.e.\ $t$.

\smallskip
\noindent\underline{Independence $Z_t$ and $\Sigma_t$.} Let $(f,G)$ be continuous bounded maps. Using the weak convergence of $(\mu^k)_{k \ge 1}$ and $\bigl(\overline{\Lambda}^k \bigr)_{k \ge 1}$, we obtain
\begin{align*}
    \int_0^T \E  \left[ f \left(t, Z_t \right) G \left( \Sigma_t \right) \right]\,\mathrm{d}t &= \Lim_{k \to \infty} \int_{[0,T] \x [0,1]} \E \left[ f \left(t, Z^k_t \right) G \left( \Sigma^k_{t,\pi} \right) \right]\,\Ir_{t}(\mathrm{d}\pi)\,\mathrm{d}t 
    \\
    &= \Lim_{k \to \infty} \int_{[0,T] \x [0,1]}  \langle f(t,\cdot), \mub^k_t \rangle \E \left[ G \left( \Sigma^k_{t,\pi} \right) \right]\,\Ir_{t}(\mathrm{d}\pi)\,\mathrm{d}t
    \\
    &=\int_{[0,T]}  \E \left[\langle f(t,\cdot), \mub_t \rangle  G \left( \Sigma_{t} \right) \right]\,\mathrm{d}t.
\end{align*}
This being true for any $(f,G)$, we deduce that $\Lc(Z_t \mid \Sigma_t)=\mub_t$ $\mathrm{d}\P \otimes\mathrm{d}t$--a.e. Since the process $(\mub_t)_{t \in [0,T]}$ is deterministic, we deduce that $Z_t \perp \Sigma_t$ for a.e. $t$.

\medskip
\noindent\emph{Step 3: Identification of the interaction marginal and conclusion.}
By construction of $\overline{\Lambda}$ and marginalization,
\[
    \Lambda_t(\mathrm{d}x,\mathrm{d}u,\mathrm{d}r_1,\mathrm{d}r_2,A_{\rm reg})
    \ =\ \Lc(Z_t,N_t)=\E \left[ \Lc(Z_t,N_t \mid \Sigma_t) \right]\quad\text{for a.e.\ }t.
\]
Since, given $\Sigma_t$, $(\gamma_t,Z_t,N_t,\Zh_t,\Nh_t)$ is compatible, the auxiliary characterization gives
\[
    \Lc(Z_t,N_t \mid \Sigma_t)\in \Mcb \quad \mbox{and}\quad Z_t \perp \Sigma_t
    \qquad\text{for a.e.\ }t.
\]

Combining Step~1 (admissibility of $(\mu,\Lambda)$) with Step~2, we conclude that $(\mu,\Lambda)$ satisfies the defining properties of a relaxed control, i.e., $(\mu,\Lambda)\in\Pc_R$.

\medskip
Hence, $\Pc_R$ is closed in the Wasserstein topology.
\end{proof}

\medskip
A \emph{relaxed} control $\left( \mu,\Lambda \right)$ will be called \emph{strong} if there exist $\alpha \in \Ac_{\rm reg}$ and $\gamma \in \Ac_{\rm int}$ s.t. a.e. $t$
\begin{align*}
    \Lambda_t(\mathrm{d}u,\mathrm{d}x,\mathrm{d}r_1,\mathrm{d}r_2,\mathrm{d}a)
    = \Lc\left( X_t,U,M^{1,\gamma}_{\mu,t}(X_t,U),M^{2,\gamma}_{\mu,t}(X_t,U), \alpha(t,X_t,U) \right) (\mathrm{d}x,\mathrm{d}u,\mathrm{d}r_1,\mathrm{d}r_2,\mathrm{d}a).
\end{align*}
We will write $\Pc_S$ for the set of strong controls. For any $(\mu,\Lambda) \in \Pc_S$, notice that, it is easy to see that it is equivalent to have: a.e. $t$, 
$$
    \Lambda_t(\mathrm{d}x,\mathrm{d}u,\mathrm{d}r_1,\mathrm{d}r_2,A_{\rm reg}) \in \Mc.
$$

\begin{proposition}[Stability through almost everywhere convergence.] \label{prop:a.e.cong} Let $(\gammabb^n,\beta^n)_{n \ge 1} \subset \Acb_{\rm int} \x \Acb_{\rm reg}$ be a sequence such that
    \[
        \lim_{n\to\infty} (\gammabb^n,\beta^n) = (\gammabb,\beta)\quad\text{a.e. on }[0,T]\x\bigl(\R^d\x[0,1]^2\bigr)^2\x[0,1]^2,
    \]
    for some $(\gammabb,\beta)\in\Acb_{\rm int}\x\Acb_{\rm reg}$.
    For each $n\ge 1$, we take $\bigl(\mu^{\gammabb^{n},\beta^{n}},\,\Lambda^{\gammabb^{n},\beta^{n}}\bigr) \in \Pc_R$,
    if
    \begin{align*}
        \lim_{n \to \infty} \mu^{\gammabb^{n},\beta^{n}}_0 =\mu^{\gammabb,\beta}_0\quad\text{in weak topology}
    \end{align*}
    then
    \[
        \lim_{n \to \infty} \bigl(\mu^{\gammabb^{n},\beta^{n}},\,\Lambda^{\gammabb^{n},\beta^{n}}\bigr)
        \;=\;
        \bigl(\mu^{\gammabb,\beta},\,\Lambda^{\gammabb,\beta}\bigr)
        \quad\text{in weak topology}.
    \]
\end{proposition}

\begin{remark}
\textnormal{(i)}  
An immediate consequence of {\rm\Cref{prop:a.e.cong}} is the \emph{existence} of a relaxed control pair 
\((\mu^{\gammabb,\beta},\Lambda^{\gammabb,\beta})\) for any 
\(\gammabb \in \Acb_{\rm int}\) and \(\beta \in \Acb_{\rm reg}\) given an initial distribution $\mu^{\gammabb,\beta}_0$.  
Indeed, given any such pair \((\gammabb,\beta)\), we can construct a sequence of Lipschitz $($or smooth $)$ approximations 
\((\gammabb^k,\beta^k)_{k \ge 1} \subset \Acb_{\rm int} \times \Acb_{\rm reg}\) 
such that 
\[
    (\gammabb^n,\beta^n) \longrightarrow (\gammabb,\beta)
    \quad\text{a.e. on } [0,T]\times(\R^d\times[0,1]^2)^2\times[0,1]^2.
\]
For each $n$, the regularity of $(\gammabb^n,\beta^n)$ guarantees the existence of a corresponding relaxed solution 
$(\mu^{\gammabb^n,\beta^n},\Lambda^{\gammabb^n,\beta^n})$.  
By {\rm \Cref{prop:a.e.cong}}, the sequence converges in the Wasserstein sense, 
yielding a limiting pair 
\[
    (\mu^{\gammabb,\beta},\Lambda^{\gammabb,\beta})
    := \lim_{n\to\infty} (\mu^{\gammabb^n,\beta^n},\Lambda^{\gammabb^n,\beta^n}),
\]
which provides the desired relaxed control associated with the $($possibly merely measurable$)$ maps $(\gammabb,\beta)$.

\medskip
\textnormal{(ii)}  
Combining this stability property with the uniqueness result established in 
{\rm \Cref{prop:relaxed_uniqueness}}, we conclude that for every 
\((\gammabb,\beta)\in \Acb_{\rm int} \times \Acb_{\rm reg}\), 
the relaxed control pair 
\((\mu^{\gammabb,\beta},\Lambda^{\gammabb,\beta})\) 
is \emph{well--defined and unique}.  
\end{remark}

\begin{proof}
\emph{Notation.} For each $n\ge1$, set
\[
    (\mu^{n},\Lambda^{n}):=\bigl(\mu^{\gammabb^{n},\,\beta^{n}},\,\Lambda^{\gammabb^{n},\,\beta^{n}}\bigr).
\]
By \Cref{prop:charac_relaxed}, we may represent for a.e.\ $t\in[0,T]$:
\[
    \mu^{n}_t=\Lc(X^{n}_t,U),\qquad
    \Lambda^{n}_t=\int_0^1\Lc\!\Bigl(X^{n}_t,\,U,\,\Nb^{1,\gammabb^{n}_t(\pi)}_{\mu^{n}_t}(X^{n}_t,U,V),\,\Nb^{2,\gammabb^{n}_t(\pi)}_{\mu^{n}_t}(X^{n}_t,U,V),\,\beta^{n}(t,X^{n}_t,U,V,\Vt,\pi)\Bigr) \,\mathrm{d}\pi,
\]
where $(U,V,\Vt)\sim{\rm Unif}([0,1])^{\otimes3}$ are independent of $W$, and $X^n$ solves
\begin{align*}
    \mathrm{d}X^{n}_t
    = \int_{[0,1]^3}
        b\!\left(t, X^{n}_t, \Nb^{1,\gammabb^{n}_t(\pi)}_{\mu^{n}_t}(X^{n}_t,U,v), \Nb^{2,\gammabb^{n}_t(\pi)}_{\mu^{n}_t}(X^{n}_t,U,v), \beta^{n}(t,X^{n}_t,U,v,\tilde v,\pi)\right)\,\mathrm{d}v\,\mathrm{d}\tilde v\,\mathrm{d}\pi\,\mathrm{d}t
     + \sigma(t,X^{n}_t)\,\mathrm{d}W_t.
\end{align*}

\medskip
\noindent\emph{Step 1: Conditional densities and uniform estimates.}
Since $X^{n}$ satisfies a non--degenerate SDE with bounded coefficients, it follows from \cite[Chapter 2 Section 3 Theorem 4]{krylov1980controlled} that for almost every $u \in [0,1]$, the conditional law $\Lc(X^{n}_t \mid U=u)$ admits a density $f^{n}_u(t,x)$ with respect to the Lebesgue measure. This density can be chosen so that the map $(u,t,x) \mapsto f^{n}_u(t,x)$ is measurable, and there exists a constant $C > 0$ (independent of $n$ and $u$) such that, for some $q > 1$,
\begin{align} \label{eq:estimates}
    \int_0^T \int_{\R^d} |f^{n}_u(t,x)|^q \,\mathrm{d}x\,\mathrm{d}t < C.
\end{align}

Moreover, since the coefficients $b$ and $\sigma$ are bounded and the sequence of initial laws 
$\bigl(\mu^{\gammabb^{n},\beta^{n}}_0\bigr)_{n\ge1}$ is convergent (hence relatively compact) for the weak topology, it follows that the family of trajectories $\bigl(t \mapsto \mu^{n}_t\bigr)_{n\ge1}$ is relatively compact in $C\big([0,T];\Pc(\R^d\times[0,1])\big)$. 
No additional assumption on the initial distribution is required, precisely because of the boundedness of the coefficients $(b,\sigma)$; see for instance the argument in \cite[Theorem~A.2]{djete2019general}.  
The key observation is that each process can be decomposed as 
$X^n_\cdot = (X^n_\cdot - X^n_0) + X^n_0$, 
and boundedness of $(b,\sigma)$ ensures that 
\[
    \sup_{n\ge1}\E\!\left[\sup_{t\in[0,T]}|X^n_t - X^n_0|^q\right] < \infty, 
    \qquad \forall\, q \ge 1.
\]
Hence, the sequence of path laws 
$\bigl(\Lc((X^n_t - X^n_0)_{t\in[0,T]})\bigr)_{n\ge1}$ 
is relatively compact in $\Wc_q$ for any $q\ge1$. 
Combining this with the relative compactness of the initial distributions 
$\bigl(\Lc(X^n_0)\bigr)_{n\ge1}$ for the weak topology, we conclude that the sequence 
$\bigl(\Lc((X^n_t)_{t\in[0,T]})\bigr)_{n\ge1}$ 
is relatively compact in the weak topology. 
Consequently, the family $\bigl(t \mapsto \mu^n_t\bigr)_{n\ge1}$ is relatively compact in 
$C\big([0,T];\Pc(\R^d\times[0,1])\big)$.

\medskip
Let $(\mu_t=\Lc(Y_t,U))_{t \in [0,T]}$ denote the limit of a convergent subsequence, for which we retain the same notation for simplicity. We can also check that a.e. $\Lc(Y_t|U=u)$ admits a density $f_u(t,x)$ (see for instance \cite[Proposition 9.1.]{10.1214/23-AAP1993}).

\medskip
\noindent\emph{Step 2: A general convergence lemma for expectations.}
Let $(\phi^n)_{n\ge1}$ be bounded continuous functions $\phi^n:[0,T]\x\R^d\x[0,1]^2\to\R$ such that $\phi^n\to\phi$ a.e.\ for some Borel $\phi$. We claim
\begin{equation}\label{eq:conv}
    \lim_{n\to\infty}\int_{[0,T]}
    \Big|\E\big[\phi^n(t,X^n_t,U)\big]-\E\big[\phi(t,Y_t,U)\big]\Big|\;\mathrm{d}t=0.
\end{equation}
Indeed, by conditioning on $U=u$ and using the densities,
\begin{align*}
    \E\big[\phi^n(t,X^n_t,U)\big]-\E\big[\phi(t,Y_t,U)\big]
    &= \int_0^1\!\int_{\R^d} \phi^n(t,x,u) f_u^n(t,x)\,\mathrm{d}x\,\mathrm{d}u
       - \int_0^1\!\int_{\R^d} \phi(t,x,u) f_u(t,x)\,\mathrm{d}x\,\mathrm{d}u
    \\
    &= \underbrace{\int_0^1\!\int_{\R^d} \big(\phi^n-\phi^{n_0}\big)(t,x,u)\,f_u^n(t,x)\,\mathrm{d}x\,\mathrm{d}u}_{=:I_{1}^{n,n_0}(t)}
    \\
    &\quad + \underbrace{\int_0^1\!\int_{\R^d} \phi^{n_0}(t,x,u)\,f_u^n(t,x)\,\mathrm{d}x\,\mathrm{d}u
       - \int_0^1\!\int_{\R^d} \phi(t,x,u)\,f_u(t,x)\,\mathrm{d}x\,\mathrm{d}u}_{=:I_{2}^{n,n_0}(t)}.
\end{align*}
Fix $K>0$ and split $I_{1}^{n,n_0}$ according to $|x|\le K$ and $|x|>K$:
\begin{align*}
    \int_0^T |I_{1}^{n,n_0}(t)|\,\mathrm{d}t
    &\le \int_0^T\!\!\int_0^1\!\!\int_{|x|\le K}|\phi^n-\phi^{n_0}|(t,x,u)\,f_u^n(t,x)\,\mathrm{d}x\,\mathrm{d}u\,\mathrm{d}t
    \\
    &\quad + 2\sup_{k\ge1}\|\phi^k\|_\infty\,\int_0^T \P(|X^n_t|>K)\,\mathrm{d}t.
\end{align*}
By H\"older with $q$ from \eqref{eq:estimates} and $q'=\frac{q}{q-1}$,
\[
    \int_0^T\!\!\int_0^1\!\!\int_{|x|\le K}|\phi^n-\phi^{n_0}|\,f_u^n (t,x)\,\mathrm{d}x\,\mathrm{d}u\,\mathrm{d}t
    \le \|\phi^n_K-\phi^{n_0}_K\|_{L^{q'}([0,T]\x[-K,K]^d\x[0,1])}\;\|f^n\|_{L^q},
\]
where $\phi^n_K:=\phi^n\mathbf{1}_{\{|x|\le K\}}$ and $\|f^n\|_{L^q}$ is uniformly bounded by \eqref{eq:estimates}. Since $\phi^n\to\phi$ a.e., for fixed $K$ we can choose $n,n_0$ large so that the $L^{q'}$–norm is small. The tail term vanishes as $K\to\infty$ by the uniform integrability of the sequence $(t \mapsto \mu^n_t)_{n \ge 1}$ (because relatively compact):
\[
    \lim_{K\to\infty}\sup_{n\ge1}\int_0^T \P(|X^n_t|>K)\,\mathrm{d}t=0.
\]
Thus
\begin{equation}\label{eq:first_conv}
    \lim_{n_0\to\infty}\lim_{n\to\infty}\int_0^T |I_{1}^{n,n_0}(t)|\,\mathrm{d}t=0.
\end{equation}
For $I_{2}^{n,n_0}$, by definition of $\mu^n$ and weak convergence of $\mu^n\to\mu$ in $C([0,T];\Pc(\R^d \x [0,1]))$,
\begin{align}
    \int_0^T |I_{2}^{n,n_0}(t)|\,\mathrm{d}t
    &= \int_0^T \Big| \E\big[\phi^{n_0}(t,X^n_t,U)\big]-\E\big[\phi(t,Y_t,U)\big]\Big|\,\mathrm{d}t \nonumber
    \\
    &\xrightarrow[n\to\infty]{}\int_0^T \Big| \E\big[\phi^{n_0}(t,Y_t,U)\big]-\E\big[\phi(t,Y_t,U)\big]\Big|\,\mathrm{d}t \xrightarrow[n_0\to\infty]{}0. \label{eq:second_conv}
\end{align}
Combining \eqref{eq:first_conv} and \eqref{eq:second_conv} yields \eqref{eq:conv}.

\medskip
\noindent\emph{Step 3: Convergence of interaction kernels.}
Let $H:A_{\rm int}\x(\R^d\x[0,1])\x\Er\to\R$ be bounded continuous. For each $(t,x,u,v)$,
\[
    \langle H,\Nb^{1,\gammabb^n_t}_{\mu^n_t}(x,u,v)\rangle
    \;=\;
    \E\!\left[H\Big(\gammabb^n(t,x,u,v,X^n_t,U,V,\Vb),\,X^n_t,U,\,\Gr(u,U)\Big)\right].
\]
Since $\gammabb^n\to\gammabb$ a.e.\ and $(X^n,U)\Rightarrow(Y,U)$ as in Step~2, adapting in an easy way \eqref{eq:conv} to the bounded continuous test functions
\[
    \phi^n(t',x',u',v',\overline{v},\pi):=H\big(\gammabb^n(t,x,u,v,x',u',v',\overline{v},\pi),x',u',\Gr(u,u')\big),
\]
we obtain, for any $K>0$,
\[
    \lim_{n\to\infty}\int_{[0,T]\x[-K,K]^d\x[0,1]^3}
    \Big|\langle H,\Nb^{1,\gammabb^n_t(\pi)}_{\mu^n_t}(x,u,v)\rangle-\langle H,\Nb^{1,\gammabb_t(\pi)}_{\mu_t}(x,u,v)\rangle\Big|
    \,\mathrm{d}t\,\mathrm{d}x\,\mathrm{d}u\,\mathrm{d}v\,\mathrm{d}\pi=0.
\]
By the arbitrariness of $H$, we can use a diagonal extraction and then give (up to a subsequence that we do not rename for simplicity), for a.e.\ $(t,x,u,v,\pi)$ and,
\begin{equation}\label{eq:weak_conv_N}
    \Nb^{1,\gammabb^n_t(\pi)}_{\mu^n_t}(x,u,v)\ \Longrightarrow \Nb^{1,\gammabb_t(\pi)}_{\mu_t}(x,u,v)\quad\text{in weak topology}.
\end{equation}
The same argument applies to the second interaction component, yielding
\[
    \Nb^{2,\gammabb^n_t(\pi)}_{\mu^n_t}(x,u,v)\ \Longrightarrow \Nb^{2,\gammabb_t(\pi)}_{\mu_t}(x,u,v)\quad\text{in weak topology}, \text{ for a.e.\ }(t,x,u,v,\pi).
\]

\medskip
\noindent\emph{Step 4: Joint convergence of the full control tuple.}
Since $\beta^n\to\beta$ a.e., combining \eqref{eq:conv} and \eqref{eq:weak_conv_N} we obtain: for any bounded continuous
\[
    \Phi:[0,1]\x\R^d\x\Hc^2\x A_{\rm reg}\to\R,
\]
\begin{align}
    &\int_{[0,T] \x [0,1]} \Big| \E \Big[\Phi \big(U,  X^n_t, \Nb^{1,\gammabb^n_t(\pi)}_{\mu^n_t}(  X^n_t,U, V),\;  \Nb^{2,\gammabb^n_t(\pi)}_{\mu^n_t}(  X^n_t,U, V),\;\beta^n(t,X^n_t,U,V,\Vt,\pi) \big) \Big] \nonumber
    \\
    &\qquad - \E \Big[\Phi \big(U,  Y_t, \Nb^{1,\gammabb_t(\pi)}_{\mu_t}(  Y_t,U, V),\;  \Nb^{2,\gammabb_t(\pi)}_{\mu_t}(  Y_t,U, V),\;\beta(t,Y_t,U,V,\Vt,\pi) \big) \Big] \Big| \, \mathrm{d}t\, \mathrm{d}\pi \;\longrightarrow\; 0.
    \label{eq:conv_general}
\end{align}
In particular, \eqref{eq:conv_general} yields $\Lambda^n\Rightarrow\Lambda$ in Wasserstein, where
\[
    \Lambda_t=\int_0^1\Lc\!\Bigl(Y_t,U, \Nb^{1,\gammabb_t(\pi)}_{\mu_t}(Y_t,U,V), \Nb^{2,\gammabb_t(\pi)}_{\mu_t}(Y_t,U,V), \beta(t,Y_t,U,V,\Vt,\pi)\Bigr)\,\mathrm{d}\pi.
\]

\medskip
\noindent\emph{Step 5: Identification via Itô and well--posedness.}
Fix $J\in L^\infty([0,1])$ and $\varphi\in C_b^2(\R)$. Applying Itô’s formula to $X^n$ and taking expectations,
\begin{align*}
    \E\big[J(U)\varphi(X^n_t)\big]
    &= \E\big[J(U)\varphi(X_0)\big]
    + \frac{1}{2}\int_0^t \E\big[J(U){\rm Tr}\bigl[\varphi''(X^n_s)\,\sigma(s,X^n_s)\sigma(s,X^n_s)^\top\bigr]\big]\,\mathrm{d}s
    \\
    &\quad + \int_0^t \int_0^1 \E\Big[J(U)\varphi'(X^n_s)\,b\big(s,X^n_s, \Nb^{1,\gammabb^n_s(\pi)}_{\mu^n_s}(X^n_s,U,V), \Nb^{2,\gammabb^n_s(\pi)}_{\mu^n_s}(X^n_s,U,V), \beta^n(s,X^n_s,U,V,\Vt,\pi)\big)\Big]\,\mathrm{d}\pi\,\mathrm{d}s.
\end{align*}
Passing to the limit using \eqref{eq:conv} and \eqref{eq:conv_general} gives
\begin{align*}
    \E\big[J(U)\varphi(Y_t)\big]
    &= \E\big[J(U)\varphi(X_0)\big]
    + \frac{1}{2}\int_0^t \E\big[J(U){\rm Tr} \bigl[\varphi''(Y_s)\,\sigma(s,Y_s) \sigma(s,Y_s)^\top \bigr]\big]\,\mathrm{d}s
    \\
    &\quad + \int_0^t \int_0^1\E\Big[J(U)\varphi'(Y_s)\,b\big(s,Y_s, \Nb^{1,\gammabb_s(\pi)}_{\mu_s}(Y_s,U,V), \Nb^{2,\gammabb_s(\pi)}_{\mu_s}(Y_s,U,V), \beta(s,Y_s,U,V,\Vt,\pi)\big)\Big]\,\mathrm{d}\pi\,\mathrm{d}s,
\end{align*}
i.e., $(Y,U)$ solves the same martingale problem (or Fokker--Planck equation) as the state driven by $(\gammabb,\beta)$ with a given initial distribution $\mu_0^{\gammabb,\beta}=\Lim_{n \to \infty} \mu_0^{\gammabb^n,\beta^n}$. By the uniqueness in law for \eqref{eq:general_relaxed}, we identify $\mu_t=\Lc(Y_t,U)=\mu^{\gammabb,\beta}_t$ for all $t$. As \eqref{eq:conv_general} also identifies the relaxed control component, we conclude
\[
    \bigl(\mu^n,\Lambda^n\bigr)\ \Longrightarrow\ \bigl(\mu^{\gammabb,\beta},\Lambda^{\gammabb,\beta}\bigr)\quad\text{in Wasserstein}.
\]

\medskip
Since every convergent subsequence has the same limit, the full sequence converges. This completes the proof.
\end{proof}

\medskip
We now establish an approximation result that combines almost everywhere convergence with weak convergence of interaction kernels. 
Fix continuous controls $\gammabb \in \Acb_{\rm int}$ and $\beta \in \Acb_{\rm reg}$.

\medskip
We use the shorthand notation
\[
    y=(x,v),\quad z=(x,u,v),\qquad 
    \yh=(\xh,\vh),\quad \zh=(\xh,\uh,\vh).
\]
Define the relaxed interaction kernel $\Gammabb:[0,T]\x(\R^d\x[0,1]^2)^2\x[0,1]\to\Pc(A_{\rm int})$ by
\[
    \Gammabb(t,z,\zh,\pi)(\mathrm{d}e):=\Lc\big(\gammabb(t,z,\zh,\Ub,\pi)\big)(\mathrm{d}e),
\]
where $\Ub\sim{\rm Unif}([0,1])$.

\medskip
Let $p:\R^d\to\R_+$ be the probability density.  
By \Cref{prop:appro_control}, there exists a sequence of continuous maps
\[
    \Gammabb^k : [0,T]\x(\R^d\x[0,1]^2)^2\x[0,1]^2\;\longrightarrow\;\Pc(A_{\rm int}), \qquad k\ge1,
\]
together with, for each $k$, a sequence of continuous selectors
\[
    \gammabb^{k,n} : [0,T]\x(\R^d\x[0,1]^2)^2 \x [0,1]\;\longrightarrow\; A_{\rm int}, \qquad n\ge1,
\]
such that
\begin{align}\label{eq:strong_weak_conv}
    \Gammabb^k \;\Longrightarrow\; \Gammabb \quad \text{in the weak topology as } k\to\infty.
\end{align}

\medskip
Moreover, since $p$ is continuous and integrable, for any bounded continuous test function $F$ we have the following convergence: for every compact $V\subset\R^d$,
\begin{align*}
    &\lim_{k\to\infty}\,\lim_{n\to\infty}\,\sup_{(t,\pi,u,\hat u)}\;\sup_{y\in V\x[0,1]}
    \Bigg|\;
        \int_{\R^d\x[0,1]} F\big(\gammabb^{k,n}(t,z,\zh,\pi),\zh\big)\,p(\xh)\,\mathrm{d}\yh
    \\
    &\hspace{7cm}
        - \int_{\R^d\x[0,1]\x A_{\rm int}} F(e,\zh)\,\Gammabb^k(t,z,\zh,\pi)(\mathrm{d}e)\,p(\xh)\,\mathrm{d}\yh
    \;\Bigg| = 0,
\end{align*}
and symmetrically,
\begin{align*}
    &\lim_{k\to\infty}\,\lim_{n\to\infty}\,\sup_{(t,\pi,u,\hat u)}\;\sup_{\hat y\in V\x[0,1]}
    \Bigg|\;
        \int_{\R^d\x[0,1]} F\big(\gammabb^{k,n}(t,z,\zh,\pi),z\big)\,p(x)\,\mathrm{d}y
    \\
    &\hspace{7cm}
        - \int_{\R^d\x[0,1]\x A_{\rm int}} F(e,z)\,\Gammabb^k(t,z,\zh,\pi)(\mathrm{d}e)\,p(x)\,\mathrm{d}y
    \;\Bigg| = 0.
\end{align*}

\begin{proposition}[Stability under a.e.\ and weak convergence]  \label{prop:a.e.cong_plus_weak} If 
    \begin{align*}
        \lim_{n \to \infty} \mu^{\gammabb^n,\beta^n}_0 =\mu^{\gammabb,\beta}_0\quad\text{in Wasserstein topology}
    \end{align*}
    then
    \[
        \lim_{k\to\infty}\lim_{n\to\infty}\bigl(\mu^{\gammabb^{k,n},\beta},\,\Lambda^{\gammabb^{k,n},\beta}\bigr)
        \;=\;
        \bigl(\mu^{\gammabb,\beta},\,\Lambda^{\gammabb,\beta}\bigr)
        \quad\text{in the Wasserstein topology}.
    \]
\end{proposition}

\begin{proof}

We prove the result by combining regularity estimates for the conditional densities with a diagonal extraction argument, which ensures uniform convergence on compact sets, and by carefully passing to the limit in the interaction kernels.  

\medskip
\noindent\emph{Step 1: Uniform regularity.}  
Fix $k,n\ge1$ and let $(X^{k,n}_t)_{t\in[0,T]}$ be the solution of the McKean--Vlasov SDE
\[
    \mathrm{d}X_t
    = \int_{[0,1]^3} b\!\left(t,X_t,
        \Nb^{1,\gammabb^{k,n}_t(\pi)}_{\mu^{k,n}_t}(X_t,U,v),
        \Nb^{2,\gammabb^{k,n}_t(\pi)}_{\mu^{k,n}_t}(X_t,U,v),
        \beta(t,X_t,U,v,\hat v,\pi)\right)\,
        \mathrm{d}v\,\mathrm{d}\hat v\,
        \mathrm{d}\pi\,\mathrm{d}t
    + \sigma(t,X_t)\,\mathrm{d}W_t,
\]
with $\mu^{k,n}_t=\Lc(X^{k,n}_t,U)$. 
Since $b$ and $\sigma$ are bounded, and $\left( \mu^{n,k}_0 \right)_{n \ge 1}$ is relatively compact for the weak topology, hence $\{\mu^{k,n}\}_{k,n}$ is relatively compact in $C([0,T];\Pc(\R^d \x [0,1]))$. Let $\mu$ denote the limit of a convergent subsequence (notation unchanged).

\medskip
Let us denote by $f^{k,n}_u(t,x)$ the density of the conditional law $\Lc(X^{k,n}_t \mid U = u)$. 
Observe that the map $u \mapsto \Lc(X^{k,n}_t \mid U = u)$ is measurable but is only defined up to $\mathrm{d}u$–null sets, as it represents an equivalence class of measurable functions. 
Consequently, there exists a Borel set $A^{k,n} \subset [0,1]$ with full Lebesgue measure such that the map $u \mapsto \Lc(X^{k,n}_t \mid U = u)$ admits a well--defined representative on $A^{k,n}$. 
In the sequel, whenever we take suprema or evaluate expressions involving this map, it is understood that the operations are performed over the set $A^{k,n}$. This mechanism will be adopted for any measurable map.

\medskip
From parabolic regularity results (see for instance \cite[Theorem 4]{AronsonSerrin67} or \cite[Theorem 6.2.7]{FK-PL-equations} or \cite[Proposition A.1.]{MFD-2020-closed}), for any compact $[s,t]\times Q\subset(0,T)\times\R^d$, there exists $\alpha\in(0,1)$ such that
\begin{align}\label{eq:holder_estimates}
    \sup_{k,n\ge1,\,u\in[0,1]}
    \Bigg(
        \sup_{(r,x)} |f^{k,n}_u(r,x)|
        + \sup_{(r,x)\neq(r',x')}\frac{|f^{k,n}_u(r',x')-f^{k,n}_u(r,x)}{|r-r'|^{\alpha/2}+|x-x'|^\alpha}
    \Bigg)<\infty.
\end{align}

\medskip
\noindent\emph{Step 2: Diagonal extraction.}  
For any $(\hat u_{k,n})_{k,n}\subset[0,1]$, set $F_{k,n}(t,x):=f^{k,n}_{\hat u_{k,n}}(t,x)$. By Arzelà--Ascoli and \eqref{eq:holder_estimates}, there exist subsequences $(k_j)$, $(n_l)$ and $F\in C((0,T)\times\R^d;\R_+)$ such that
\begin{align}\label{eq:conv_uniform}
    \lim_{j\to\infty}\lim_{l\to\infty}F_{k_j,n_l}
    =F\quad\text{uniformly on compacts}.
\end{align}
Consequently, for bounded test functions $(\varphi_{k,n})_{k,n}$ with $\sup_{k,n,t,x}|\varphi_{k,n}(t,x)|<\infty$, 
\begin{align}\label{eq:conv_switch}
    \lim_{K\to\infty}\lim_{j\to\infty}\lim_{l\to\infty}
    \int_{\R^d}\varphi_{k_j,n_l}(t,x)\Big(F_{k_j,n_l}(t,x)-F(t,x)\mathbf{1}_{|x|\le K}\Big)\,\mathrm{d}x=0.
\end{align}

\medskip
\noindent\emph{Step 3: Uniform approximation of kernels.}  
Let $(t_{k,n},\pi_{k,n},z_{k,n},\uh_{k,n})$ approach the supremum of
\begin{align*}
     &\sup_{t,z,\hat u}\;
     \Bigg|\;
        \int_{\R\x[0,1]}H\big(\gammabb^{k,n}(t,z,\zh,\pi),\,\xh,\uh,\,\Gr(u,\uh)\big) f^{k,n}_{\hat u}(t,\xh)\,\mathrm{d}\yh
     \\
     &\quad - \int_{\R\x[0,1]\x A_{\rm int}}H(e,\,\xh,\uh,\,\Gr(u,\uh))\,\Gammabb^k(t,z,\zh,\pi)(\mathrm{d}e)\, f^{k,n}_{\hat u}(t,\xh)\,\mathrm{d}\yh
     \;\Bigg|.
\end{align*}
Using \eqref{eq:conv_switch} along with the properties of the sequence $(\gamma^{k,n})_{n,k \ge 1}$, with $p_{k,n}=(t_{k,n},x_{k,n},u_{k,n},v_{k,n})$,  we obtain:
\begin{align*}
    &\lim_{j \to \infty}\lim_{l \to \infty} \int_{\R^d \times [0,1]} H \left( \gammabb^{k_j,n_l}(  p_{k_j,n_l},  \xh, \uh_{k_j,n_l}, \vh, \pi_{k_j,k_l}),\, \xh, \uh_{k_j,n_l},\, \Gr(u_{k_j,n_l}, \uh_{k_j,n_l}) \right) F_{k_j,n_l}(t, \xh)\, \mathrm{d}\yh
    \\
    &= \lim_{K \to \infty} \lim_{j \to \infty} \lim_{l \to \infty} \int_{\R^d \times [0,1]} H \left( \gammabb^{k_j,n_l}(  p_{k_j,n_l},  \xh, \uh_{k_j,n_l}, \vh, \pi_{k_j,k_l}),\, \xh, \uh_{k_j,n_l},\, \Gr(u_{k_j,n_l}, \uh_{k_j,n_l}) \right) F(t, \xh) \mathbf{1}_{|\hat x| \le K}\, \mathrm{d}\yh
    \\
    &= \lim_{K \to \infty} \lim_{j \to \infty} \lim_{l \to \infty} \int_{\R^d \times [0,1] \times A_{\rm int}} H \left( e,\, \xh, \uh_{k_j,n_l} ,\, \Gr(u_{k_j,n_l}, \uh_{k_j,n_l})\right)\, \Gammabb^{k_j}(p_{k_j,n_l},  \xh,\uh_{k_j,n_l}, \vh, \pi_{k_j,k_l})(\mathrm{d}e)\, F(t, \xh) \mathbf{1}_{|\hat x| \le K}\, \mathrm{d}\yh
    \\
    &= \lim_{j \to \infty}\lim_{l \to \infty} \int_{\R^d \times [0,1] \times A_{\rm int}} H \left( e,\, \hat x,\uh_{k_j,n_l},\, \Gr(u_{k_j,n_l} \uh_{k_j,n_l}) \right)\,\Gammabb^{k_j}(p_{k_j,n_l},  \xh, \uh_{k_j,n_l}, \vh, \pi_{k_j,k_l})(\mathrm{d}e)\, F(t, \xh)\, \mathrm{d}\yh.
\end{align*}

This conclusion holds for any subsequence of $\left(  p_{k,n},\uh_{k,n} \right)_{n,k \ge 1}$, allowing us to deduce that
\begin{align} \label{eq:sup_conv}
    \lim_{k \to \infty}\lim_{n \to \infty} \sup_{t,\pi,z, \hat u} \Bigg| \int_{\R^d \times [0,1]} & H \left( \gammabb^{k,n}(t,  z,  \zh,\pi),\, \xh,\uh,\, \Gr(u, \uh) \right) f^{k,n}_{\hat u}(t, \xh)\,\mathrm{d}\yh \nonumber
    \\
    &\quad - \int_{\R^d \times [0,1] \times A_{\rm int}} H \left( e,\, \xh,\uh,\, \Gr(u, \uh) \right)\, \Gammabb^k(t,  z,  \zh,\pi)(\mathrm{d}e)\, f^{k,n}_{\hat u}(t, \xh)\, \mathrm{d}\yh \Bigg|= 0.
\end{align}

\medskip
\noindent\emph{Step 4: Identification of the limit kernels.}  
For bounded continuous $H$,
\begin{align*}
    \langle H,\Nb^{1,\gammabb^{k,n}_t(\pi)}_{\mu^{k,n}_t}(z)\rangle
    &=\int_{\R\x[0,1]^2}H(\gammabb^{k,n}(t,z,\zh,\pi),\,\xh,\uh,\,\Gr(u,\uh))f^{k,n}_{\hat u}(t,\xh)\,\mathrm{d}\zh
    \\
    &=\int_{\R\x[0,1]^2\x A_{\rm int}}H(e,\xh,\uh,\,\Gr(u,\uh))\,\Gammabb^k(t,z,\zh,\pi)(\mathrm{d}e)f^{k,n}_{\hat u}(t,\xh)\,\mathrm{d}\zh
    +R^{k,n}(t,z,\pi),
\end{align*}
where $R^{k,n}(t,z,\pi)\to0$ by \eqref{eq:sup_conv}.  
Passing to the limit using \eqref{eq:conv_uniform} and $\Gammabb^k\Rightarrow\Gammabb$, we obtain
\[
    \lim_{k\to\infty}\lim_{n\to\infty}\langle H,\Nb^{1,\gammabb^{k,n}_t(\pi)}_{\mu^{k,n}_t}(z)\rangle
    =\langle H,\Nb^{1,\gammabb_t(\pi)}_{\mu_t}(z)\rangle.
\]
Thus, for a.e.\ $(t,z,\pi)$,
\[
    \Nb^{1,\gammabb^{k,n}_t(\pi)}_{\mu^{k,n}_t}(z)\to\Nb^{1,\gammabb_t(\pi)}_{\mu_t}(z).
\]
A symmetric argument yields the same convergence for the second kernel.

\medskip
\noindent\emph{Step 5: Convergence of the full system.}  
The almost everywhere convergence of both interaction kernels, the fact that $\lim_{n \to \infty} \mu^{\gammabb^n,\beta^n}_0 =\mu^{\gammabb,\beta}_0$, together with the convergence of the conditional densities and the techniques used in the proof of \Cref{prop:a.e.cong}, implies that the entire sequence $\{\mu^{k,n}\}_{k,n}$ converges to $\mu$. Consequently,
\[
    \lim_{k\to\infty}\lim_{n\to\infty}\bigl(\mu^{\gammabb^{k,n},\beta},\,\Lambda^{\gammabb^{k,n},\beta}\bigr)
    =(\mu^{\gammabb,\beta},\,\Lambda^{\gammabb,\beta})
    \quad\text{in Wasserstein}.
\]

\end{proof}

\medskip
In this section, we restrict attention to continuous feedback maps of the form 
\[
    (t,x,u,v,\hat x,\hat u,\hat v,\pi)\;\longmapsto\;\gammabb(t,x,u,v,\hat x,\hat u,\hat v,\pi),
    \qquad 
    (t,x,u,v,\tilde v,\pi)\;\longmapsto\;\beta(t,x,u,v,\tilde v,\pi),
\]
with the simplifying assumption that the auxiliary variable $\overline{v}$ does not enter into the definition of $\gammabb$.  
Our objective is to construct approximating sequences that progressively remove the dependence on the randomization variables $v,\hat v,\tilde v, \pi$, so as to recover \emph{strong controls} in the limit.  
This corresponds to replacing randomized relaxed strategies by deterministic measurable selectors while preserving convergence of the associated controlled dynamics.

\begin{proposition}[Approximation by strong controls] \label{prop:strong_approx}
    Let $(\gammabb,\beta)\in\Acb_{\rm int}\times\Acb_{\rm reg}$ be continuous as above and, $(\eta^n)_{n \ge 1} \subset \Pc(\R^d \x [0,1])$ be a sequence verifying $\Lim_{n \to \infty} \eta^n=\eta$ in weak topology for some $\eta$.
    Then there exists a sequence of strong controls $\{(\mu^n,\Lambda^n)\}_{n\ge1}\subset\Pc_S$ such that $\mu^n_0 =\eta^n$,
    \[
        \lim_{n\to\infty} (\mu^n,\Lambda^n) 
        \;=\; (\mu^{\gammabb,\beta},\,\Lambda^{\gammabb,\beta})
        \qquad\text{in the Wasserstein topology}
    \]
    and the initial distribution of $\mu^{\gammabb,\beta}$ is $\mu^{\gammabb,\beta}_0=\eta$.
\end{proposition}

\begin{proof}

We proceed in three steps.  

\medskip
\noindent\emph{Step 1: Construction of deterministic selectors.}  
Let $p:\R^d\to\R_+$ be the probability density introduced in the previous step.  
There exists a sequence of Borel measurable maps
\[
    \Phi_k=(\varphi_k,\widetilde \varphi_k):\R^d\longrightarrow [0,1]^2,\quad c_n:[0,T] \to [0,1], \qquad k,n\ge1,
\]
such that in $\Pc(\R^d\times[0,1]^2)$ and $\Pc([0,1] \x [0,T])$ respectively,
\begin{align} \label{eq:approx_selector}
    \lim_{k\to\infty}\delta_{\Phi_k(x)}(\mathrm{d}v,\mathrm{d}\tilde v)\,p(x)\,\mathrm{d}x
    \;=\; \mathrm{d}v\,\mathrm{d}\tilde v\,p(x)\,\mathrm{d}x
    \quad \mbox{and}\quad \Lim_{n \to \infty}\delta_{c_n(t)}(\mathrm{d}\pi)\,\mathrm{d}t=\mathrm{d}\pi \,\mathrm{d}t.
\end{align}
In other words, the measures generated by $\Phi_k$ approximate the product of the uniform distribution on $[0,1]^2$ with $p(x)\,\mathrm{d}x$ while the measures generated by $c_n$ approximate the product $\mathrm{d}\pi \,\mathrm{d}t$.  
We then define the approximating feedback maps
\[
    \gamma^{n,k}(t,x,u,\hat x,\hat u):=\gammabb\!\left(t,x,u,\varphi_k(x),\hat x,\hat u,\varphi_k(\hat x),c_n(t)\right),
    \qquad
    \beta^{n,k}(t,x,u):=\beta\!\left(t,x,u,\Phi_k(x),c_n(t)\right).
\]

\medskip
\noindent\emph{Step 2: Convergence of the interaction terms.}  
Let $F$ be a bounded continuous function.  
By construction of $\Phi_k$, for every compact $V\subset\R^d$ we have the uniform convergence
\begin{align}\label{eq:sup-conv_1_better}
    &\lim_{n\to\infty}\lim_{k\to\infty}\;\sup_{(t,x,u,\hat u)\in[0,T]\times V\times[0,1]^2}
    \Bigg|
        \int_{\R^d} F\!\left(\gamma^{n,k}(t,x,u,\hat x,\hat u),t,x,u,\Phi_k(x),\hat x,\hat u\right)\,p(\hat x)\,\mathrm{d}\hat x
    \\
    &\hspace{3cm}
        - \int_{\R^d\times[0,1]}F\!\left(\gammabb(t,x,u,\varphi_k(x),\hat x,\hat u,\hat v, c_n(t)),t,x,u,\Phi_k(x),\hat x,\hat u\right)\,\mathrm{d}\hat v\,p(\hat x)\,\mathrm{d}\hat x
    \Bigg|=0, \nonumber
\end{align}
and similarly
\begin{align}\label{eq:sup-conv_2_better}
    &\lim_{n\to\infty}\lim_{k\to\infty}\;\sup_{(t,u,\hat x,\hat u)\in[0,T]\times[0,1]\times V\times[0,1]}
    \Bigg|
        \int_{\R^d} F\!\left(\gamma^{n,k}(t,x,u,\hat x,\hat u),t,x,u,\Phi_k(x),\hat x,\hat u\right)\,p(x)\,\mathrm{d}x
    \\
    &\hspace{3cm}
        - \int_{\R^d\times[0,1]}F\!\left(\gammabb(t,x,u,v,\hat x,\hat u,\varphi_k(\hat x),c_n(t)),t,x,u,\Phi_k(x),\hat x,\hat u\right)\,\mathrm{d}v\,p(x)\,\mathrm{d}x
    \Bigg|=0. \nonumber
\end{align}
Indeed, let \(\left( t_{n,k}, x_{n,k}, u_{n,k}, \hat u_{n,k} \right)_{n,k \ge 1} \subset [0,T]\times V \times [0,1]^2\) be a sequence achieving the essential supremum in \eqref{eq:sup-conv_1_better}. Since the domain \([0,T] \times V \times [0,1]^3\) is compact, any continuous function on this set is uniformly continuous. Also, there exists a subsequence \((n_l,k_j,)_{j,l \ge 1}\) such that
\[
    \lim_{l \to \infty}\lim_{j \to \infty} \left( t_{n_l,k_j}, x_{n_l,k_j}, u_{n_l,k_j}, \Phi_{k_j}(x_{n_l,k_j}), \hat u_{n_l,k_j}, c_{n_l}(t_{n_l,k_j}) \right) = \left( t, x, u, \Phi=(\varphi,\widetilde{\varphi}), \hat u, \pi \right).
\]

We then have, with $q_{n,k}:=(t_{n,k},x_{n,k},u_{n,k})$,
\begin{align*}
    &\lim_{l \to \infty}\lim_{j \to \infty} \int_{\R^d} F\left(\gammabb(q_{n_l,k_j}, \varphi_{k_j}(x_{n_l,k_j}), \hat x, \hat u_{n_l,k_j}, \varphi_{k_j}(\hat x), c_{n_l}(t_{n_l,k_j})),\, q_{n_l,k_j}, \Phi_{k_j}(x_{n_l,k_j}), \hat x, \hat u_{n_l,k_j}\right) p(\hat x)\, \mathrm{d}\hat x \\
    &= \lim_{l \to \infty}\lim_{j \to \infty} \int_{\R^d} F\left(\gammabb(t, x, u, \varphi_{k_j}(x_{n_l,k_j}), \hat x, \hat u, \varphi_{k_j}(\hat x),\pi),\, t, x, u,\Phi, \hat x, \hat u\right) p(\hat x)\, \mathrm{d}\hat x \\
    &= \lim_{j \to \infty} \int_{\R^d} F\left(\gammabb(t, x, u, \varphi, \hat x, \hat u, \varphi_{k_j}(\hat x),\pi),\, t, x, u, \Phi, \hat x, \hat u\right) p(\hat x)\, \mathrm{d}\hat x \\
    &= \int_{\R^d \times [0,1]} F\left(\gammabb(t, x, u, \varphi, \hat x, \hat u, \hat v,\pi),\, t, x, u, \Phi, \hat x, \hat u\right) \mathrm{d}\hat v\, p(\hat x)\, \mathrm{d}\hat x
\end{align*}
where the last equality follows from \eqref{eq:approx_selector}.
Since this limit holds for any subsequence, we deduce the convergence in \eqref{eq:sup-conv_1_better}. A similar argument yields the convergence in \eqref{eq:sup-conv_2_better}.

\medskip
\noindent\emph{Step 3: Convergence of the controlled dynamics.}  
Let $(X^{n,k}_t)_{t\in[0,T]}$ be the solution to: $\Lc(X^{n,k}_0,U)=\eta^k$ and
\[
    \mathrm{d}X^{n,k}_t
    = b\Big(t,X^{n,k}_t,N^{1,\gamma^{n,k}_t}_{\mu^{n,k}_t}(X^{n,k}_t,U),N^{2,\gamma^{n,k}_t}_{\mu^{n,k}_t}(X^{n,k}_t,U),\beta^{n,k}(t,X^{n,k}_t,U)\Big)\,\mathrm{d}t
    + \sigma(t,X^{n,k}_t)\,\mathrm{d}W_t,
\]
with marginal $\mu^{n,k}_t=\Lc(X^{n,k}_t,U)$.  
By \eqref{eq:sup-conv_1_better} and \eqref{eq:sup-conv_2_better}, together with the arguments developed in the proof of \Cref{prop:a.e.cong_plus_weak}  (uniform regularity of densities and stability of McKean--Vlasov SDEs under weak convergence), we deduce that the sequence $(t\mapsto\mu^{n,k}_t)_{n,k\ge1}$ converges in $C([0,T];\Pc(\R^d\times[0,1]))$ to $t\mapsto\mu_t=\Lc(X_t,U)$, where $X$ is the solution associated with $(\gammabb,\beta)$ starting at the distribution $\eta$.

\medskip
\noindent\emph{Conclusion.}  
Thus, the sequence of strong controls $(\mu^{n,k},\Lambda^{n,k})$ associated with $(\gamma^{n,k},\beta^{n,k})$ with initial distribution $\eta^k$ converges in the Wasserstein topology to the relaxed control $(\mu^{\gammabb,\beta},\Lambda^{\gammabb,\beta})$ with initial distribution $\eta$. This establishes the proposition.

\end{proof}

\medskip
At this stage, recall that the set $\Pc_R$ of relaxed controls has already been shown to be closed under the Wasserstein topology (\Cref{prop:closed_relaxed}). 
We also established three key approximation results:  
\Cref{prop:a.e.cong} ensures stability under almost everywhere convergence of the control maps,  
\Cref{prop:a.e.cong_plus_weak} extends this stability to mixed regimes combining almost everywhere and weak convergence,  
and \Cref{prop:strong_approx} shows how specific relaxed controls can be approximated by strong ones.  

\medskip
By combining these ingredients, we can pass from general relaxed controls to limits of strong controls. 
This yields the following density property.

\begin{proposition} \label{prop:strong_relaxed_density}
    The set of strong controls $\Pc_S$ is dense in the set of relaxed controls $\Pc_R$ with respect to the weak convergence topology.
\end{proposition}

\subsection{Proofs of \Cref{thm:equivalence_relaxed_strong}, \Cref{prop:existence_continuity} and \Cref{prop:equiv_open}   }

\subsubsection{Proof of \Cref{thm:equivalence_relaxed_strong}}

Let $\nu \in \Pc_p(\R^d)$ with $p \in \{0\} \cup [1,\infty)$. 
By a direct application of {\rm\Cref{prop:closed_relaxed}} with the initial distribution fixed at $\nu$, we obtain that the set $\Pib(\nu)$ is closed under the Wasserstein topology $\Wc_p$ (recalling that for $p = 0$, this corresponds to the weak topology). 
Furthermore, applying {\rm\Cref{prop:strong_relaxed_density}} yields that $\Pi(\nu)$ is dense in $\Pib(\nu)$ with respect to $\Wc_p$.

\medskip
Let us now establish the equivalence between the randomized/relaxed and the strong formulations. 
This result follows as a direct consequence of the previous propositions. 
We first observe that $V_{\rm MFC}(\nu) \le \Vb_{\rm MFC}(\nu)$, and notice that
\[
    \Vb_{\rm MFC}(\nu)
    :=
    \sup_{\eta \in \Pc_\nu} \sup_{(\gammabb,\beta) \in \Acb_{\rm int} \times \Acb_{\rm reg},\; \mu^{\gammabb,\beta}_0=\eta}
    \Gc\left( \mu^{\gammabb,\beta}, \Lambda^{\gammabb,\beta} \right),
\]
where, for any $(\varrho,q) \in \Cc_{\rm pr}^d \times \M(\R^d \times [0,1] \times \Rc)$,
\begin{align*}
    \Gc(\varrho,q)
    &:=
    \int_0^T \int_{\R^d \times [0,1] \times \Rc}
        L\!\left(s,x,u,r_1,r_2,a\right)
        q_s(\mathrm{d}x,\mathrm{d}u,\mathrm{d}r_1,\mathrm{d}r_2,\mathrm{d}a)\,\mathrm{d}s
    \\
    &\quad + \int_{\R^d \times [0,1]}
        g\!\left(x,
            \int_{\R^d \times [0,1]}
                \delta_{\left(y,\Gr(u,v)\right)}(\mathrm{d}x',\mathrm{d}e)\,
                \varrho_T(\mathrm{d}y,\mathrm{d}v)
        \right)\varrho_T(\mathrm{d}x,\mathrm{d}u).
\end{align*}

For any $(\gammabb,\beta) \in \Acb_{\rm int} \times \Acb_{\rm reg}$ with $\mu^{\gammabb,\beta}_0 = \eta$, 
the density of strong controls (see {\rm\Cref{prop:strong_relaxed_density}}) ensures the existence of a sequence 
$(\gamma^n,\beta^n)_{n \ge 1} \subset \Ac_{\rm int} \times \Ac_{\rm reg}$ such that 
$\mu^{\gamma^n,\beta^n}_0 = \eta$ and 
\[
    \Lim_{n \to \infty} 
    \left( \mu^{\gamma^n,\beta^n}, \Lambda^{\gamma^n,\beta^n} \right)
    =
    \left( \mu^{\gammabb,\beta}, \Lambda^{\gammabb,\beta} \right)
    \quad \text{in } \Wc_p.
\]
Under {\rm\Cref{assum:main1_MF_CI}}, the continuity of $L$ and $g$, together with the weak and stable convergence of the associated measures, yields
\[
    \Lim_{n \to \infty}
    \Gc\!\left( \mu^{\gamma^n,\beta^n}, \Lambda^{\gamma^n,\beta^n} \right)
    =
    \Gc\!\left( \mu^{\gammabb,\beta}, \Lambda^{\gammabb,\beta} \right),
\]
where the use of stable convergence is required to handle the potential discontinuity of the kernel $\Gr$, since for each $n \ge 1$ we have $\mu^{\gamma^n,\beta^n}_t(\R^d \times \mathrm{d}u) = \mathrm{d}u$. 
Consequently, we obtain 
\[
    \Gc\!\left( \mu^{\gammabb,\beta}, \Lambda^{\gammabb,\beta} \right) 
    \le V_{\rm MFC}(\nu),
\]
for any $(\gammabb,\beta) \in \Acb_{\rm int} \times \Acb_{\rm reg}$, which finally implies that
\[
    V_{\rm MFC}(\nu) = \Vb_{\rm MFC}(\nu).
\]

\subsubsection{Proof of \Cref{prop:existence_continuity}}
By combining the boundedness of $b$ and $\sigma$ with {\rm\Cref{prop:closed_relaxed}}, we deduce that for any $\nu \in \Pc_p(\R^d)$ with $p \in \{0\} \cup [1,\infty)$, the set of randomized (or relaxed) controls $\Pib(\nu)$ is not only closed but also compact for the weak topology. 
The compactness property essentially follows from two key ingredients: 
first, all admissible measures share the same fixed marginal in the state component, namely 
$\mu^{\gammabb,\betabb}_0(\mathrm{d}x,[0,1]) = \nu(\mathrm{d}x)$; 
second, the coefficients $(b,\sigma)$ are uniformly bounded. 
The fixed initial distribution prevents any dispersion at time~$0$, while the boundedness of the dynamics ensures uniform moment and tightness estimates, jointly yielding the desired compactness.

Since both $L$ and $g$ are continuous and bounded, the functional 
\[
    (\varrho,q) \longmapsto \Gc(\varrho,q)
\]
is continuous on $\Pib(\nu)$ under the weak topology. 
Consequently, the quantity $\Vb_{\rm MFC}(\nu)$, defined as the supremum of $\Gc$ over the compact set $\Pib(\nu)$, achieves its maximum. 
We therefore obtain the existence of an optimal control, as stated in the proposition.

\medskip
Let us now establish the continuity property of the value function. 
Let $(\nu^n)_{n \ge 1} \subset \Pc_p(\R^d)$ be a sequence such that $\Lim_{n \to \infty} \nu^n = \nu$ in $\Wc_p$ for some $\nu \in \Pc_p(\R^d)$. 
For each $n \ge 1$, let $(\eta^n, \gammabb^n, \beta^n)$ be a $1/n$--optimal control for $\Vb_{\rm MFC}(\nu^n)$, that is,
\begin{align*}
    \Vb_{\rm MFC}(\nu^n) 
    \le 
    \Jb(\eta^n, \gammabb^n, \beta^n) + \frac{1}{n}.
\end{align*}
The convergence $\nu^n \to \nu$ in $\Wc_p$ implies that $\sup_{n \ge 1} \Wc_p(\nu^n, \nu_1) < \infty$.
Let us consider the sequence $\bigl(\mu^{\gammabb^n, \beta^n}, \Lambda^{\gammabb^n, \beta^n}\bigr)_{n \ge 1}$, which belongs to the set
\begin{align*}
    \widehat{\Pi}
    :=
    \left\{
        (\mu, \Lambda) \text{ relaxed control such that } 
        \Wc_p\bigl( \mu_0(\mathrm{d}x' \x [0,1]), \nu_1 \bigr)
        \le 
        \sup_{n \ge 1} \Wc_p(\nu^n, \nu_1)
        < \infty
    \right\}.
\end{align*}
Using the boundedness of $b$ and $\sigma$, it follows that the set $\widehat{\Pi}$ is not only closed but also compact for the weak topology. 
Hence, there exists a subsequence $(n_k)_{k \ge 1}$ and a relaxed control $(\widetilde{\mu}, \widetilde{\Lambda})$ such that
\[
    \Lim_{k \to \infty} 
    \bigl(
        \mu^{\gammabb^{n_k}, \beta^{n_k}},
        \Lambda^{\gammabb^{n_k}, \beta^{n_k}}
    \bigr)
    =
    (\widetilde{\mu}, \widetilde{\Lambda})
    \quad \text{for the weak topology.}
\]
In particular, we have $\widetilde{\mu}_0(\mathrm{d}x \x [0,1]) = \nu$, and therefore $(\widetilde{\mu}, \widetilde{\Lambda}) \in \Pib(\nu)$. 
Using the continuity of $\Gc$ with respect to weak convergence, we obtain
\begin{align*}
    \limsup_{n \to \infty} \Vb_{\rm MFC}(\nu^n)
    \le 
    \limsup_{n \to \infty} \Jb(\eta^n, \gammabb^n, \beta^n)
    = 
    \lim_{k \to \infty}
    \Gc\bigl(
        \mu^{\gammabb^{n_k}, \beta^{n_k}},
        \Lambda^{\gammabb^{n_k}, \beta^{n_k}}
    \bigr)
    =
    \Gc(\widetilde{\mu}, \widetilde{\Lambda})
    \le 
    \Vb_{\rm MFC}(\nu).
\end{align*}

Conversely, let $(\gamma, \beta) \in \Pi(\nu)$ be a continuous control. 
By {\rm\Cref{prop:strong_approx}}, there exists a sequence $(\mu^n, \Lambda^n)_{n \ge 1}$ such that for each $n$, $(\mu^n, \Lambda^n) \in \Pi(\nu^n)$ and 
\[
    \Lim_{n \to \infty} (\mu^n, \Lambda^n)
    =
    (\mu^{\gamma, \beta}, \Lambda^{\gamma, \beta})
    \quad \text{for the weak topology.}
\]
By the continuity of $\Gc$, this yields
\[
    \Gc\bigl(\mu^{\gamma, \beta}, \Lambda^{\gamma, \beta}\bigr)
    =
    \Lim_{n \to \infty} \Gc\bigl(\mu^n, \Lambda^n\bigr)
    \le 
    \liminf_{n \to \infty} \Vb_{\rm MFC}(\nu^n).
\]
Since the supremum defining $V_{\rm MFC}(\nu)$ can be taken over continuous controls $(\gamma, \beta)$, we deduce that
\[
    V_{\rm MFC}(\nu)
    \le 
    \liminf_{n \to \infty} V_{\rm MFC}(\nu^n).
\]
Combining the two inequalities and using the equality $V_{\rm MFC} = \Vb_{\rm MFC}$, we conclude that
\[
    \Lim_{n \to \infty} V_{\rm MFC}(\nu^n)
    =
    V_{\rm MFC}(\nu),
\]
which proves the desired continuity property.

\subsubsection{Proof of \Cref{prop:equiv_open}}
Let $\left( \mu^{\circ,\gamma,\beta},\,\Lambda^{\circ,\gamma,\beta} \right) \in \Pi^\circ(\nu)$. 
It is straightforward to verify that $\left( \mu^{\circ,\gamma,\beta},\,\Lambda^{\circ,\gamma,\beta} \right)$ satisfies the admissibility conditions required in the relaxed formulation introduced previously. 
To conclude that $\left( \mu^{\circ,\gamma,\beta},\,\Lambda^{\circ,\gamma,\beta} \right) \in \Pib(\nu)$, it remains to check that, for almost every $t$, 
\[
    \Lambda^{\circ,\gamma,\beta}_t(\mathrm{d}x,\mathrm{d}u,\mathrm{d}r_1,\mathrm{d}r_2, A_{\rm reg}) 
    \in \Mcb_{\rm aux} = \Mcb.
\]

\medskip
To that end, we define the process $(X_t)_{t \in [0,T]}$ by setting $\Lc(X_0,U) = \Lc(X^{\gamma,\beta,u}_0)(\mathrm{d}x)\mathrm{d}u$ and letting
\begin{align*}
    \mathrm{d}X_t 
    &=
    b \Bigl(
        t, X_t, 
        R^1(t,X_0,U,W_{t \wedge \cdot}),
        R^2(t,X_0,U,W_{t \wedge \cdot}), 
        \beta(t,X_0,U,W_{t \wedge \cdot})
    \Bigr) \mathrm{d}t
    + 
    \sigma(t,X_t)\mathrm{d}W_t,
\end{align*}
where $R = (R^1,R^2)$ is defined for all $(t,x,u,w)$ by
\begin{align*}
    R^1(t,x,u,w)
    &:= \Lc\Bigl( 
        \gamma\bigl(t,x,u,w,X_0,U,W_{t \wedge \cdot}\bigr),\, 
        X_t, U,\, \Gr(u,U)
    \Bigr), 
    \\
    R^2(t,x,u,w)
    &:= \Lc\Bigl( 
        \gamma\bigl(t,X_0,U,W_{t \wedge \cdot},x,u,w\bigr),\, 
        X_t, U,\, \Gr(u,U)
    \Bigr).
\end{align*}

Since 
\[
    \Lc\bigl(X_0^{\gamma,\beta,u},W^u\bigr)(\mathrm{d}x,\mathrm{d}w)\mathrm{d}u
    =
    \Lc(X_0, W, U)(\mathrm{d}x,\mathrm{d}w,\mathrm{d}u),
\]
we deduce that 
\[
    \Lc(X^{\gamma,\beta,u}, W^u) 
    = 
    \Lc(X, W \mid U = u)
\]
for almost every $u \in [0,1]$. 
This yields the identification $\mu = \mu^{\circ,\gamma,\beta}$ and $\Lambda^{\circ,\gamma,\beta} = \Lambda$, where 
\[
    \mu_t := \Lc(X_t, U),
    \quad
    \Lambda := 
    \Lc\Bigl( 
        X_t, U,
        R^1(t,X_0,U,W_{t \wedge \cdot}),
        R^2(t,X_0,U,W_{t \wedge \cdot}),
        \beta(t,X_0,U,W_{t \wedge \cdot})
    \Bigr)
    (\mathrm{d}x,\mathrm{d}u,\mathrm{d}r_1,\mathrm{d}r_2,\mathrm{d}a)\mathrm{d}t.
\]

\medskip
Let $(\Xt_0, \Ut, \Wt)$ be an independent copy of $(X_0, U, W)$ with the same law, i.e. $\Lc(\Xt_0,\Ut,\Wt) = \Lc(X_0,U,W)$. 
Define $(\Xt_t)_{t \in [0,T]}$ as the solution of the same SDE as $(X_t)$, but driven by $(\Xt_0,\Ut,\Wt)$ instead of $(X_0,U,W)$. 
We then introduce, for each $t \in [0,T]$, the auxiliary variable
\begin{align*}
    \Gamma_t 
    := 
    \Bigl(
        \gamma(t, X_0,U,W_{t \wedge \cdot}, \Xt_0,\Ut,\Wt_{t \wedge \cdot}),
        (X_t,U),\, 
        G(t,X_0,U,W_{t \wedge \cdot}),
        (\Xt_t,\Ut),\, 
        G(t,\Xt_0,\Ut,\Wt_{t \wedge \cdot})
    \Bigr).
\end{align*}
By standard arguments, one can verify that $\Gamma_t$ satisfies the compatibility condition required in the definition of $\Mcb_{\rm aux}$. 
Consequently,
\[
    \Lambda^{\circ,\gamma,\beta}_t(\mathrm{d}x,\mathrm{d}u,\mathrm{d}r_1,\mathrm{d}r_2, A_{\rm reg})
    =
    \Lambda_t(\mathrm{d}x,\mathrm{d}u,\mathrm{d}r_1,\mathrm{d}r_2, A_{\rm reg})
    \in \Mcb_{\rm aux} = \Mcb.
\]
We thus conclude that 
\[
    \bigl( \mu^{\circ,\gamma,\beta}, \Lambda^{\circ,\gamma,\beta} \bigr) \in \Pib(\nu).
\]

\medskip
Let $(\gamma,\alpha) \in \Ac_{\rm int} \times \Ac_{\rm reg}$ be Lipschitz maps.  
Under {\rm\Cref{assum:main1_MF_CI}}, given the random variables $(X_0,U)$ and the Brownian motion $W$, the SDE satisfied by $X^{\gamma,\alpha}$ admits a unique strong solution.  
Hence, there exists a Borel measurable function 
\[
    F : [0,T] \times \R^d \times [0,1] \times \Cc^d \longrightarrow \R^d
\]
such that, for every $t \in [0,T]$,
\[
    X^{\gamma,\alpha}_t = F\bigl(t, X_0, U, W_{t \wedge \cdot}\bigr)
    \quad \text{a.s.}
\]

\medskip
We now define the pair $(\gamma^\circ,\alpha^\circ) \in \Ac^\circ_{\rm int} \times \Ac^\circ_{\rm reg}$ by
\begin{align*}
    \gamma^\circ\!\left(t,(x,u,w),(x',u',w')\right)
    &:=
    \gamma\!\left(
        t,\,
        F\bigl(t,x,u,w_{t \wedge \cdot}\bigr),\,u,\,
        F\bigl(t,x',u',w'_{t \wedge \cdot}\bigr),\,u'
    \right),
    \\
    \alpha^\circ(t,x,u,w)
    &:=
    \alpha\!\left(
        t,\,
        F\bigl(t,x,u,w_{t \wedge \cdot}\bigr),\,u
    \right).
\end{align*}

\medskip
By construction, for almost every $u \in [0,1]$, the processes $X^{\gamma,\alpha}$ and $X^{\gamma^\circ,\alpha^\circ,u}$ have the same conditional law, that is,
\[
    \Lc\bigl(X^{\gamma,\alpha} \mid U=u\bigr)
    =
    \Lc\bigl(X^{\gamma^\circ,\alpha^\circ,u}\bigr)\quad\mbox{whenever }\quad\Lc\bigl(X^{\gamma,\alpha}_0 \mid U=u\bigr)
    =
    \Lc\bigl(X^{\gamma^\circ,\alpha^\circ,u}_0\bigr).
\]
Consequently, we obtain the equalities
\[
    \bigl(\mu^{\circ,\gamma^\circ,\alpha^\circ},\,\Lambda^{\circ,\gamma^\circ,\alpha^\circ}\bigr)
    =
    \bigl(\mu^{\gamma,\alpha},\,\Lambda^{\gamma,\alpha}\bigr),
\]
and thus,
\[
    \bigl(\mu^{\gamma,\alpha},\,\Lambda^{\gamma,\alpha}\bigr)
    =
    \bigl(\mu^{\circ,\gamma^\circ,\alpha^\circ},\,\Lambda^{\circ,\gamma^\circ,\alpha^\circ}\bigr)
    \in 
    \Pi^\circ\!\bigl(\Lc(X^{\gamma,\alpha}_0)\bigr).
\]

\medskip
This implies that
\[
    J\!\left(\Lc(X^{\gamma,\alpha}_0,U),\,\gamma,\alpha\right)
    \le 
    V^\circ_{\mathrm{MFC}}\!\left(\Lc(X^{\gamma,\alpha}_0)\right).
\]
Since the supremum in $V_{\mathrm{MFC}}$ can be taken over Lipschitz controls, we deduce that 
\[
    V_{\mathrm{MFC}} \le V^\circ_{\mathrm{MFC}}.
\]
Moreover, from the previous proof, we already know that $V^\circ_{\mathrm{MFC}} \le \Vb_{\mathrm{MFC}}$.  
Using the equality $V_{\mathrm{MFC}} = \Vb_{\mathrm{MFC}}$, we finally obtain
\[
    V_{\mathrm{MFC}} = V^\circ_{\mathrm{MFC}},
\]
which concludes the proof.

\subsection{From strong formulation to {\it{n}}--particle: proof of \Cref{thm:strong_to_n} }

Let us briefly recall the framework. We consider
\[
(t, u, x, \hat u, \hat x) \mapsto \left( \gamma(t, u, x, \hat u, \hat x),\; \beta(t, u, x) \right)
\]
a Lipschitz continuous and bounded map together with the process $X^{\gamma,\beta}$. We choose to directly work with Lipschitz control because of the almost surely stability proved in \Cref{prop:a.e.cong}.  

\medskip
We also recall that
we introduced the collections \(\gammab^n := (\gamma_{ij}^n)_{1 \le i,j \le n}\) and \(\betab^n := (\beta^{1,n}, \dots, \beta^{n,n})\), as follows:
\begin{align*}
    \beta^{i,n}(t,x_1,\dots,x_n) &:= \beta(t, x_i, u^i_n), \\
    \gamma^n_{ij}(t,x_1,\dots,x_n) &:= \gamma(t, x_i, u^i_n, x_j, u^j_n),
\end{align*}
for all \((t, x_1, \dots, x_n) \in [0,T] \times (\R^d)^n\), where \(u^i_n := \frac{i}{n}\) for \(i = 1, \dots, n\).
Let \(\Xbb^n := (X^{1,n}, \dots, X^{n,n})\) denote the solution to \eqref{eq:n_particle}, associated with the interaction kernels \(\gammab^n\) and controls \(\betab^n\) with $\Lc(\Xbb^n_0)=\nu^n$. 

\begin{proof}
    {\it{Step 1: Compactness and tightness.}} Let 
    $$
        \Zc := C([0,T]; \R^d) \x [0,1] \x C \left( [0,T];\R^d \x [0,1]^2 \x \Er  \right) \x C \left( [0,T];\R^d \x [0,1]^2  \right) \x C \left( [0,T]; \R^d \x [0,1]^2 \right)
    $$
    and define for \(n \ge 1\),
\[
    \Pr^n := \frac{1}{n} \sum_{i=1}^n 
    \Lc \Big( X^i,\,u^i_n,\, \muh^{i,n},\, \mub^n,\, \mu^n \Big)
    \in \Pc(\Zc),
\]
where, for \(t \in [0,T]\),
\[
    \muh^{i,n}_t := \sum_{j=1}^n 
    \delta_{(X^j_t,\,u^j_n )}(\mathrm{d}x,\mathrm{d}u)\,\mathbf{1}_{\{ v \in I^n_i \}}\mathrm{d}v \,\delta_{\xi^n_{ij}}( \mathrm{d}e)
    ,
    \quad
    \mub^n_t :=  \sum_{i=1}^n 
    \delta_{\bigl( X^i_t,\,u^i_n \bigr)}(\mathrm{d}x,\mathrm{d}u)\, \mathbf{1}_{\{ v \in I^n_i \}}\mathrm{d}v,
\]
\[
    \mu^n_t(\mathrm{d}x,\mathrm{d}u) := \mub^n_t( \mathrm{d}x,\mathrm{d}u \x [0,1])=\frac{1}{n}\sum_{i=1}^n 
    \delta_{\bigl( X^i_t,\,u^i_n \bigr)}(\mathrm{d}x,\mathrm{d}u),
    \qquad
    u^i_n := \frac{i}{n}, \quad I^n_i:=\bigl(u^{i-1}_n, u^i_n \bigr], \quad 1 \le i \le n.
\]
Since the sequence of initial distributions $\bigl( \mu^{n,\gammab^n,\betab^n}_0 \bigr)_{n \ge 1}$ converges in $\Wc_p$, it is then relatively compact in $\Wc_p$.
Combined with the boundedness of $(b,\sigma)$, it follows from the same techniques as in the proof of \Cref{prop:a.e.cong} that the sequence \((\Pr^n)_{n \ge 1}\) is relatively compact in \(\Wc_p\).  
Let \(\Pr\) be the limit of a convergent subsequence; for simplicity, we use the same notation for the sequence and its subsequence.  
We can write
\[
    \Pr = \P \circ (X, U, \muh, \mub, \mu)^{-1},
\]
where \((U, X, \muh, \mub, \mu)\) are defined on the original probability space \((\Omega, \F, \P)\) and  
\((X_t, \muh_t, \mub_t, \mu_t)_{t \in [0,T]}\) is \(\F\)-predictable.

\medskip
{\it{Step 2: Identification of interaction terms through $(\mub,\muh)$.}} 

\medskip
By using the weak convergence, it is immediate that, \(\P\)--a.e., for all \(t \in [0,T]\),
\begin{align} \label{eq:relation_mu}
    \mu_t(\mathrm{d}x, \mathrm{d}u) = \mub_t( \mathrm{d}x, \mathrm{d}u,[0,1]), 
    \qquad 
    \mub_t (\mathrm{d}x,\mathrm{d}u, \mathrm{d}v) = \muh_t( \mathrm{d}x,\mathrm{d}u, \mathrm{d}v, \Er).
\end{align}

\medskip

For any bounded continuous functions \(f\) and \(L\), we have
\[
    \E \big[ f(X_t,U)\, L(\mub) \big]
    = \lim_{n \to \infty} \frac{1}{n} \sum_{i=1}^n \E \big[ f(X^i_t,u^i_n)\, L(\mub^n) \big]
    = \E \big[ \langle f, \mu_t \rangle \, L(\mub) \big].
\]

Since this identity holds for all \((f,L)\) and all \(t \in [0,T]\), it follows that
\begin{align} \label{eq:FP}
    \mu_t = \Lc(X_t, U \mid \mub) \quad \text{a.e.}
\end{align}

By an application of \Cref{prop:limit_kernel}, if \(( \Xh,\Uh, \Vh, \Eh)\) denotes the canonical variable on \( \R^d \times [0,1]^2 \times \Er\), we obtain that, \(\P\)--a.e.\ \(\om\), for all \(t \in [0,T]\),
\[
    \muh_t(\om) \circ \bigl( \Xh,\Uh, \Vh, \Eh  \bigr)^{-1} = \muh_t(\om) \circ \big( \Xh,\, \Uh,\, \Uh,\, \Gr(U(\om), \Uh) \big)^{-1}.
\]

Consequently, by using the relationship of $\mu$, $\mub$ and $\muh$ mentioned in \eqref{eq:relation_mu} combined with \eqref{eq:FP}, we conclude that, \(\P\)--a.e.\ \(\om\), for all \(t \in [0,T]\),
\[
    \muh_t(\om) = \Lc \Big(  X_t,\,U,\, U,\, \Gr(U(\om), U) \;\big|\; \mub \Big).
\]

{\it{Step 3: Identification of the limit through $\mu$}}

\medskip
Notice that, we can express, for each \(1 \le i \le n\),
\begin{align} \label{eq:first_N}
    M^{1,i,n}_{\gamma,t} = \Nc^1(t, X^i_t,u^i_n, \muh^{i,n}_t),
    \quad 
    \Nc^1(t,x,u,m) := \Lc\big( \gamma(t,x,u, X^m, U^m),\, X^m,\,U^m,\, \, E^m\big),
\end{align}
and
\begin{align} \label{eq:second_N}
    M^{2,i,n}_{\gamma,t} = \Nc^2(t, X^i_t,u^i_n, \muh^{i,n}_t),
    \quad
    \Nc^2(t,x,u,m) := \Lc\big( \gamma(t, X^m,U^m, x,u),\, X^m,\,U^m,\, E^m \big),
\end{align}
where \(( X^m,U^m, V^m, E^m)\) is a canonical variable such that \(m = \Lc( X^m,U^m, V^m, E^m)\).  

\medskip
Hence, for any bounded smooth \(\varphi :  \R^d \times [0,1] \to \R\),
\begin{align*} 
    \mathrm{d}\langle \varphi, \mu^n_t \rangle
    &=  \frac{1}{n} \sum_{i=1}^n \partial_x \varphi(X^i_t,u^i_n) \,
       b\big( t, X^i_t, \Nc^1(t,X^i_t,u^i_n,\muh^{i,n}_t), \Nc^2(t,X^i_t,u^i_n,\muh^{i,n}_t), \beta(t,X^i_t,u^i_n) \big)
       \, \mathrm{d}t \nonumber
    \\
    &\quad+ \frac12 \int_{\R^d \times [0,1]} {\rm Tr}\left(\partial_{xx}^2 \varphi(x,u) \,
       \sigma(t,x) \sigma(t,x)^\top \right) \, \mu^n_t(\mathrm{d}x,\mathrm{d}u) \, \mathrm{d}t
       + \mathrm{d}M^n_t,
\end{align*}
where
\[
    M^n_\cdot := \frac{1}{n} \sum_{i=1}^n \int_0^\cdot
    \partial_x \varphi(X^i_s,u^i_n) \, \sigma(s,X^i_s) \, \mathrm{d}W^i_s.
\]
By using the boundedness of $\sigma$ and $\varphi$, observe that $\lim_{n \to \infty} \E \left[ \sup_{t \in [0,T]} |M^n_t|^2 \right]=0$. 
Combining the previous observations with the weak convergence, for any smooth functional \(L\),
\begin{align*}
    \E \big[ \langle \varphi, \mu_{t_0} \rangle L(\mu) \big] 
    &= \lim_{n \to \infty} \E \big[ \langle \varphi, \mu^n_{t_0} \rangle L(\mu^n) \big] \\
    &= \lim_{n \to \infty} \E \big[ \langle \varphi, \mu^n_{0} \rangle L(\mu^n) \big] \\
    &\quad + \lim_{n \to \infty} \frac{1}{n} \sum_{i=1}^n \E \Bigg[ \int_0^{t_0} \partial_x \varphi(X^i_t,u^i_n) \,
       b\big( t, X^i_t, \Nc^1(t,X^i_t,u^i_n,\muh^{i,n}_t), \Nc^2(t,X^i_t,u^i_n,\muh^{i,n}_t), \beta(t,X^i_t,u^i_n) \big)
       \, \mathrm{d}t \, L(\mu^n) \Bigg] \\
    &\quad + \lim_{n \to \infty} \E \Bigg[ \int_0^{t_0} \frac12 \int_{\R^d \times [0,1]} {\rm Tr} \left(\partial_{xx}^2 \varphi(x,u) \,
       \sigma(t,x) \sigma(t,x)^\top \right) \, \mu^n_t(\mathrm{d}x,\mathrm{d}u) \, \mathrm{d}t \, L(\mu^n) \Bigg].
\end{align*}

Passing to the limit yields
\begin{align*}
    \E \big[ \langle \varphi, \mu_{t_0} \rangle L(\mu) \big] 
    &= \E \bigg[ \Big( \langle \varphi, \mu_{0} \rangle 
       + \int_0^{t_0} \partial_x \varphi(X_t,U) \,
         b\big( t, X_t, M^{1,\gamma}_{\mu,t}(t,X_t,U), M^{2,\gamma}_{\mu,t}(t,X_t,U), \beta(t,X_t,U) \big)
         \, \mathrm{d}t \\
    &\quad \quad \quad + \int_0^{t_0} \frac12 \int_{\R^d \times [0,1]} {\rm Tr} \left(\partial_{xx}^2 \varphi(x,u) \,
       \sigma(t,x) \sigma(t,x)^\top \right) \, \mu_t(\mathrm{d}x,\mathrm{d}u) \, \mathrm{d}t \Big) L(\mu) \bigg].
\end{align*}

Since this holds for all \(L\), we obtain \(\P\)-a.s., for all $t_0 \in [0,T]$,
\begin{align*}
    \langle \varphi, \mu_{t_0} \rangle 
    &= \langle \varphi, \mu_{0} \rangle 
       + \int_0^{t_0} \int_{ \R\times[0,1]} \partial_x \varphi(x,u) \,
         b\big( t, x, N^{1,\gamma}_{\mu_t}(t,  x,u), N^{2,\gamma}_{\mu_t}(t, x,u), \beta(t,x,u) \big)
         \, \mu_t(\mathrm{d}x,\mathrm{d}u) \, \mathrm{d}t \\
    &\quad + \frac12 \int_0^{t_0} \int_{ \R^d \times [0,1]} {\rm Tr} \left(\partial_{xx}^2 \varphi(x,u) \,
       \sigma(t,x) \sigma(t,x)^\top \right) \, \mu_t(\mathrm{d}x,\mathrm{d}u) \, \mathrm{d}t.
\end{align*}
Using the fact that $\Lim_{n \to \infty} \Lc\bigl( \mu^{n,\gammab^n,\betab^n}_0 \bigr)=\delta_{\eta} $ in $\Wc_p$, we check that $\mu_0=\eta$ $\P$--a.e.
By uniqueness of the Fokker--Planck equation derived above (see \Cref{prop:relaxed_uniqueness} for instance), we deduce that for all $t \in [0,T],$
\[
    \mu_t = \Lc\left(X^{\gamma,\beta}_t,U \right)\quad \mbox{with}\quad \Lc\left(X^{\gamma,\beta}_0,U \right)=\eta.
\]
This is true for any sub--sequence of $\left( \frac{1}{n} \sum_{i=1}^n 
    \Lc \left( \muh^{i,n},\, \mub^n,\, \mu^n \right) \right)_{n \ge 1}$. We can therefore deduce that the whole sequence converges and,
\begin{align*}
    \Lim_{n \to \infty} \frac{1}{n} \sum_{i=1}^n 
    \Lc \left( \muh^{i,n},\, \mub^n,\, \mu^n \right) = \P \circ (\muh, \mub, \mu)^{-1}\quad\mbox{with}\quad
\end{align*}
with for each $t \in [0,T]$, $\P$--a.e.
\begin{align*}
    \mu_t = \Lc\bigl(X^{\gamma,\beta}_t,U \bigr),\,\mub_t=\Lc\bigr(X^{\gamma,\beta}_t,U,U \bigl),\,\muh_t=V(t,U)\,\mbox{where}\,V(t,u)=\Lc\bigl(X^{\gamma,\beta}_t,U,U,\Gr(u,U) \bigr).
\end{align*}

This leads to $\Lim_{n \to \infty}\Lc\left( \mu^{n,\gammab^n,\betab^n}, \Lambda^{n,\gammab^n,\betab^n} \right)=\delta_{\left(\mu^{\gamma,\beta}, \Lambda^{\gamma,\beta} \right)}$ in $\Wc_p$. In addition,
\begin{align*}
    &\Lim_{n \to \infty} J_n(\nu^n,\gammab^n,\betab^n)
    \\
    &=\Lim_{n \to \infty}\frac{1}{n} \sum_{i=1}^n \E \left[ \int_0^T L \left( t, X^i_t, \Nc^1(t,X^i_t,u^i_n,\muh^{i,n}_t), \Nc^2(t,X^i_t,u^i_n,\muh^{i,n}_t), \beta(t,X^i_t,u^i_n) \right) \mathrm{d}t + g \left( X^i_T,\,\, \muh^i_T \circ (\Xh,\Eh)^{-1}\,\, \right) \right]
    \\
    &=J(\eta,\gamma,\beta).
\end{align*}
This completes the proof.

\end{proof}

\subsection{From {\it{n}}--particle to relaxed formulation: proof of \Cref{thm:from_n_to_limit}}

We fix $n\ge 1$ and let $\Xbb^n := (X^{1,n},\dots,X^{n,n})$ be the unique weak solution to \eqref{eq:n_particle} driven by the interaction kernels $\gammab^n$ and the regular controls $\alphab^n$.
Recall the time–indexed empirical objects
\[
    \mu^n_t \;:=\; \frac{1}{n}\sum_{i=1}^n \delta_{Z^{i,n}_t},
    \qquad 
    \Gamma^n_t(\mathrm{d}z,\mathrm{d}m,\mathrm{d}a)\,\mathrm{d}t
    \;:=\;
    \frac{1}{n}\sum_{i=1}^n 
    \delta_{\bigl(Z^{i,n}_t,\;M^{i,n}_{\gamma^n,t},\;\alpha^{i,n}(t,\Xbb^n)\bigr)}(\mathrm{d}z,\mathrm{d}m,\mathrm{d}a)\,\mathrm{d}t,
\]
where 
\[
    Z^{i,n}_t := (X^{i,n}_t,u^i_n),
    \qquad 
    M^{i,n}_{\gamma^n,t} := \bigl(M^{1,i,n}_{\gamma^n,t},\,M^{2,i,n}_{\gamma^n,t}\bigr).
\]
Thus, $\mu^n_t$ is the empirical state–label distribution at time $t$, while $\Gamma^n_t$ is the empirical occupation measure on the extended control space (state–label, interaction laws, regular action).

\medskip
\textbf{Auxiliary state spaces and lifted empirical environment.}
Introduce
\[
    \Xi 
    := \R^d \times [0,1] \times \Pc\!\bigl(A_{\rm int}\times \R^d \times \Er\bigr)^2 \times A_{\rm reg},
    \qquad
    \Delta 
    := A_{\rm int} \times \Er \times \Xi^2 \times \Pc(\Xi).
\]
An element $h\in\Xi$ encodes $(x,u)$, two interaction laws $(m^1,m^2)\in\Pc(A_{\rm int}\times\R^d\times\Er)^2$, and a regular action $a\in A_{\rm reg}$. 
An element of $\Delta$ collects one \emph{incoming} interaction mark $r\in A_{\rm int}$, an edge mark $e\in\Er$, two $\Xi$–states $h,\hat h$ describing an interacting pair, and a probability $\theta\in\Pc(\Xi)$ standing for the current empirical environment.

On this enlarged space we define the \emph{lifted} empirical interaction measure
\[
    \Gammabb^n_t\!\left(\mathrm{d}r,\mathrm{d}e,\mathrm{d}h,\mathrm{d}\hat h,\mathrm{d}\theta\right)\,\mathrm{d}t
    :=
    \frac{1}{n^2}\sum_{i,j=1}^n
    \delta_{\bigl(\gamma^{ij}_n(t,\Xbb^n),\;\xi^n_{ij},\;H^{i,n}_t,\;H^{j,n}_t,\;\Gamma^n_t\bigr)}
    \left(\mathrm{d}r,\mathrm{d}e,\mathrm{d}h,\mathrm{d}\hat h,\mathrm{d}\theta\right)\,\mathrm{d}t,
\]
where
\[
    H^{i,n}_t
    :=
    \bigl(Z^{i,n}_t,\;M^{i,n}_{\gamma^n,t},\;\alpha^{i,n}(t,\Xbb^n)\bigr) \in \Xi.
\]
We also set 
$$
    \left(\mub^{n}_t:=\frac{1}{n}\sum_{i=1}^n \delta_{\left( X^i_t,\, \Rr^{i,n}_t \right)}  \right)_{t \in [0,T]} \in C\left([0,T], \Pc\left(\R^d \x \Pc(\R^d \x \Er) \right) \right).
$$
For later compactness and identification arguments, we consider the joint laws
\[
    \overline{\Pr}^{\,n} 
    := \P\circ\bigl(\,\mu^n,\,\mub^n,\;\Gamma^n,\;\Gammabb^n\,\bigr)^{-1}
    \;\in\;
    \Pc\!\Bigl(C([0,T];\Pc(\R^d\times[0,1]))\,\times \,C\left([0,T], \Pc(\R^d \x \Er) \right)\times \M(\Xi)\times \M(\Delta)\Bigr).
\]
We can check the relative compactness in the weak topology by the same techniques as in the proof of \Cref{prop:a.e.cong}. Looking at $X^i=X^i-X^i_0 + X^i_0$ and using that $(X^i-X^i_0)$ is bounded and the sequence of initial distributions are relatively compact.   
Up to extraction, let $\overline{\Pr}^{n}\Rightarrow \overline{\Pr}:=\P\circ(\mu, \mub,\Gamma,\Gammabb)^{-1}$. 
On the canonical space $\Delta$ we write the coordinate random variables as
\[
    \bigl(\gamma,\;E,\;H=(X,U,M,\beta),\;\Hh=(\Xh,\Uh,\Mh,\betah),\;\Theta\bigr),
\]
with $M=(M^1,M^2)$ and $\Mh=(\Mh^1,\Mh^2)$.

\medskip
\textbf{Step 1: Limiting Fokker–Planck identity for $\mu$.}
Applying Itô’s formula to $\frac1n\sum_{i=1}^n\varphi(X^{i,n}_t,u^i_n)$ for any bounded $C^2$ test $\varphi:\R^d\times[0,1]\to\R$, and writing the drift through the lift $\Gammabb^n_t$, we obtain
\begin{align*}
    \mathrm{d}\langle \varphi,\mu^n_t\rangle
    &= \int_{\Xi} \partial_x\varphi(x,u)\;
       b\bigl(t,x,m,a\bigr)\;
       \Gammabb^n_t\!\bigl(A_{\rm int},\Er,\mathrm{d}h,\Xi,\mathrm{d}\theta\bigr)\,\mathrm{d}t
\\[-2pt]
    &\quad
    + \frac12 \int_{\R^d\times[0,1]}
        \mathrm{Tr}\!\Bigl(\partial_{xx}^2\varphi(x,u)\,\sigma\sigma^\top(t,x)\Bigr)
      \,\mu^n_t(\mathrm{d}x,\mathrm{d}u)\,\mathrm{d}t
    + \mathrm{d}M^n_t,
\end{align*}
where
\(
    M^n_t
    = \frac1n \sum_{i=1}^n \int_0^t \partial_x\varphi(X^{i,n}_s,u^i_n)\,\sigma(s,X^{i,n}_s)\,\mathrm{d}W^{i}_s,
\)
and $\E\bigl[\sup_{t\le T}|M^n_t|^2\bigr]\to 0$ by boundedness of $\sigma$ and $\varphi$.
Passing to the limit in distribution along $\overline{\Pr}^{\,n}\Rightarrow \overline{\Pr}$ yields, $\P$–a.s., for all $t\in[0,T]$,
\begin{align}
\label{eq:limit-FP}
    \mathrm{d}\langle \varphi,\mu_t\rangle
    &= \int_{\Xi} \partial_x\varphi(x,u)\; b\bigl(t,x,m,a\bigr)\;
       \Gammabb_t\!\bigl(A_{\rm int},\Er,\mathrm{d}h,\Xi,\Pc(\Xi)\bigr)\,\mathrm{d}t
\\[-2pt]
    &\quad
    + \frac12 \int_{\R^d\times[0,1]}
        \mathrm{Tr}\!\Bigl(\partial_{xx}^2\varphi(x,u)\,\sigma\sigma^\top(t,x)\Bigr)
      \,\mu_t(\mathrm{d}x,\mathrm{d}u)\,\mathrm{d}t.
\nonumber
\end{align}

\medskip
\textbf{Step 2: Structural properties of the limit $\Gammabb$.}
Recall that $(\gamma,E,H,\Hh,\Theta)$ are the canonical coordinates on $\Delta$, and for $\mathrm{d}\P\otimes\mathrm{d}t$–a.e.\ $(\omega,t)$, $\Gammabb_t(\omega)\in\Pc(\Delta)$.

\smallskip
\emph{(a) Conditional independence of $H$ and $\Hh$ given $\Theta$.}
For any bounded continuous $\phi,\varphi$ on $\Xi$ and $\Psi$ on $[0,T]\times \Pc(\Xi)$,
\begin{align*}
    &\E\!\left[\left|\int_0^T 
      \E^{\Gammabb_t}\!\Big[\Psi(t,\Theta)\Big(\phi(H)\varphi(\Hh) 
      - \langle\phi,\Theta\rangle\,\langle\varphi,\Theta\rangle\Big)\Big]\mathrm{d}t\right|\right]
\\
    &\qquad
    = \lim_{n\to\infty}
      \E\!\left[\left|\int_0^T \frac1{n^2}\sum_{i,j=1}^n 
        \Psi(t,\Gamma^n_t)\Big(\phi(H^{i,n}_t)\varphi(H^{j,n}_t) 
      - \langle\phi,\Gamma^n_t\rangle\,\langle\varphi,\Gamma^n_t\rangle\Big)\mathrm{d}t\right|\right]
      = 0,
\end{align*}
which implies, for $\P$–a.e.\ $\omega$ and $\Gammabb_t(\omega)$–a.e.,
\[
    \Lc^{\Gammabb_t(\omega)}(H,\Hh\mid\Theta)
    = \Lc^{\Gammabb_t(\omega)}(H\mid\Theta)\otimes \Lc^{\Gammabb_t(\omega)}(H\mid\Theta)
    = \Theta\otimes\Theta .
\]

\smallskip
\emph{(b) Identification of $M^1$ and $\Mh^2$ as conditional interaction laws.}
For bounded continuous $\phi$ on $\Xi$, $\varphi$ on $A_{\rm int}\times \R^d\times \Er$, and $\Psi$ as above,
\begin{align*}
    &\E\!\left[\left|\int_0^T 
      \E^{\Gammabb_t}\!\Big[\Psi(t,\Theta)\Big(\varphi(\gamma,\Xh,E)\,\phi(H) 
      - \langle\varphi,M^1\rangle\,\phi(H)\Big)\Big]\mathrm{d}t\right|\right]
\\
    &\qquad
    = \lim_{n\to\infty}
      \E\!\left[\left|\int_0^T \frac1{n^2}\sum_{i,j=1}^n 
        \Psi(t,\Gamma^n_t)\Big(\varphi(\gamma^{ij}_n(t,\Xbb^n),X^{j,n}_t,\xi^n_{ij})\,\phi(H^{i,n}_t)
      - \langle\varphi,M^{1,i,n}_t\rangle\,\phi(H^{i,n}_t)\Big)\mathrm{d}t\right|\right]
      = 0,
\end{align*}
hence, $\P$–a.e.\ $\omega$ and $\Gammabb_t(\omega)$–a.e.,
\[
    M^1 = \Lc^{\Gammabb_t(\omega)}(\gamma,\Xh,E\mid H,\Theta).
\]
By symmetry,
\[
    \Mh^2 = \Lc^{\Gammabb_t(\omega)}(\gamma,X,E\mid \Hh,\Theta).
\]

\smallskip
\emph{(c) Consistency of the edge mark.}
Using the convergence of the step–kernels $\Gr^n$ to $\Gr$ (in cut–norm, tested against Lipschitz functions) and \Cref{prop:limit_kernel_1}, we obtain, $\P$–a.e.\ $\omega$ and $\Gammabb_t(\omega)$–a.e.,
\[
    \Lc^{\Gammabb_t(\omega)}(E,H,\Hh\mid\Theta)
    = \Lc^{\Gammabb_t(\omega)}(\Gr(U,\Uh),H,\Hh\mid\Theta).
\]

\smallskip
\emph{(d) Marginal consistency.}
For any bounded continuous $\phi$ on $\R^d\times[0,1]$ and $F$ on $[0,T]\times\Pc(\Xi)$,
\begin{align*}
    &\E\!\left[\left|\int_0^T 
      \Big(\E^{\Gammabb_t}\!\big[\phi(X,U)\,F(t,\Theta)\big]
      - \langle\phi,\mu_t\rangle \,\E^{\Gammabb_t}[F(t,\Theta)]\Big)\mathrm{d}t\right|\right]
\\
    &\qquad
    = \lim_{n\to\infty}\E\!\left[\left|\int_0^T 
      \Big(\tfrac1n\sum_{i=1}^n \phi(X^{i,n}_t,u^i_n)\,F(t,\Gamma^n_t)
      - \langle\phi,\mu^n_t\rangle \,F(t,\Gamma^n_t)\Big)\mathrm{d}t\right|\right]=0,
\end{align*}
which implies $\mu_t(\omega)=\Lc^{\Gammabb_t(\omega)}(X,U\mid \Theta)$ for $\mathrm{d}t$–a.e.\ $t$.

\medskip
\textbf{Step 3: Membership in the relaxed set and identification of the limit of $(\mu^n,\Gamma^n)$.}
From (a)–(d) and the reformulation of the relaxed class $\Mcb$ (via $\Mcb_{\rm aux}$), we infer that, $\P$–a.e.\ $\omega$ and $\Gammabb_t(\omega)$–a.e.,
\[
    \Lc^{\Gammabb_t(\omega)}\bigl(Z,M\mid \Theta\bigr)
    \in \Mcb,
    \qquad 
    \text{with } \; Z=(X,U).
\]
In particular, under $\Gammabb_t(\omega)$, the conditional law $\Theta=\Lc^{\Gammabb_t(\omega)}(H\mid\Theta)$ has $(x,u)$–marginal $\mu_t(\omega)$, so the randomization occurs only at the interaction component (consistent with the definition of a relaxed control).

Moreover, by construction $\Gamma^n$ is the $\Xi$–marginal of $\Gammabb^n$:
\[
    \Gamma^n_t(\mathrm{d}h)\,\mathrm{d}t
    =
    \Gammabb^n_t\!\bigl(A_{\rm int},\Er,\mathrm{d}h,\Xi,\mathrm{d}\theta\bigr)\,\mathrm{d}t.
\]
Passing to the limit along the subsequence, we may define
\[
    \Gamma_t(\mathrm{d}h)\,\mathrm{d}t
    :=
    \Gammabb_t\!\bigl(A_{\rm int},\Er,\mathrm{d}h,\Xi,\Pc(\Xi)\bigr)\,\mathrm{d}t
    \;=\;
    \E^{\Gammabb_t}\!\bigl[\,\Lc^{\Gammabb_t}(H\mid\Theta)(\mathrm{d}h)\,\bigr]\mathrm{d}t,
\]
and set $\Pr := \P\circ(\mu,\Gamma)^{-1}$.
By the characterization above, $\P$–a.s.\ $(\mu,\Gamma)$ is a relaxed control in the sense of Section~\ref{sec:relaxed}. 
Since every subsequential limit of $\Pr^n:=\P\circ(\mu^n,\Gamma^n)^{-1}$ must coincide with such a $\Pr$, we conclude:

\begin{itemize}
\item The sequence $(\Pr^n)_{n\ge 1}$ is relatively compact in the topology of weak convergence on $\Pc\bigl(C([0,T];\Pc(\R^d\times[0,1]))\times \M(\Xi)\bigr)$.
\item Every limit point $\Pr$ is supported on the set of (admissible) relaxed controls $(\mu,\Gamma)$ i.e. $\P$--a.e. $\om$, $(\mu(\om),\Lambda(\om)) \in \Pib\left(\mu_0(\om)(\mathrm{d}x,[0,1]) \right)$ and $\mu_0(\om) \in \Pc(\R^d \x [0,1])$.
\end{itemize}

This completes the argument.

\medskip
\textbf{Step 4: convergence of the value function.}

In order to establish the convergence involving the value function, we start by detailing the variable $\mub$. Let us observe that, for any smooth maps $(F,\varphi)$, using the weak convergence we obtain
\begin{align*}
    &\E \left[\int_0^T F(t)\left( \E^{\Gammabb_t}\left[ \varphi\left(X,\,M^1(A_{\rm int}, \mathrm{d}x,[0,1],\mathrm{d}e) \right) \right]- \langle \varphi, \mub_t \rangle \right) \mathrm{d}t \right]
    \\
    &=
    \lim_{n \to \infty} \E \left[\int_0^T F(t)\left(  \frac{1}{n}\sum_{i=1}^n\varphi\left(X^i_t,\,M^{1,i,n}_{\gammab^n,\,t}(A_{\rm int}, \mathrm{d}x,[0,1],\mathrm{d}e) \right)- \langle \varphi, \mub^n_t \rangle \right) \mathrm{d}t \right]=0.
\end{align*}
This being true for any $(F,\varphi)$, we deduce that $\mathrm{d}\P \otimes\mathrm{d}t$--a.e. 
\begin{align*}
    &\mub_t=\Lc^{\Gammabb_t}\left(X,\,M^1(A_{\rm int}, \mathrm{d}x,[0,1],\mathrm{d}e)   \right)
    =
    \Lc^{\Gammabb_t}\left(X,\, \Lc^{\Gammabb_t}\left(\Xh,E \,\mid H,\Theta  \right)   \right) 
    \\
    &=
    \Lc^{\Gammabb_t}\left(X,\, \Lc^{\Gammabb_t}\left(\Xh,\Gr \bigl(U, \Uh \bigr) \,\mid U,\Theta  \right)   \right) =\Lc^{\Gammabb_t}\left(X,\, V(U,\mu_t)   \right)
\end{align*}
where $V(u,\mu_t):=\Lc^{\mu_t}\left(\Xh,\Gr \bigl(u, \Uh \bigr)  \right)$. We used the identities in $Step\,2$, and the independence of $(X,U)$ and $\Theta$. The processes $\mub$ and $\mu$ being continuous, we deduce that $\P$-a.e. $\mub_t=\Lc^{\mu_t}\left(X,\, V(U,\mu_t)   \right)$ for all $t \in [0,T]$. Consequently
\begin{align*}
    &\Lim_{n \to \infty}J_n(\nu^n,\gammab^n,\betab^n) = \Lim_{n \to \infty}\E \left[ \int_0^T \int_{\Xi}  L\bigl(t,x,m,a\bigr)\;
       \Gamma^n_t(\mathrm{d}h)\,\mathrm{d}t + \langle g, \mub^n_T \rangle \right] 
       \\
       &= \E \left[ \int_0^T \int_{\Xi}  L\bigl(t,x,m,a\bigr)\;
       \Gamma_t(\mathrm{d}h)\,\mathrm{d}t + \int_{\R^d} g (x, V(u,\mu_T)) \mu_T (\mathrm{d}x,\mathrm{d}u) \right].
\end{align*}
We can conclude the proof.

\subsection{Proof of \Cref{prop:conv_value_function}}

Let us start by checking that $\liminf_{n \to \infty} V_n(\nu^n) \ge V_{\rm MFC}(\nu)$. Indeed, let $\eta \in \Pc_\nu$ and $(\gamma,\beta) \in \Ac_{\rm int} \x \Ac_{\rm reg}$ be continuous Lipchitz maps. By \Cref{thm:strong_to_n}, we can construct $(\gammab^n,\betab^n) \in \Ac_{n, \rm int}^{n^2} \x \Ac_{n,\rm reg}^n$ s.t. $\lim_{n \to \infty} J_n(\nu^n,\gammab^n,\betab^n)=J(\eta,\gamma,\beta)$. Therefore, we obtain that $J(\eta,\gamma,\beta) \le \Liminf_{n \to \infty} V_n(\nu^n)$ for any $(\gamma,\beta)$ Lipschitz continuous. Since the supremum in $V_{\rm MFC}(\nu)$ can be taken over Lipschitz maps, we deduce that $V_{\rm MFC}(\nu) \le \Liminf_{n \to \infty} V_n(\nu^n)$. We now show that $\limsup_{n \to \infty} V_n(\nu^n) \le V_{\rm MFC}(\nu)$. Let $(\gammab^n,\betab^n) \in \Ac_{n, \rm int}^{n^2} \x \Ac_{n,\rm reg}^n$ s.t. $\lim_{n \to \infty}$ be an $1/n$--optimal control of $V_n(\nu^n)$ i.e. $V_n(\nu^n) \le J_n(\nu^n,\gammab^n,\betab^n) + 1/n$. Under the convergence conditions of the initial distributions $(\nu^n)_{n \ge 1}$, by applying \Cref{thm:from_n_to_limit}, the sequence $(\Pr^n:=\Lc\left( \mu^{n,\gammab^n,\betab^n},\,\Lambda^{n,\gammab^n,\betab^n} \right))_{n \ge 1}$ is relatively compact for the weak convergence, and each limit point $\Pr=\P \circ (\mu,\Lambda)^{-1}$ is supported by the set of relaxed controls. In addition, for the convergent subsequence $(\Pr^{n_k})_{k \ge 1}$, we have $\Lim_{k \to \infty} J_{n_k}(\nu^{n_k},\gammab^{n_k},\betab^{n_k})=\E^\P \left[ \Gc \left( \mu, \Lambda \right) \right].$ Since $\Lim_{n \to \infty} \Lc\left( \mu^{n,\gammab^n,\betab^n}_0 (\mathrm{d}x,[0,1]) \right)= \delta_{\nu}$, we deduce that $(\mu,\Lambda) \in \Pi(\nu)$ a.e. Consequently, $\E^\P \left[ \Gc \left( \mu, \Lambda \right) \right] \le V_{\rm MFC}(\nu)$. This leads to $\Limsup_{n \to \infty} V_n(\nu^n) \le \limsup_{n \to \infty} J_{n}(\nu^{n},\gammab^{n},\betab^{n}) \le V_{\rm MFC}(\nu)$. We can conclude the proof of the proposition.

\subsection{ Proof of \Cref{prop_example1} }

We establish the optimality of $\widehat{\gamma}$ by analyzing the $n$--player approximation of the mean--field control problem.
By the convergence result in \Cref{prop:conv_value_function}, we know that
\[
    V_{\mathrm{MFC}}(\nu)
    = \lim_{n \to \infty} V_n\bigl(\nu^{\otimes n}\bigr),
\]
where $V_n$ denotes the value function of the $n$--agent problem.
For each $n \ge 1$, consider the $\F$--adapted processes 
\[
    (Y^n,\Zbb^n:=(Z^{1,n},\dots,Z^{n,n}),\Xbb^n)
\]
satisfying, for all $t \in [0,T]$ and $1 \le i \le n$,
\begin{align*}
    X^{i,n}_t = X^i_0 + \int_0^t  \Phih\left(s,X^i_{s}, \mu^n_s \right)  \;\mathrm{d}s 
        +
        W^i_t,\quad Y^n_t
    =
    G\left( \mu^n_T\right)
    - \sum_{i=1}^n\int_t^T Z^{i,n}_s \mathrm{d}W^i_s
\end{align*}
where $\mu^n_t := \frac{1}{n}\sum_{i=1}^n \delta_{X^{i,n}_t}$ and the initial law satisfies 
$\Lc(X^1_0,\dots,X^n_0)=\nu^{\otimes n}$.
Applying the Clark--Ocone formula (see, e.g., \cite[Proposition~1.5]{nualart2006malliavin}), we obtain that
\[
    Z^{i,n}_t = \partial_{x^i} v^n\bigl(t,\Xbb^n_t\bigr),
\]
where $v^n$ denotes the value function of the $n$--player problem.  
Differentiating $v^n$, for $(t,\xbb) \in [0,T] \x (\R^d)^n$, we get
\[
    \partial_{x^i} v^n\bigl(t,\xbb\bigr)
    = \sum_{j=1}^n
        \E\!\left[
            J^{j,t,\xbb,i}_T
            \frac{1}{n}\partial_x \delta_m G
            \!\left(
                \frac{1}{n}\sum_{k=1}^n \delta_{X^{k,t,\xbb}_T}
            \right)
            \!\!\left(X^{j,t,\xbb}_T\right)
        \right],
\]
where $\Lc(X^{1,t,\xbb},\dots,X^{n,t,\xbb})
    = \Lc(\Xbb^n \mid \Xbb^n_{t \wedge \cdot}=\xbb)$, 
and for $s \in [t,T]$, the Jacobian processes $(J^{j,t,\xbb,i}_s)_{1\le j,i\le n}$ solve
\[
    J^{j,t,\xbb,i}_s = \mathbf{1}_{i=j}
    + \int_t^s J^{j,t,\xbb,i}_r
        \partial_x\Phih\!\left(r,X^{j,t,\xbb}_r,\nu^{n,t,\xbb}_r\right)\mathrm{d}r
    + \int_t^s \sum_{k=1}^n 
        J^{k,t,\xbb,i}_r
        \frac{1}{n}\partial_y\delta_m\Phih\!\left(r,X^{j,t,\xbb}_r,\nu^{n,t,\xbb}_r\right)
        \!\left(X^{k,t,\xbb}_r\right)\mathrm{d}r,
\]
with $\nu^{n,t,\xbb}_r := \frac{1}{n}\sum_{k=1}^n \delta_{X^{k,t,\xbb}_r}$. Notice that since the map $r \mapsto (r)^+$ involved in $\Phih$ is not differentiable in $0$, we have considered its weak derivative which exists because it is a Lipschitz map.
By the monotonicity assumptions $\partial_y\delta_m\Phih \ge 0$, it follows that $J^{j,t,\xbb,i}_s \ge 0$ for all $s \in [t,T]$, and therefore
\[
    Z^{i,n}_t = \partial_{x^i} v^n(t,\Xbb^n_t) \ge 0.
\]
This nonnegativity plays a key role in identifying the optimal control.
Recalling that the interaction term satisfies
\[
    \Phih(s,X^i_s,\mu^n_s)
    = \frac{1}{n}\sum_{j=1}^n
        \Phi\!\left(s,X^i_s,X^j_s,\mu^n_s\right)
        \mathbf{1}_{\{\Phi(s,X^i_s,X^j_s,\mu^n_s)\,nZ^{i,n}_s \ge 0\}},
\]
we deduce that the optimal control at the $n$--player level is given by
\[
    \gamma^{\star,n}_{ij}(s,\Xbb^n)
    := \mathbf{1}_{\{\Phi(s,X^i_s,X^j_s,\mu^n_s)\,nZ^{i,n}_s \ge 0\}}
    = \mathbf{1}_{\{\Phi(s,X^i_s,X^j_s,\mu^n_s) \ge 0\}},
\]
since $Z^{i,n}_s \ge 0$.  
Moreover, by construction, we can check that (see for instance \cite[Proposition 3.1.]{djete2024connectedproblemsprincipalmultiple})
\[
    V_n\bigl(\nu^{\otimes n}\bigr) = \E[Y^n_0].
\]

\medskip
Finally, using the continuity of $\Phih$ and the convergence 
$\mu^n \to \mu^\star$ as $n \to \infty$, we obtain
\[
    V_{\mathrm{MFC}}(\nu)
    = \lim_{n \to \infty} V_n\bigl(\nu^{\otimes n}\bigr)
    = \lim_{n \to \infty} 
        J_n\!\left(\nu^{\otimes n},
            (\gamma^{\star,n}_{ij})_{1\le i,j\le n},
            a_{\rm reg}
        \right)
    = \widehat{J}(\widehat{\gamma}),
\]
for some $a_{\rm reg} \in A_{\rm reg}$ where $\widehat{\gamma}(t,x,y) = \mathbf{1}_{\{\Phi(t,x,y,\mu^\star_t)\ge 0\}}$.
Hence $\widehat{\gamma}$ is indeed the optimal interaction control.

\subsection{ Proof of \Cref{prop_example2} }

\textbf{Step 1: Existence of an optimal relaxed control and its state dynamics.}
By the existence of optimal control in \Cref{prop:existence_continuity}, there exists an optimal \emph{relaxed} interaction control
\(
    \gammabb
\in \Acb_{\rm int}
\)
(maximizing the relaxed objective).
Under the standing assumptions in Example~2, the relaxed dynamics of the optimally
controlled state process \(X\) can be written as
\begin{align*}
    \mathrm{d}X_s
    &= \int_{[0,1]^3} \Bigg(
       \int_{\R \times [0,1]^2}
          \gammabb\!\left(s, X_s,U, v\,;\,x',u',v',\overline{v},\widehat v\right)\;
          b_1\!\big(s,\Gr(U,u'),X_s\big)\;
          \mu_s(\mathrm{d}x',\mathrm{d}u')\,
          \mathrm{d}v'
\\[-0.3em]
    &\hspace{7.8em}+
       \int_{\R \times [0,1]^2}
          \gammabb\!\left(s, x',u',v'\,;\,X_s,U, v,\overline{v},\widehat v\right)\;
          b_2\!\big(s,\Gr(U,u'),X_s\big)\;
          \mu_s(\mathrm{d}x',\mathrm{d}u')\,
          \mathrm{d}v'
    \Bigg)\,
    \mathrm{d}\overline{v}\,\mathrm{d}v\,\mathrm{d}\widehat v\,\mathrm{d}s
    \;+\; \sigma\,\mathrm{d}W_s,
\end{align*}
with \(\mu_s:=\Lc(X_s,U)\).
The corresponding relaxed objective reads
\begin{align*}
    &\overline{J} (\Lc(X_0,U),\,\gammabb, a_{\rm reg})
    \\
    &:=
    \E\!\left[
        \int_0^T \int_{[0,1]^3}
        \int_{\R \times [0,1]^2}
            L\!\Big(
                s,\,
                \gammabb\!\left(s, X_s,U, v\,;\,x',u',v',\overline{v},\widehat v\right),\,
                \Gr(U,u'),\,
                X_s
            \Big)\,
            \mu_s(\mathrm{d}x',\mathrm{d}u')\,\mathrm{d}v'\,
        \mathrm{d}\overline{v}\,\mathrm{d}v\,\mathrm{d}\widehat v\,
        \mathrm{d}s
    \right]
\end{align*}
for some $a_{\rm reg}  \in A_{\rm reg}$.

\medskip
\textbf{Step 2: Barycentric projection to a closed--loop control.}
Define the \emph{deterministic} closed--loop interaction control
\[
    \gamma(t,x,u,x',u')
    \;:=\;
    \int_{[0,1]^4}
        \gammabb\!\left(t, x,u, v\,;\,x',u',v',\overline{v},\widehat v\right)\,
    \mathrm{d}v\,\mathrm{d}v'\,\mathrm{d}\overline{v}\,\mathrm{d}\widehat v,
    \qquad (t,x,u,x',u')\in[0,T]\times(\R\times[0,1])^2.
\]
Since \(A_{\rm int}\) is convex and \(\gammabb\in A_{\rm int}\), we have \(\gamma\in A_{\rm int}\), hence
\(\gamma\in\Ac_{\rm int}\).

\medskip
\textbf{Step 3: Dynamics under \(\gamma\) coincide with the relaxed dynamics.}
In the drift, the control enters \emph{linearly} (as a scalar weight) in front of \(b_1\) and \(b_2\).
Therefore, by Fubini’s theorem,
replacing \(\gammabb\) by its average \(\gamma\) leaves the drift unchanged:
the state process driven by \(\gamma\) solves the same SDE (in law) as the one driven by \(\gammabb\).
Consequently, the associated marginal flow \(\mu=\Lc(X,U)\) is the same under \(\gamma\) and under \(\gammabb\).

\medskip
\textbf{Step 4: No loss under projection (payoff comparison).}
Assume the map \(a\mapsto L(s,a,r,x)\) is \emph{concave} on \(A_{\rm int}\) for every \((s,r,x)\).
Then, by Jensen’s inequality,
\[
    L\!\big(s,\gamma(t,x,u,x',u'),\Gr(U,u'),X_s\big)
    \;\ge\;
    \int_{[0,1]^4}
        L\!\Big(
            s,\,
            \gammabb\!\left(s, X_s,U, v\,;\,x',u',v',\overline{v},\widehat v\right),\,
            \Gr(U,u'),\,
            X_s
        \Big)\,
    \mathrm{d}v\,\mathrm{d}v'\,\mathrm{d}\overline{v}\,\mathrm{d}\widehat v.
\]
Integrating in \((x',u')\), then in time and expectation, we obtain
\[
    \widehat{J}(\gamma) \;\ge\; \overline{J} (\Lc(X_0,U),\,\gammabb, a_{\rm reg})\,\quad \mbox{for some }a_{\rm reg}  \in A_{\rm reg}.
\]
Since \(\gammabb\) is optimal in the relaxed class, it follows that \(\gamma\) is also optimal (there is no loss of optimality when projecting to a closed--loop control).

\medskip
\textbf{Conclusion.}
We have constructed a closed--loop control \(\gamma\in\Ac_{\rm int}\) which achieves the same (indeed, no smaller) value as the optimal relaxed control \(\gammabb\).
Hence \(\gamma\) is an optimal closed--loop interaction control for Example~2.

\begin{appendix}
\section{Technical results}

\subsection{A uniform weak convergence approximation of measures }

Let $(A, \Delta)$ be a compact metric space, and let $\R \times \R \ni (x,y) \mapsto \Lambda_{x,y}(\mathrm{d}a) \in \Pc(A)$ be a measurable family of probability measures. For each $k \in \N^*$, since $A$ is compact, there exists a finite partition $(A^k_i)_{1 \le i \le k}$ of $A$ and corresponding points $(a^k_i)_{1 \le i \le k} \subset A$ such that:
\begin{itemize}
    \item $A = \bigcup_{i=1}^k A^k_i$,
    \item $A^k_i \cap A^k_j = \emptyset$ for $i \neq j$,
    \item $a^k_i \in A^k_i$ and $\Delta(a, a^k_i) \le \frac{1}{k}$ for all $a \in A^k_i$.
\end{itemize}

\begin{proposition} \label{prop:appro_control}
\begin{itemize}
    \item[$(i)$] For each $k \ge 1$, there exists a collection of piecewise constant maps $\left( \Lambda^{i,k}_{x,y} \right)_{1 \le i \le k}$ from $\R^2$ to $[0,1]$ such that for all $(x,y) \in \R^2$,
    \[
        \sum_{i=1}^k \Lambda^{i,k}_{x,y} = 1,
    \]
    and such that the approximation
    \[
        \sum_{i=1}^k \delta_{a^k_i}(\mathrm{d}a) \Lambda^{i,k}_{x,y}
    \]
    converges weakly to $\Lambda_{x,y}(\mathrm{d}a)$ for almost every $(x,y)$, when restricted to compact sets:
    \[
        \lim_{k \to \infty} \sum_{i=1}^k \delta_{a^k_i}(\mathrm{d}a)\, \Lambda^{i,k}_{x,y} \cdot 1_{\{|x| \le k\}} \cdot 1_{\{|y| \le k\}} = \Lambda_{x,y}(\mathrm{d}a).
    \]
    
    \medskip
    \item[$(ii)$] Fix $k \ge 1$. Then there exists a sequence of continuous maps $\left( \alpha^n : \R^2 \to A \right)_{n \ge 1}$ such that for any continuous and bounded function $\varphi : A \times \R^2 \to \R$, the following convergence holds:
    \begin{align*}
        \lim_{n \to \infty} \esup_{x \in [-k,k]} \left| \int_{[-k,k]} \varphi\left( \alpha^n(x,y),\, x, y \right) \mathrm{d}y 
        - \sum_{i=1}^k \int_{A \times [-k,k]} \varphi(a^k_i, x, y)\, \Lambda^{i,k}_{x,y}(\mathrm{d}a)\, \mathrm{d}y \right| = 0,
    \end{align*}
    and similarly,
    \begin{align*}
        \lim_{n \to \infty} \esup_{y \in [-k,k]} \left| \int_{[-k,k]} \varphi\left( \alpha^n(x,y),\, x, y \right) \mathrm{d}x 
        - \sum_{i=1}^k \int_{A \times [-k,k]} \varphi(a^k_i, x, y)\, \Lambda^{i,k}_{x,y}(\mathrm{d}a)\, \mathrm{d}x \right| = 0.
    \end{align*}
\end{itemize}
\end{proposition}

\begin{remark}
    $(i)$ The proof of this proposition, and in particular Item~$(ii)$, is essentially based on an adaptation of {\rm\cite[Lemma~4.7]{el1987compactification}}. 
There are two main points to emphasize in Item~$(ii)$. 
First, the supremum appearing in the convergence results originates from the construction in {\rm \cite[Lemma~4.7]{el1987compactification}}. 
Second, this construction also allows us to employ the same sequence of maps $\alpha^n : \R^2 \to A$ in both convergence statements.

    \medskip
    $(ii)$ By a straightforward adaptation of the arguments (see below), the constructions involving the mappings $(x,y) \mapsto \Lambda_{x,y}$ and $(x,y) \mapsto \alpha^n(x,y)$ on $\R \times \R$ can be extended to the general case $(x,y) \in \R^d \times \R^\ell$, for any integers $(d,\ell) \in \{1,2,\ldots\}^2$. 
We restrict here to the one--dimensional setting only to simplify the exposition and the proof.

\end{remark}

\begin{proof}
    \textbf{Item 1}. It is straightforward to construct a continuous approximation in $(x,y)$ of $\Lambda_{x,y}$—e.g., via convolution with a smooth density. We then assume that $(x,y) \mapsto \Lambda_{x,y}$ is continuous. Now, for each $k \ge 1$, we can choose a partition 
\[
    -k = t_{-k} < t_{-k+1} < \dots < t_{k-1} < t_k = k,
\]
and define a piecewise constant family of probability measures $\left( \Lambda^k_{x,y}(\mathrm{d}a) \right)_{(x,y) \in \R^2}$ such that:
\begin{itemize}
    \item For any $(x,y)$ with $t_j \le x < t_{j+1}$ and $t_\ell \le y < t_{\ell+1}$, we set $\Lambda^k_{x,y} := \Lambda^k_{t_j,t_\ell}$;
    \item The following weak convergence holds almost everywhere:
    \[
        \lim_{k \to \infty} \Lambda^k_{x,y}(\mathrm{d}a)\, 1_{\{|x| \le k\}}\, 1_{\{|y| \le k\}} = \Lambda_{x,y}(\mathrm{d}a).
    \]
\end{itemize}

Defining $\Lambda^{i,k}_{x,y} := \Lambda^k_{x,y}(A^k_i)$ for the partition $(A^k_i)_{1 \le i \le k}$ introduced earlier, we obtain the desired approximation:
\[
    \lim_{k \to \infty} \sum_{i=1}^k \delta_{a^k_i}(\mathrm{d}a)\, \Lambda^{i,k}_{x,y} = \Lambda_{x,y}(\mathrm{d}a), \quad \text{a.e. } (x,y) \in \R^2.
\]

\textbf{Item 2} To improve readability, we present the proof on the interval $[0,1]$ instead of $[-k,k]$. The arguments can be easily adapted to $[-k,k]$ by a straightforward translation and symmetry argument.

\medskip
Recall that $\Lambda^{i,k}_{s,t}$ is piecewise constant: there exists $m \in \N^*$ such that
\[
  \Lambda^{i,k}_{s,t}=\Lambda^{i,k}_{[s]^m,[t]^m},
  \qquad [r]^m:=\frac{q}{m}\ \text{ for } r\in\Big[\frac{q}{m},\frac{q+1}{m}\Big),\ q=0,\ldots,m-1,
\]
and we set $[1]^m:=1$. For notational convenience, we write $[\cdot]$ for $[\cdot]^m$. 
Fix $n\in m\N$ (so that the $1/n$-grid refines the $1/m$-grid). Define
\[
  T^n_j:=\Big(\frac{j}{n},\frac{j+1}{n}\Big],\qquad j=0,\ldots,n-1.
\]
For each $(j,\ell)\in\{0,\ldots,n-1\}^2$ and $i=1,\ldots,k$, define the subintervals
\begin{align*}
  {}^1T^{i,n}_{j,\ell}
  &:=\Big(\,\frac{j}{n}+\sum_{e=1}^{i-1}\frac{1}{n}\Lambda^{e,k}_{[j/n],[\ell/n]}\,,\ \frac{j}{n}+\sum_{e=1}^{i}\frac{1}{n}\Lambda^{e,k}_{[j/n],[\ell/n]}\,\Big],\\
  {}^2T^{i,n}_{j,\ell}
  &:=\Big(\,\frac{\ell}{n}+\sum_{e=1}^{i-1}\frac{1}{n}\Lambda^{e,k}_{[j/n],[\ell/n]}\,,\ \frac{\ell}{n}+\sum_{e=1}^{i}\frac{1}{n}\Lambda^{e,k}_{[j/n],[\ell/n]}\,\Big].
\end{align*}
Since $\sum_{e=1}^k\Lambda^{e,k}_{[j/n],[\ell/n]}=1$, the families $\big({}^1T^{i,n}_{j,\ell}\big)_{i=1}^k$ and $\big({}^2T^{i,n}_{j,\ell}\big)_{i=1}^k$ form disjoint partitions of $T^n_j$ and $T^n_\ell$, respectively. Hence the rectangles
\[
  \big({}^1T^{i,n}_{j,\ell}\times {}^2T^{i,n}_{j,\ell}\big)_{i=1}^k
\]
form a disjoint partition of $T^n_j\times T^n_\ell$ for each $(j,\ell)$. Set
\[
  U^{i,n}:=\bigcup_{j=0}^{n-1}\ \bigcup_{\ell=0}^{n-1}\Big({}^1T^{i,n}_{j,\ell}\times {}^2T^{i,n}_{j,\ell}\Big)\subset[0,1]^2.
\]

\medskip
Let $\phi:[0,1]\to\R$ be continuous with modulus of continuity $w_\phi$. Fix $t\in[0,1]$ and let $t\in T^n_\ell$ for some $\ell$. Because $n$ is a multiple of $m$, we have $[r]=[j/n]$ for all $r\in T^n_j$, and therefore
\[
  \Lambda^{i,k}_{r,t}=\Lambda^{i,k}_{[r],[t]}=\Lambda^{i,k}_{[j/n],[\ell/n]}\quad\text{for all }(r,t)\in T^n_j\times T^n_\ell.
\]
We decompose
\begin{align*}
  \int_0^1 \phi(r)\big(\Lambda^{i,k}_{r,t}-1_{U^{i,n}}(r,t)\big)\,\mathrm{d}r
  &=\sum_{j=0}^{n-1}\int_{T^n_j}\big(\phi(r)-\phi(j/n)\big)\big(\Lambda^{i,k}_{r,t}-1_{U^{i,n}}(r,t)\big)\,\mathrm{d}r\\
  &\quad+\sum_{j=0}^{n-1}\phi(j/n)\int_{T^n_j}\big(\Lambda^{i,k}_{r,t}-1_{U^{i,n}}(r,t)\big)\,\mathrm{d}r.
\end{align*}
On $T^n_j\times T^n_\ell$, $\Lambda^{i,k}_{r,t}$ is constant and equals $\Lambda^{i,k}_{[j/n],[\ell/n]}$. By construction of $U^{i,n}$,
\[
  \int_{T^n_j}1_{U^{i,n}}(r,t)\,\mathrm{d}r=\frac{1}{n}\Lambda^{i,k}_{[j/n],[\ell/n]},
\]
hence the second sum cancels termwise:
\[
  \int_{T^n_j}\big(\Lambda^{i,k}_{r,t}-1_{U^{i,n}}(r,t)\big)\,\mathrm{d}r
  =\frac{1}{n}\Lambda^{i,k}_{[j/n],[\ell/n]}-\frac{1}{n}\Lambda^{i,k}_{[j/n],[\ell/n]}=0.
\]
For the first sum, use $|\Lambda^{i,k}_{r,t}-1_{U^{i,n}}(r,t)|\le 1$ and $|\phi(r)-\phi(j/n)|\le w_\phi(1/n)$ on $T^n_j$ to obtain
\[
  \left|\int_{T^n_j}\big(\phi(r)-\phi(j/n)\big)\big(\Lambda^{i,k}_{r,t}-1_{U^{i,n}}(r,t)\big)\,\mathrm{d}r\right|
  \le \frac{1}{n}\,w_\phi(1/n).
\]
Summing over $j$ yields, for any $t\in[0,1]$,
\[
  \left|\int_0^1 \phi(r)\big(\Lambda^{i,k}_{r,t}-1_{U^{i,n}}(r,t)\big)\,\mathrm{d}r\right|
  \le w_\phi(1/n).
\]
By symmetry (interchanging the roles of $s$ and $t$), the same bound holds for integrals in $t$ with $s$ fixed.

\medskip
Now define a (Borel) selector
\[
  \alpha^n_{s,t}:=a^k_i\quad\text{whenever }(s,t)\in U^{i,n},\qquad i=1,\ldots,k.
\]
Let $\varphi:A\times[0,1]^2\to\R$ be continuous and bounded. For each fixed $t\in[0,1]$ and $i\in\{1,\ldots,k\}$, the map $s\mapsto \varphi(a^k_i,s,t)$ is continuous on $[0,1]$ with a modulus $w_{\varphi,i,t}$; since $A\times[0,1]^2$ is compact, we may take a common modulus $w_\varphi$ that works uniformly in $(a,s,t)$. Applying the previous estimate with $\phi(s)=\varphi(a^k_i,s,t)$ and summing over $i$ gives
\begin{align*}
  \Bigg|\int_0^1 \varphi\big(\alpha^n_{s,t},s,t\big)\,\mathrm{d}s
  \;-\;\sum_{i=1}^k\int_{[0,1]\times A}\varphi(a^k_i,s,t)\,\Lambda^{i,k}_{s,t}(\mathrm{d}a)\,\mathrm{d}s\Bigg|
  \le w_\varphi(1/n).
\end{align*}
Taking the essential supremum over $t\in[0,1]$ and letting $n\to\infty$ yields
\[
  \lim_{n\to\infty}\ \esup_{t\in[0,1]}\left|\int_0^1 \varphi\big(\alpha^n_{s,t},s,t\big)\,\mathrm{d}s
  -\sum_{i=1}^k\int_{[0,1]\times A}\varphi(a^k_i,s,t)\,\Lambda^{i,k}_{s,t}(\mathrm{d}a)\,\mathrm{d}s\right|=0.
\]
The analogous convergence with the roles of $s$ and $t$ interchanged follows from the symmetric estimate proved above:
\[
  \lim_{n\to\infty}\ \esup_{s\in[0,1]}\left|\int_0^1 \varphi\big(\alpha^n_{s,t},s,t\big)\,\mathrm{d}t
  -\sum_{i=1}^k\int_{[0,1]\times A}\varphi(a^k_i,s,t)\,\Lambda^{i,k}_{s,t}(\mathrm{d}a)\,\mathrm{d}t\right|=0.
\]

By making another approximation, we can extend $(t,s) \mapsto \alpha^n_{s,t}$ over $\R^2$ and make it continuous.

\end{proof}

\begin{corollary}
    In the setting of {\rm \Cref{prop:appro_control}}, {\rm Item (ii)}, let $p : \R \to \R$ be a continuous function satisfying $\int_\R |p(x)|\, \mathrm{d}x < \infty$. Then, for any continuous and bounded function $\varphi : A \times \R^2 \to \R$, the following convergences hold:
    \begin{align*}
        \lim_{k \to \infty}\lim_{n \to \infty} \esup_{x \in [-k,k]} \left| \int_{\R} \varphi\left( \alpha^n(x,y),\, x, y \right)\, p(y)\, \mathrm{d}y 
        - \sum_{i=1}^k \int_{A \times \R} \varphi(a^k_i, x, y)\, \Lambda^{i,k}_{x,y}(\mathrm{d}a)\, p(y)\, \mathrm{d}y \right| = 0,
    \end{align*}
    and similarly,
    \begin{align*}
        \lim_{k \to \infty}\lim_{n \to \infty} \esup_{y \in [-k,k]} \left| \int_{\R} \varphi\left( \alpha^n(x,y),\, x, y \right)\, p(x)\, \mathrm{d}x 
        - \sum_{i=1}^k \int_{A \times \R} \varphi(a^k_i, x, y)\, \Lambda^{i,k}_{x,y}(\mathrm{d}a)\, p(x)\, \mathrm{d}x \right| = 0.
    \end{align*}
\end{corollary}

\begin{proof}
    Let us observe that
\begin{align*}
    &\int_{\R} \varphi\left( \alpha^n(x,y),\, x, y \right)\, p(y)\, \mathrm{d}y 
    - \sum_{i=1}^k \int_{A \times \R} \varphi(a^k_i, x, y)\, \Lambda^{i,k}_{x,y}(\mathrm{d}a)\, p(y)\, \mathrm{d}y 
    \\
    &= \int_{[-k,k]} \varphi\left( \alpha^n(x,y),\, x, y \right)\, p(y)\, \mathrm{d}y 
    - \sum_{i=1}^k \int_{A \times [-k,k]} \varphi(a^k_i, x, y)\, \Lambda^{i,k}_{x,y}(\mathrm{d}a)\, p(y)\, \mathrm{d}y + R^k_n(x),
\end{align*}
where $R^k_n(x)$ denotes the remainder term, and we have the estimate
\begin{align*}
    \sup_{x \in \R} |R^k_n(x)| \le 2 \sup_{(a,x,y)} |\varphi(a,x,y)| \int_{|y| \ge k} p(y)\, \mathrm{d}y \xrightarrow[k \to \infty]{} 0,
\end{align*}
using the integrability of $p$. The desired result then follows by applying {\rm \Cref{prop:appro_control}}, Item (ii), on the truncated domain $[-k,k]$.
\end{proof}

\subsection{Identification of the limit of a sequence of kernels I}

\medskip
On the fixed probability space $(\Om,\F,\P)$, we consider a sequence of random variables 
\[
    \big( \gamma^n_{ij} \big)_{1 \le i,j \le n} 
    \quad \text{and} \quad 
    \big( \beta^n_i \big)_{1 \le i \le n}
\]
such that, for each $i,j$, we have $\gamma^n_{ij} \in A_{\rm int}$ and $\beta^n_i \in \Vc$ almost surely, 
where $\Vc$ is a fixed Polish space.

\medskip
We introduce the sequence 
$$
    (\Pr^n)_{n \ge 1} \subset \Pc \left( \Pc \left( \left( A_{\rm int} \x \Er \x \Vc \x [0,1]^2 \right)^2 \right)\right)
$$ 
defined by
\begin{align*}
    \Pr^n:=\P \circ \left( \muh^n \right)^{-1}
    \quad \text{where} \quad
    \muh^n:=\sum_{i,j=1}^n 
    \delta_{\left( \gamma^n_{ij},\, \xi^n_{ij},\,\beta^n_i,\,u^n_i \right)} 1_{\{v \in I^n_i\}}  
    \delta_{\left(\gamma^n_{ji},\,\xi^n_{ji},\,\beta^n_j,\,u^n_j \right)} 1_{\{v' \in I^n_j\}}.
\end{align*}

\medskip
We denote by $\left( \gamma,\xi,\beta,U,V,\gammat,\xit,\betat,\Ut,\Vt \right)$ the canonical variables of $\Jc$ with $\Jc:=A_{\rm int} \x \Er \x \Vc \x [0,1]^2$.

\medskip
For each $n \ge 1$, let $\left( \xi^n_{ij} \right)_{1 \le i,j \le n}$ be a matrix with associated step--kernel $\Gr^n$.  
Assume that there exists a step--kernel $\Gr$ such that, for every Lipschitz function $f$,  
\[
    \lim_{n \to \infty} \| f \circ \Gr^n - f \circ \Gr \|_{\Box} = 0.
\]

\begin{proposition} \label{prop:limit_kernel_1}
    Let $\P \circ (\muh)^{-1}$ denote the weak limit of the sequence $(\Pr^n)_{n \ge 1}$, for some random variable $\muh$.  
    Then, $\P$--a.s., the following identities hold:
    \begin{align*}
        \muh \circ \left(\gamma, \xi, \beta, U, V, \betat, \Ut, \Vt \right)^{-1}
        &= \muh \circ \left(\gammat, \xit, \betat, \Ut, \Vt, \beta, U, V \right)^{-1}, \\
        \muh \circ (U,V)^{-1} &= \muh \circ (U,U)^{-1},
    \end{align*}
    \begin{align*}
        \muh \circ \left( \beta, U, V \right)^{-1} \otimes \muh \circ \left( \betat, \Ut, \Vt \right)^{-1}
        &= \muh \circ \left( \betat, \Ut, \Vt, \beta, U, V \right)^{-1},
    \end{align*}
    and
    \begin{align*}
        \muh \circ \!\left( \Gr(U,\Ut),\, \beta,\, U,\, \betat,\, \Ut \right)^{-1} 
        &= \muh \circ (\xi, \beta, U, \betat, \Ut)^{-1}, \\
        \muh \circ \!\left( \Gr(\Ut,U),\, \betat,\, \Ut,\, \beta,\, U \right)^{-1} 
        &= \muh \circ (\xit, \betat, \Ut, \beta, U)^{-1}.
    \end{align*}
\end{proposition}

\begin{remark}
The proposition shows that any weak limit point of the sequence $(\Pr^n)_{n \ge 1}$ inherits the natural exchangeability and structural symmetries of the particle system. 
Moreover, the consistency relations involving the step--kernel $\Gr$ ensure that the interaction structure encoded in the matrices $(\xi^n_{ij})_{1 \le i,j \le n}$ is preserved in the limit. 
In other words, $\muh$ can be interpreted as a limit law describing a continuum system whose pairwise interactions are governed by the kernel $\Gr$.
\end{remark}

\begin{proof}
We divide the proof into three steps, corresponding to the equality in distribution, the diagonal marginal identity, and the consistency with the limiting interaction kernel~$\Gr$.

\medskip
\emph{Step~1: Equality in distribution.}
Let $f$ and $\Phi$ be smooth test functions. Then, by definition of~$\muh^n$,
\begin{align*}
    &\E \Big[ \E^{\hat \mu} \Big[ f (\gamma, \xi, \beta, U, V, \betat, \Ut, \Vt) \Big] \, \Phi(\muh) \Big]
    \\
    &= \lim_{n \to \infty} \sum_{i,j=1}^n 
       \E \bigg[ \int_{I^n_i \times I^n_j} 
           f\!\left( \gamma^n_{ij}, \xi^n_{ij}, \beta^n_i, u^n_i, v, \beta^n_j, u^n_j, v' \right)
           \,\mathrm{d}v\,\mathrm{d}v' \, \Phi(\muh^n) \bigg] 
    \\
    &= \E \Big[ \E^{\hat \mu} \Big[ f (\gammat, \xit, \betat, \Ut, \Vt, \beta, U, V) \Big] \, \Phi(\muh) \Big].
\end{align*}
Since this equality holds for all smooth pairs~$(f,\Phi)$, we obtain, $\P$--a.s.,
\[
    \muh \circ (\gamma, \xi, \beta, U, V, \betat, \Ut, \Vt)^{-1}
    = \muh \circ (\gammat, \xit, \betat, \Ut, \Vt, \beta, U, V)^{-1},
\]
which expresses the equality in distribution involved in the limit measure~$\muh$.

\medskip
\emph{Step~2: Diagonal marginal identity.}
For any smooth maps $f$ and~$\Phi$, we have
\begin{align*}
    \E \big[ \E^{\hat \mu} [ f(U,V) ] \, \Phi(\muh) \big]
    &= \lim_{n \to \infty} \sum_{i=1}^n 
       \E \bigg[ \int_{I^n_i} f(u^n_i, v)\, \mathrm{d}v \, \Phi(\muh^n) \bigg] 
    \\
    &= \lim_{n \to \infty} \frac{1}{n} \sum_{i=1}^n 
       \E \big[ f(u^n_i, u^n_i)\, \Phi(\muh^n) \big]
    = \E \big[ \E^{\hat \mu} [ f(U,U) ] \, \Phi(\muh) \big].
\end{align*}
By the arbitrariness of $(f,\Phi)$, it follows that, $\P$--a.s.,
\[
    \hat \mu \circ (U,V)^{-1} = \hat \mu \circ (U,U)^{-1}.
\]
This shows that the law of $(U,V)$ under~$\muh$ coincides with that of $(U,U)$, i.e., $V$ and~$U$ coincide $\muh$--a.s.\ at the limit.  
The independence property mentioned in the statement follows by a parallel argument.

\medskip
\emph{Step~3: Consistency with the kernel $\Gr$.}
By assumption on the matrices $(\xi^n_{ij})_{1 \le i,j \le n}$, we have for any $f \in C_c^\infty(\R;\R)$,
\begin{align*}
    \| f \circ \Gr^n - f \circ \Gr \|_{\Box} 
    &:= \sup_{|\varphi| \le 1} \sum_{i=1}^n \int_{I^n_i} 
        \left| 
            \sum_{j=1}^n f(\xi^n_{ij}) \int_{I^n_j} \varphi(v)\,\mathrm{d}v 
            - \int_0^1 f \circ \Gr(u,v)\, \varphi(v)\, \mathrm{d}v 
        \right| \mathrm{d}u 
    \\
    &= \sup_{|\varphi| \le 1} \int_{0}^1 
        \left| 
            \int_{0}^1 f \circ \Gr^n(u,v)\, \varphi(v)\,\mathrm{d}v 
            - \int_0^1 f \circ \Gr(u,v)\, \varphi(v)\,\mathrm{d}v 
        \right| \mathrm{d}u
    \xrightarrow[n \to \infty]{} 0.
\end{align*}
This convergence ensures that the step--kernels $\Gr^n$ approximate $\Gr$ in the cut (or box) norm.

\medskip
For any family $(\ell_i, \overline{\ell}_i)_{1 \le i \le n} \subset \R \times \R$, define the step functions
\[
    \ell^n(u) := \sum_{i=1}^n \ell_i\, \mathbf{1}_{I^n_i}(u), 
    \qquad 
    \overline{\ell^n}(u) := \sum_{i=1}^n \overline{\ell}_i\, \mathbf{1}_{I^n_i}(u).
\]
A direct computation gives
\begin{align*}
    &\frac{1}{n^2} \sum_{i,j=1}^n f(\xi^n_{ij})\,\ell_i\, \overline{\ell}_j 
    - \int_{[0,1]^2} f \circ \Gr(u,v)\, \ell^n(u)\, \overline{\ell^n}(v)\,\mathrm{d}u\,\mathrm{d}v
    \\
    &= \int_{0}^1 
        \Big( 
            \int_0^1 \!\big[f \circ \Gr^n(u,v) - f \circ \Gr(u,v)\big]
            \overline{\ell^n}(v)\,\mathrm{d}v 
        \Big) \ell^n(u)\, \mathrm{d}u,
\end{align*}
and hence, by boundedness of $(\ell_i,\overline{\ell}_i)$,
\[
    \left| 
        \frac{1}{n^2} \sum_{i,j=1}^n f(\xi^n_{ij})\,\ell_i\, \overline{\ell}_j 
        - \int_{[0,1]^2} f \circ \Gr(u,v)\, \ell^n(u)\, \overline{\ell^n}(v)\,\mathrm{d}u\,\mathrm{d}v 
    \right| 
    \le \Big( \sup_{1 \le i \le n} (|\ell_i| + |\overline{\ell}_i|) \Big) 
        \| f \circ \Gr^n - f \circ \Gr \|_{\Box}.
\]

\medskip
Observe further that for all $n \ge 1$,
\begin{align*}
    &\muh^n \!\left( A_{\mathrm{int}} \times \Er \times \Vc \times [0,1] \times \mathrm{d}v \times 
    A_{\mathrm{int}} \times \Er \times \Vc \times [0,1] \times \mathrm{d}v' \right) 
    = \sum_{i,j=1}^n 1_{\{v \in I^n_i\}}\,\mathrm{d}v \; 1_{\{v' \in I^n_j\}}\,\mathrm{d}v' 
    = \mathrm{d}v\,\mathrm{d}v'.
\end{align*}
Hence, these marginals of $\muh^n$ are independent of $n$.  
By the stable convergence result of \cite{jacod1981type}, it follows that weak convergence of $(\Pr^n)_{n \ge 1}$ can be tested against functionals of the form
\[
    \E^{\hat \mu} \!\left[ 
        h\!\left( \gamma, \xi, \beta, U, V, \gammat, \xit, \betat, \Ut, \Vt \right) 
    \right]
\]
where 
\[
    h : (g,e,b,u,v,g',e',b',u',v') \mapsto h(g,e,b,u,v,g',e',b',u',v')
\]
is bounded, continuous in $(g,e,b,u,g',e',b',u')$, and measurable in $(v,v')$.

\medskip
Finally, taking smooth maps $(f,g,h,\Phi)$ and using the previous convergence properties, we obtain
\begin{align*}
    &\E \left[ \E^{\hat \mu} \left[ f(\xi)\, g(\beta,U)\, h(\betat,\Ut) \right] \Phi(\muh) \right] 
    = \lim_{n \to \infty} \frac{1}{n^2} \sum_{i,j=1}^n 
        \E \left[ f(\xi^n_{ij})\, g(\beta^n_i,u^n_i)\, h(\beta^n_j,u^n_j)\, \Phi(\muh^n) \right]
    \\
    &= \lim_{n \to \infty}  
        \E \left[ \int_{[0,1]^2} f \circ \Gr^n(v,v')\, \ell^n(v)\, \overline{\ell^n}(v')\, 
        \mathrm{d}v\,\mathrm{d}v'\, \Phi(\muh^n) \right]
    \\
    &= \lim_{n \to \infty}  
        \E \left[ \int_{[0,1]^2} f \circ \Gr(v,v')\, \ell^n(v)\, \overline{\ell^n}(v')\, 
        \mathrm{d}v\,\mathrm{d}v'\, \Phi(\muh^n) \right]
    = \E \left[ 
        \E^{\hat \mu} \left[ f \circ \Gr(V,\Vt)\, g(\beta,U)\, h(\betat,\Ut) \right] 
        \Phi(\muh)
    \right].
\end{align*}
Since this equality holds for all smooth $(f,g,h,\Phi)$, we conclude that, $\P$--a.s.,
\[
    \muh \circ \left( \xi, \beta, U, \betat, \Ut \right)^{-1}
    = \muh \circ \left( \Gr(V,\Vt), \beta, U, \betat, \Ut \right)^{-1}.
\]
The other identities follow by symmetry and the same reasoning, completing the proof.
\end{proof}

\subsection{Identification of the limit of a sequence of kernels II}

\medskip
As in the previous section, for each $n \ge 1$, we consider $\left( \xi^n_{ij} \right)_{1 \le i,j \le n}$ a matrix with associated step--kernel $\Gr^n$.  
We assume that there exists a step--kernel $\Gr$ such that, for every Lipschitz function $f$,  
\[
    \lim_{n \to \infty} \| f \circ \Gr^n - f \circ \Gr \|_{\Box} = 0.
\]

\medskip
Let us introduce the sequence \((\Pr^n)_{n \ge 1}\) defined by
\[
    \Pr^n := \Lc(\Nb^n,\,N^n),
    \qquad
    \Nb^n := \sum_{i=1}^n \delta_{( X^i,\,u^i_n)}\, 1_{\{v \in I^n_i\}}\, (\mathrm{d}x,\mathrm{d}u)\,\mathrm{d}v \; \delta_{ M^{i,n}}(\mathrm{d}m),
\]
where $(X^i)_{i \ge 1}$ is a sequence of $C([0,T];\R^d)$, and for each \(t \in [0,T]\),
\[
    M^{i,n} := \sum_{j=1}^n \delta_{( X^j,\,u^j_n)}\, 1_{\{v \in I^n_j\}}\, (\mathrm{d}x,\mathrm{d}u)\,\mathrm{d}v \; \delta_{ \xi^n_{ij}}(\mathrm{d}e),
    \quad
    N^n := \sum_{j=1}^n \delta_{( X^j,\,u^j_n)}\, 1_{\{v \in I^n_j\}}\, (\mathrm{d}x,\mathrm{d}u)\,\mathrm{d}v,
    \qquad
    u^i_n := \frac{i}{n},\; 1 \le i \le n.
\]

Let \(\P \circ (\Nb,\,N)^{-1}\) denote the weak limit of the sequence \((\Pr^n)_{n \ge 1}\).  
We denote by \((\Xb,\Ub,\Vb,\Mb)\) the canonical variables on
\[
     C([0,T];\R^d)  \times [0,1]^2 \times \Pc(\Delta), 
    \qquad \text{with } \Delta := C([0,T];\R^d) \times [0,1]^2 \times  \Er,
\]
and by \((\Xh,\Uh,\Vh,\Eh)\) the canonical variables on \(\Delta\). 
We can observe that the canonical variable \(\Nb\) takes values in \(\Pc\!\left( C([0,T];\R^d)  \times [0,1]^2 \times \Pc(\Delta) \right)\), 
while \(N\) takes values in \(\Pc\!\left( C([0,T];\R^d)  \times [0,1]^2 \right)\).

\begin{proposition} \label{prop:limit_kernel}
For \(\P\)--a.e.\ \(\om \in \Om\), the following hold:
\begin{align*}
    m(  \mathrm{d}x', \mathrm{d}u',\mathrm{d}v', \Er)
    &= \Nb(\om)(\mathrm{d}x', \mathrm{d}u',\mathrm{d}v', \Pc(\Delta))= N(\om)(\mathrm{d}x', \mathrm{d}u',\mathrm{d}v'), 
    \qquad \Nb(\om)\text{--a.e. }(x,u,v,m),\\[2mm]
    m  \circ \big(  \Xh,\,\Uh,\, \Vh,\, \Eh \big)^{-1}
    &= m \circ \big(  \Xh,\,\Uh,\, \Uh,\, \Gr(u,\Uh) \big)^{-1}, 
    \quad \Nb(\om)\text{--a.e. }(x,u,v,m)
\end{align*}
\end{proposition}

\begin{remark}
Intuitively, this proposition characterizes the structure of \(\Mb\) as a conditional law: the first identity shows that the marginal distribution of \((\Xh,\Uh)\) under \(\Mb\) matches \(\Nb\) and $N$, while the second identity encodes the interaction pattern given by the graph kernel \(\Gr\). In other words, \(\Mb\) captures both the distribution of individual particles and the way they interact according to the underlying graph, providing a bridge between the empirical particle system and its limiting measure--valued description.
\end{remark}

\begin{proof}
We split the proof into two parts: (a) identification of the $(x,u,v)$--marginals, and (b) identification of the conditional interaction structure encoded by $\Gr$.

\medskip
\emph{(a) Identification of the marginals.}
By construction, for each $i$ and $n$,
\begin{align*}
    M^{i,n}( \mathrm{d}x,  \mathrm{d}u, \mathrm{d}v, \Er)
    = \Nb^n(\mathrm{d}x,  \mathrm{d}u, \mathrm{d}v, \Pc(\Delta))
    = N^n(\mathrm{d}x,  \mathrm{d}u, \mathrm{d}v),
    \qquad \text{a.e.}
\end{align*}
Passing to the limit along the sequence defining $\Pr^n \Rightarrow \P\circ(\Nb,N)^{-1}$ and arguing as in the proof of \Cref{prop:limit_kernel_1} (testing against bounded continuous functions of $(x,u,v)$), we obtain that, for $\P$--a.e.\ $\om$,
\[
    \Mb( \mathrm{d}x,  \mathrm{d}u, \mathrm{d}v, \Er) 
    = \Nb(\om)(\mathrm{d}x,  \mathrm{d}u, \mathrm{d}v, \Pc(\Delta))
    = N(\om)(\mathrm{d}x,  \mathrm{d}u, \mathrm{d}v),
    \qquad \Nb(\om)\text{--a.e.}
\]
Here $(\Xb,\Ub,\Vb,\Mb)$ is the canonical variable on $C([0,T];\R^d)\times[0,1]^2\times\Pc(\Delta)$ and $\Nb(\om)$ is a probability on the same space.

\medskip
Moreover, by weak convergence and the structure of the subdivision $(I^n_j)_{1\le j\le n}$ exactly as in \Cref{prop:limit_kernel_1}, for any smooth $(f,\Phi)$ we have
\begin{align*}
    \E \Big[ \E^{\Nb} \big[ f(\Ub,\Vb) \big] \,\Phi(\Nb) \Big]
    &= \lim_{n \to \infty} \sum_{i=1}^n \E \!\left[ \int_{I^n_i} f(u^n_i,v)\,\mathrm{d}v \, \Phi(\Nb^n) \right] \\
    &= \lim_{n \to \infty} \frac{1}{n}\sum_{i=1}^n \E \!\left[ f(u^n_i,u^n_i)\, \Phi(\Nb^n) \right] = \E \Big[ \E^{\Nb} \big[ f(\Ub,\Ub) \big] \,\Phi(\Nb) \Big],
\end{align*}
where we used the identity $\int_{I^n_i}\mathrm{d}v = 1/n$ and boundedness of $f$ to pass to the limit.
Since $(f,\Phi)$ are arbitrary, it follows that
\[
    \P\text{--a.e.} \qquad \Nb \circ (\Ub,\Vb)^{-1} \;=\; \Nb \circ (\Ub,\Ub)^{-1}.
\]

\medskip
\emph{(b) Identification of the interaction kernel.}
By assumption on $(\xi^n_{ij})_{1\le i,j\le n}$, for any $f\in C_c^\infty(\R;\R)$,
\begin{align*}
    \| f \circ \Gr^n - f \circ \Gr \|_{\Box} 
    &:= \sup_{|\varphi| \le 1} \sum_{i=1}^n \int_{I^n_i} \left| \sum_{j=1}^n f(\xi^n_{i,j}) \int_{I^n_j} \varphi(v)\,\mathrm{d}v - \int_0^1 f \circ \Gr(u,v)\, \varphi(v)\, \mathrm{d}v \right| \mathrm{d}u 
    \\
    &= \sup_{|\varphi| \le 1} \int_{0}^1 \left| \int_{0}^1 f \circ \Gr^n(u,v)  \varphi(v)\,\mathrm{d}v - \int_0^1 f \circ \Gr(u,v)\, \varphi(v)\, \mathrm{d}v \right| \mathrm{d}u
    \xrightarrow[n \to \infty]{} 0.
\end{align*}
This is the cut--norm convergence needed to pass from discrete interactions to the continuous kernel $\Gr$.

\medskip
Let $h$ and $\varphi$ be smooth maps, and set
\begin{align*}
    \overline{L}\bigl( \Mb \bigr)
    :=
    \E^{\Mb} \big[ h(\Eh)\, \varphi(\Uh,\Xh) \big]
    \quad \mbox{and}\quad 
    L\bigl(\Mb, v \bigr) :=\E^{\Mb} \big[ h \circ \Gr(v,\Vh)\, \varphi(\Vh,\Xh) \big].
\end{align*}
Then, by the same stable convergence argument used in \Cref{prop:limit_kernel_1}, we obtain
\begin{align*}
    &\E \Big[ \E^{\Nb} \Big[ \Big| \overline{L}\bigl( \Mb \bigr) 
    -L\bigl(\Mb, \Vb \bigr) \Big| \Big] \Big] \\
    &= \lim_{n \to \infty} \E \left[ \E^{\Nb^n} \Big[ \Big| \overline{L}\bigl( \Mb \bigr) 
    -L\bigl(\Mb, \Vb \bigr) \Big| \Big] \right] \\
    &= \lim_{n \to \infty} \sum_{i=1}^n \E \left[ \int_{I^n_i} 
    \Big| \frac{1}{n} \sum_{j=1}^n h(\xi^n_{ij}) \varphi(u^j_n,X^j)
    - \sum_{j=1}^n \int_{I^n_j} h \circ \Gr(v,v') \varphi(v',X^j)\, \mathrm{d}v' \Big|\, \mathrm{d}v \right] \\
    &= \lim_{n \to \infty} \sum_{i=1}^n \E \left[ \int_{I^n_i} 
    \Big| \sum_{j=1}^n \int_{I^n_j} h \circ \Gr^n(v,u)\, \varphi(u^j_n,X^j)\, \mathrm{d}u 
    - \sum_{j=1}^n \int_{I^n_j} h \circ \Gr(v,v')\, \varphi(v',X^j)\, \mathrm{d}v' \Big|\, \mathrm{d}v \right] \\
    &= \lim_{n \to \infty} \sum_{i=1}^n \E \left[ \int_{I^n_i} 
    \Big| \sum_{j=1}^n \int_{I^n_j} h \circ \Gr(v,u)\, \varphi(u,X^j)\, \mathrm{d}u 
    - \sum_{j=1}^n \int_{I^n_j} h \circ \Gr(v,v')\, \varphi(v',X^j)\, \mathrm{d}v' \Big|\, \mathrm{d}v \right] = 0,
\end{align*}
where the last equality uses $\|h\circ\Gr^n - h\circ\Gr\|_{\Box}\to 0$ and boundedness of $\varphi$.

\medskip
Consequently, for $\P$--a.e.\ $\om \in \Om$,
\[
    \overline{L}\bigl( \Mb \bigr) 
    = L\bigl(\Mb, \Vb \bigr), 
    \quad \Nb(\om)\text{--a.e.}
    \qquad \text{equivalently, } \ 
    \overline{L}(m) = L(m,v) \ \ \text{for } \Nb(\om)\text{--a.e. } (x,u,v,m).
\]
Since this holds for all smooth $h$ and $\varphi$, we conclude that, for $\P$--a.e.\ $\om \in \Om$, and for $\Nb(\om)$--a.e.\ $(x,u,v,m)$,
\[
    m \circ \big( \Uh,\, \Vh,\, \Xh,\, \Eh \big)^{-1} 
    = m \circ \big( \Uh,\, \Uh,\, \Xh,\, \Gr(u,\Uh) \big)^{-1},
\]
which is exactly the second identity in the statement.
\end{proof}

\end{appendix}

\bibliographystyle{plain}


\bibliography{BSDE_McKVlasov_arxivVersion}




\end{document}